\documentclass{amsart}
\pdfoutput=1
\usepackage[textsize=tiny,textwidth=4cm]{todonotes}
\usepackage{lipsum}
\usepackage{amsmath}
\usepackage{amssymb}
\usepackage{amsfonts}
\usepackage{mathtools}
\usepackage{array}
\usepackage{multirow}
\usepackage{capt-of}
\usepackage[skip=10pt,font=footnotesize,width=0.9\textwidth]{caption}
\usepackage[shortcuts]{extdash}
\usepackage{graphicx}
\usepackage{epstopdf}
\usepackage{algorithmic}
\usepackage{listings}
\usepackage[
  algo2e,
  algosection,
  linesnumbered,
  ruled
]{algorithm2e}
\usepackage[textsize=tiny]{todonotes}
\usepackage{bbm}
\usepackage{colortbl}
\usepackage{wrapfig}
\usepackage{placeins}
\usepackage[bottom=3.5cm,left=3.6cm,right=3.8cm]{geometry}
\usepackage[capitalize,nameinlink]{cleveref}
\crefname{section}{section}{sections}
\crefname{subsection}{subsection}{subsections}
\Crefname{section}{Section}{Sections}
\Crefname{subsection}{Subsection}{Subsections}
\crefformat{equation}{\textup{#2(#1)#3}}
\crefrangeformat{equation}{\textup{#3(#1)#4--#5(#2)#6}}
\crefmultiformat{equation}{\textup{#2(#1)#3}}{ and \textup{#2(#1)#3}}
{, \textup{#2(#1)#3}}{, and \textup{#2(#1)#3}}
\crefrangemultiformat{equation}{\textup{#3(#1)#4--#5(#2)#6}}%
{ and \textup{#3(#1)#4--#5(#2)#6}}{, \textup{#3(#1)#4--#5(#2)#6}}{, and \textup{#3(#1)#4--#5(#2)#6}}

\Crefformat{equation}{#2Equation~\textup{(#1)}#3}
\Crefrangeformat{equation}{Equations~\textup{#3(#1)#4--#5(#2)#6}}
\Crefmultiformat{equation}{Equations~\textup{#2(#1)#3}}{ and \textup{#2(#1)#3}}
{, \textup{#2(#1)#3}}{, and \textup{#2(#1)#3}}
\Crefrangemultiformat{equation}{Equations~\textup{#3(#1)#4--#5(#2)#6}}%
{ and \textup{#3(#1)#4--#5(#2)#6}}{, \textup{#3(#1)#4--#5(#2)#6}}{, and \textup{#3(#1)#4--#5(#2)#6}}
\usepackage{xr}

\makeatletter
\DeclareRobustCommand*{\bbl@ap}[1]{\textormath{\textsuperscript{#1}}{^{\mathrm{#1}}}}%
\DeclareRobustCommand*{\bbl@ped}[1]{\textormath{$_{\mbox{\fontsize\sf@size\z@ \selectfont#1}}$}{_\mathrm{#1}}}%
\let\ap\bbl@ap
\let\ped\bbl@ped
\makeatother

\newtheorem{theorem}{Theorem}[section]

\newtheorem{proposition}[theorem]{Proposition}

\newtheorem{remark}[theorem]{Remark}

\usepackage{dsfont}
\usepackage{dutchcal}
\usepackage{bm}
\usepackage[normalem]{ulem}

\usepackage{lipsum}
\usepackage{amsfonts}
\usepackage{epstopdf,epstopdf}
\usepackage{algorithmic}
\Crefname{ALC@unique}{Line}{Lines} % <- Preamble
\usepackage[labelfont=rm]{subfig} % <- Preamble
\ifpdf
  \DeclareGraphicsExtensions{.eps,.pdf,.png,.jpg}
\else
  \DeclareGraphicsExtensions{.eps}
\fi

\newcommand{\utdc}{\mathcal{u}} 
 
\newcommand{\OoI}{\mathcal{z}}
\newcommand{\ldu}{u}

\newcommand{\ldw}{w}

\newcommand{\imagI}{\mathcal{i}}

\newcommand{\paramu}{\bm{m}}
\newcommand{\U}{\mathcal{U}}
\newcommand{\Uh}{\mathcal{U}_h}
\newcommand{\Urb}{\mathcal{U}_k}

\newcommand{\Nh}{N_h}
\newcommand{\Nk}{N_k}
\newcommand{\Nz}{N_z} % No. of freq. points for Weeks method
\newcommand{\Nr}{N_{r}} % No of receivers
\newcommand{\Pspace}{\mathcal{P}}

\def\Week{\mathcal{W}}
\def\CC{\mathbb{C}}
\def\CK{\beta_b}

\def\wr{\mathcal{w}_R}
\def\wi{\mathcal{w}_I}
\def\reals{s_R }
\def\imags{s_I }

\usepackage{amsopn}

\ifpdf
  \DeclareGraphicsExtensions{.eps,.pdf,.png,.jpg}
\else
  \DeclareGraphicsExtensions{.eps}
\fi

\title[Model order reduction for seismic applications]{Model order reduction for seismic applications}

\author{Rhys Hawkins}
\address{School of Computing and Research School of Earth Sciences, Australian National University, ACT 2601, Canberra, Australia. Formerly at Department of Earth Sciences, Utrecht University, The Netherlands {rhys.hawkins@anu.edu.au}.}

\author{Muhammad Hamza Khalid}
\address{Department of Applied Mathematics,
	University of Twente, P.O. Box 217, 7500 AE Enschede, The Netherlands
  {mhamzakhalid15@gmail.com}.}

\author{Matthias Schlottbom}
\address{Department of Applied Mathematics,
	University of Twente, P.O. Box 217, 7500 AE Enschede, The Netherlands
 {m.schlottbom@utwente.nl}.}

\author{Kathrin Smetana}
\address{Department of Mathematical Sciences, Stevens Institute of Technology, 1 Castle Point Terrace, Hoboken, NJ 07030, United States of America. {ksmetana@stevens.edu}
}
\date{\today}

\thanks{This work is part of the research programme DeepNL, financed by the Dutch Research Council (NWO) under project number DeepNL.2018.033. Rhys Hawkins was partially supported by the Australian Government through the Australian Research Council's Discovery Projects funding scheme (project DP200100053). }



\begin{document}
	
\begin{abstract}
In this paper we propose a model order reduction approach to speed up the computation of seismograms, i.e. the solution of the seismic wave equation evaluated at a receiver location, for different model parameters. This is highly relevant for seismic applications such as full waveform inversion, seismic tomography, or monitoring tools of seismicity that are computationally challenging as the discretized (forward) model has often a huge number of unknowns and needs to be solved many times for different model parameters.  Our approach achieves a reduction of the unknowns by a factor of approximately 1000 for various numerical experiments for a 2D subsurface model of Groningen, the Netherlands (a region known for its seismic activity) even if the (known) wave speeds of the subsurface are relatively varied.  Moreover, when using multiple cores to construct the reduced model we can approximate the (time domain) seismogram in a lower wall clock time than an implicit Newmark-beta method. To realize this reduction, we exploit the fact that seismograms are low-pass filtered for the observed seismic events. We thus consider the Laplace-transformed problem in the frequency domain and implicitly restrict ourselves to the frequency range of interest by adjusting the parameters of a source function which is a popular model in computational seismology for the temporal response of an earthquake. Therefore, we can avoid the high frequencies that would require many reduced basis functions to reach the desired accuracy and generally make the reduced order approximation of wave problems challenging. Instead, we can prove for our ansatz that for a fixed subsurface model the reduced order approximation converges exponentially fast in the frequency range of interest in the Laplace domain. We build the reduced model from solutions of the Laplace-transformed problem via a (Proper Orthogonal Decomposition-)Greedy algorithm targeting the construction of the reduced model to the time domain seismograms; the latter is achieved by using an a posteriori error estimator that does not require computing any time domain counterparts. Finally, we show that we obtain a stable reduced model thus overcoming the challenge that standard model reduction approaches do not necessarily yield a stable reduced model for wave problems.
\end{abstract}
	
\maketitle

\section{Introduction}\label{sect:intro}
Seismic applications, such as seismic tomography \cite{Rawlinson2014,Tromp2005SeismicKernels}, full waveform inversion (FWI) \cite{bozdaug2011misfit,fichtner2010full}, generating synthetic data for early warnings \cite{panza2012seismic} or neural-network based monitoring tools of seismicity \cite{kaufl2016solving}, require solving the parametric seismic wave equation multiple times for different parameter configurations; see \cref{sec:seismo_application} for details. 
However, for realistically sized problems in seismology, the wavefield approximation using established discretization methods \cite{fichtner2009simulation,komatitsch2002spectrala,Tromp2019SeismicScales} is computationally expensive, making full waveform  modeling for many parameters and, thus, the before-mentioned applications computationally infeasible.

In this paper, we use projection-based model order reduction (MOR) via reduced basis (RB) methods (see e.g., \cite{Rozza2008ReducedEquations}) to address the above-mentioned computational challenges. 
The general idea of RB methods for time-dependent problems is to construct a low-dimensional RB space from the numerically approximated solution trajectories of the high-dimensional problem for certain well-selected parameters (snapshots) in a given parameter set. 
This can be achieved in a (quasi-)optimal manner  \cite{MR2821591,buffa2012priori,devore2013greedy,HaasdonkPODGreedyConvergence} by RB methods such as the Proper Orthogonal Decomposition (POD) \cite{MR1204279,MR1868765,MR0910462}, the Greedy algorithm \cite{grepl2005,veroy2003posteriori}, or the POD-Greedy algorithm \cite{haasdonk2008reduced}. 
The POD method relies on an SVD-based compression of the set of solutions (evaluated in the time grid points) for all parameters in a dense so-called training set. 
The POD-Greedy algorithm iteratively enriches the RB space by POD basis functions that approximate (in time) the solution trajectory for a selected parameter from the given parameter set; the latter maximizes (an a posteriori error estimator of) the RB approximation error. 
The Greedy algorithm iteratively enriches the RB space by using the solution trajectory at which the a posteriori error estimator is maximized. 
By projecting the full order model (FOM) onto the low-dimensional RB space, we obtain a small reduced order model (ROM).

However, the construction of stable and small ROMs for the seismic wave equation can be challenging. 
For non-smooth solution trajectories of the wave equation, the RB approximation error may decay as slowly as $\mathcal{O}(n^{-1/2})$ \cite{greif2019decay}, indicating large-sized ROMs. 
In addition, for stable FOMs the above-mentioned RB methods may not guarantee stable ROMs for wave problems \cite{Peng2016SymplecticSystems}. 
Some attempts focusing on the faster RB approximation error decay use (non-linear) transformations (cf., \cite{welper2017interpolation,reiss2018shifted,taddei2020registration}) or neural networks (cf., \cite{lee2020model,MR4549101}). Moreover, the stability of the ROMs can be achieved by using (space-time) ultraweak formulations (cf., \cite{brunken2019parametrized,Dahmen2012,dahmen2014double,henning2022ultraweak}), discontinuous Petrov-Galerkin methods (cf., \cite{bui2013constructively,demkowicz2011class,zitelli2011class}), or symplectic MOR (cf., \cite{Peng2016SymplecticSystems,pagliantini2021dynamical}).  

In this paper, we introduce an RB approach tailored to the parametric seismic wave problem to obtain a fast and accurate approximation of the time domain seismograms (wavefield at receiver locations).  
Our approach yields provably stable and small-sized ROMs for a fixed subsurface model of Groningen, the Netherlands (a region known for its seismic activity).  
We demonstrate this in numerical results, where, for a fixed subsurface model, we observe a reduction from 233,318 unknowns to an average of less than 176 without compromising the $L^2$-norm accuracy; the latter remains about $10^{-3}$ for the time domain seismograms. 
Additionally, we attain a similar accuracy using 300 RB functions when there are relative variations in the shear-wave velocities of the subsurface. 
By using multiple cores in the construction of an RB space, we can approximate time domain seismograms faster than the implicit Newmark-beta method \cite{newmark1959method}. 

We achieve exponentially converging and stable ROMs by restricting the frequency domain solution computations to the low-frequency range; the latter is motivated by the standard practices in seismology where the recorded seismograms are post-processed by using low-pass filters to avoid various noise sources (e.g., anthropogenic noise) and sensitivity of the seismometer recording the data \cite{Pratt2012SeismicModel,du2000noise,sriyanto2023performance}\footnote{This is important for the applications we consider, as we are interested in information in the low-frequency regime, which is also important in FWI to avoid ``cycle skipping'' \cite{treister2017full,van20143d}. A cycle skipping refers to being stuck in local minima due to the oscillatory nature of seismograms used in the misfit function (a lesser problem for low frequencies) that may yield inaccurate earth models with gradient-based optimization methods.}. 
Here, the low-frequency restriction allows us to focus on less oscillatory solutions, which are easier to represent using few RB functions. 
Conversely, solutions at higher frequencies tend to be sensitive to small-scale changes due to their oscillatory nature and require more RB functions to represent the wavefield accurately. 
This matches our theoretical result in \cref{prop:kolmogorov} on the exponential decay of the RB approximation errors where increasing the maximum frequency ultimately results in a loss of the exponential convergence. In addition, we show in the frequency domain that the restriction to the low frequencies is important to ensure a uniformly well-posed seismic wave problem, where we lose the stability of the finite element and the RB approximation in the high frequency regime.

We implicitly remove the high-frequency components of the wavefield by adjusting the parameters of the Ricker wavelet which acts as the source function and is a popular model in computational seismology to simulate the temporal response of an earthquake.
Here, we emphasize that for the considered numerical test case, the energies of the solutions for different frequencies behave similar to the source function and we can choose a very small cut-off frequency to guarantee an acceptable numerical error (see e.g., \cref{fig:Ricker_wavelet_TD_FD_sol}). We neglect high frequencies with energy values smaller than the acceptable error. 

To construct the RB space, we present frequency domain variants of the POD-Greedy algorithm and the Greedy algorithm, which are driven by an a posteriori error estimator of the RB approximation error in the time domain seismograms. 
This a posteriori error estimator in \cref{prop:A_post_est} is one of the major contributions of this paper. The estimator has the significant advantage that we avoid any time-dependent solution computations during the RB space construction, which can otherwise be computationally demanding due to the temporal nature of other error estimators \cite{glas2020reduced,amsallem2014error}. 
The a posteriori error estimator requires computing stability constants, which is generally done by solving generalized eigenvalue problems or by using the successive constraint method \cite{MR2367928}. 
However, we can estimate the stability constants at low cost due to an explicit expression of its rigorous lower bound, which is very close to the discrete stability constant as we show in the numerical results. 
The a posteriori error estimators rely on numerical Laplace inversion via Weeks’ method, which is also used to compute the time domain approximations. 
Here, as one (minor) novel contribution, we rapidly determine the optimal frequency contour using RB approximations within the respective optimization problem from \cite{Weideman1999AlgorithmsTransform}, which would otherwise be computationally infeasible.

The Laplace domain approach for parabolic problems by N. Guglielmi and M. Mannuci \cite{guglielmi2022contour} shares some similarities with our approach but is not suitable for hyperbolic problems as discussed in  \cite[p. 6 \& 27]{guglielmi2022contour}. 
D.B.P. Huynh et al. \cite{huynh2011laplace} constructed the RB spaces for wave problems via a Greedy algorithm driven by an a posteriori error estimator, where a Butterworth filter was used to attenuate higher frequencies in the output. 
However, the evaluation of their a posteriori error estimator requires solving a contour integral over an infinite imaginary axis, which can be potentially problematic for higher frequencies and is often not necessary.
C. Bigoni and J.S. Hesthaven \cite{Bigoni2020SimulationbasedMonitoring} also used numerical Laplace inversion via Weeks' method and a (symplectic-)POD method from \cite{Peng2016SymplecticSystems} for a structural health monitoring problem.  
The optimal frequency contour was determined by minimizing the error between the Weeks' method and the Newmark-beta method which is computationally demanding. 
For similar problems, the Laplace domain Greedy algorithm has also been used, where the RB space is enriched via the true error in the RB approximations \cite{bhouri2021two} or by solving a least-square minimization problem \cite{baydoun2020greedy} that can be expensive for large training sets. 
In parallel to our work, F. Henr\'iquez and J.S. Hesthaven \cite{henriquez2024fast} proposed a Laplace domain MOR approach to solve one time-instance of a finite element semi-discretized non-parametric wave equation for a fixed, constant wave speed. 
To obtain exponential convergence in \cite[Theorem 4.8]{henriquez2024fast}, the dimension of the ROM needs to be sufficiently large depending on the finite element resolution and wave speed. In their work, the ROM is obtained by projecting the time domain problem onto a reduced space constructed via the POD method relying on the frequency domain snapshots. The ROM solution can still be computationally expensive due to the use of time-stepping methods. 

For closely related seismic problems, the construction of ROMs has also been done via, for instance, eigenfunction expansions \cite{hawkins2023model}, Krylov subspace-based methods (e.g., \cite{druskin2018compressing,druskin2014extended}), a combination of multi-scale methods with balanced truncation \cite{druskin2017multiscale}, or by assuming certain symmetry properties of collocated sources and receivers (e.g., \cite{mamonov2015nonlinear,borcea2021reduced}). While the computation of eigenfunctions can be expensive,  balanced truncation-based methods may not always guarantee stable ROMs for wave problems (cf., \cite[sec. 3.2]{benner2021frequency}) along with the requirement to solve matrix differential equations, which can be computationally expensive.

The rest of the paper is organized as follows: In \cref{sec:Motivation} we provide detailed information on the considered applications and explain the main ideas of the approach and key contributions. 
In \cref{sec:Approximation and NLI} we show a well-posedness result associated with the Galerkin approximations and present Weeks' method. 
After introducing the Galerkin RB approximations, we derive the a posteriori error estimator in \cref{sec:RB approximation and error estimation}.
In \cref{sec:MOR for NLI,sec:parametric reduction} we propose methods for the construction of the RB space over the frequency parameter and in the multi-parameter setting,  respectively. Finally, we show numerical experiments for an  earth model in \cref{sec:Numerical_results} and conclude in \cref{sec:conclusions}.

%%%%%%%%%%%%%%%%%%%%%%
\section{Motivation and key new contributions}\label{sec:Motivation}
\subsection{Seismological applications}\label{sec:seismo_application}
Seismic waveform modeling plays an important role in a broad spectrum of applications that include estimating the subsurface structure to better understand the Earth, studying source mechanisms of natural and induced seismicity (e.g. understanding earthquakes due to gas extraction), and developing an understanding of plausible ground shake due to nearby earthquakes in scenario testing for estimating seismic risk. The motivation for our work comes primarily from the computational challenges faced in the estimation of subsurface structures but also the development of earthquake monitoring tools \cite{panza2012seismic,wang2020regularized}
using the numerically modeled seismic wave propagation. For instance, subsurface structure estimation in a FWI approach or in seismic tomography typically uses an iterative gradient based optimization approach where the goal is to minimize a misfit function that penalizes differences between observed signals at seismic stations (seismograms) and simulated signals (cf., \cite{fichtner2009simulation,komatitsch2002spectrala,Rawlinson2014,Tromp2019SeismicScales,Tromp2005SeismicKernels}). The primary computational effort in this process occurs in each iteration, where a forward problem (seismic wave equation) is provided with a subsurface model for which a potentially large number of earthquakes or signals are simulated. A current limitation in FWI and seismic tomography approaches is the computational cost of forward modeling and this generally results in reducing the number of events that are simulated (removing data) or only performing a relatively small number of gradient descent iterations, possibly resulting in poorly converged results.

In computational seismology, a fast and yet accurate forward problem solver is critical for FWI for subsurface structure, estimation of earthquake source parameters, and, more generally, in many applications where large numbers of simulations with varying input parameters are required; examples are evaluating seismic hazards, generating large datasets for monitoring tool developments  or tomography optimization problems where the earth model is continuously being updated based on the difference between simulated and observed waveforms. Given the high significance of forward modeling in seismic applications, our objective is to reduce the computational costs by adopting an approach based on an inexact approximation of seismograms. The approach utilizes model order reduction to solve the forward problem with a desired accuracy level while accommodating parameter changes. In particular, we focus on instances where the material properties are updated for a fixed layer-structured subsurface model due to optimization iteration (see e.g., \cite{fichtner2009simulation,hawkins2023model,Tromp2019SeismicScales}) or changes in the frequency content.

\subsection{Model}\label{sec:MODEL}
We consider a bounded domain $\Omega\subset \mathds{R}^d$, $d=\{2,3\}$ with Lipschitz boundary $\partial \Omega$, which is partitioned into a Neumann (stress free) boundary $\Gamma_{N}$ and a Dirichlet (fixed) boundary $\Gamma_{D}$ such that $\overline{\partial \Omega} = \bar{\Gamma}_N \cup \bar{\Gamma}_D$, $\Gamma_N \cap \Gamma_D = \emptyset$, with $|\Gamma_{D}| > 0$. Seismic waveforms are modeled through elastic wave equations, where the material properties are defined using Lam\'e parameters $\lambda,\mu\in L^\infty(\Omega)$ that satisfy
\begin{equation}\label{eq:Lame_bounds}
0< \mu_0\leq\mu(x)\leq\mu_1<\infty,\quad 0<\lambda_0\leq \lambda(x)\leq \lambda_1<\infty \quad\text{for a.e. } x\in\Omega.
\end{equation}
For the sake of simplicity, we assume that the Lam\'e parameter values are constant in known layers. Otherwise, the MOR approach proposed in this paper needs to be combined with methods that address the issue of high-dimensional parameter spaces. 
Let $\Pspace := \{ (\lambda,\mu)\;:\; \lambda,\mu \text{ satisfy } \cref{eq:Lame_bounds} \}$ and $\paramu\in\Pspace$, then the elastic wave equation reads: Find $\mathcal{u}(\paramu):\Omega\times (0,T)\rightarrow \mathds{R}^d$ for $T<\infty$ such that
\begin{equation} \label{eq: Time Domain problem}
\begin{split}
 \rho \,\ddot{\mathcal{u}}(\paramu) - \nabla \cdot \sigma(\mathcal{u}(\paramu))&=  \mathcal{f},\quad \quad \quad \text{ in } \Omega \times (0,T),\\ 
 \sigma(\mathcal{u}(\paramu))\cdot \bm{n} &= 0, \;\;\;\quad\quad\text{ on } \Gamma_{N}\times (0,T),\\
 \quad \,\mathcal{u}(\paramu)&= 0,  \;\;\;\quad\quad\text{ on } \Gamma_{D}\times (0,T),\\
 \mathcal{u}(\paramu) &= 0,\;\qquad\quad \text{in } \Omega \times \{0\},\\
 \dot{\mathcal{u}}(\paramu) &= 0,\;\qquad\quad \text{in } \Omega \times \{0\}.
\end{split}
\end{equation}
Here $\rho\in L^\infty(\Omega)$ is the medium density, $\sigma(\mathcal{u}(\paramu))$ is the Cauchy stress tensor, and $\bm{n}$ is the outward unit normal. For a linear elastic material the Cauchy stress tensor satisfies $\sigma(\mathcal{u}(\paramu))= C(\paramu) : \varepsilon(\mathcal{u}(\paramu))$, where $C(\paramu)$ is the fourth-order stiffness tensor, $\varepsilon(\mathcal{u}(\paramu)) = \tfrac{1}{2} (\nabla \mathcal{u}(\paramu) + (\nabla \mathcal{u}(\paramu))^{T})$ is the infinitesimal strain tensor,
and the colon operator $:$ is defined as $(C(\paramu) : \varepsilon(\mathcal{u}(\paramu)))_{ij} = \sum_{k,l=1}^{d}C_{ijkl}(\paramu)\varepsilon_{kl}(\mathcal{u}(\paramu))$. We assume that the medium is piecewise homogeneous, isotropic and under plain strain for $d=2$ -- an assumption admissible for seismic waves with no deformations in the medium and elastic behavior independent of direction \cite{tromp2008spectral,Tromp2019SeismicScales}. Hence, the stiffness tensor is
\begin{equation}
C_{ijkl}(\paramu) = \lambda(x)\delta_{ij}\delta_{kl} + \mu(x)(\delta_{ik}\delta_{jl} + \delta_{il}\delta_{jk}), \quad 1\leq i,j,k,l \leq d, 
\end{equation}
where $\delta_{ij}$ denotes the Kronecker delta. 
Furthermore, $\mathcal{f}(x,t):=p(x;x_0)\mathcal{q}(t)$ is the external source that models the temporal seismic response on a point source location. Here, we use $\mathcal{q}(t)$ as the Ricker wavelet (second derivative of a Gaussian function), a common choice for seismic wave problems \cite{wang2015frequencies}, and $p(x;x_0)$ as the Gaussian point source for $x_0\in\Omega$, defined as
\begin{align}    \label{eq:Sources}
    p(x;x_0) := k_{a} e^{ -  {|x - x_0|^2}/(2\sigma_m^2)},\quad\mathcal{q}(t) := (1 - \frac{1}{2}{\alpha^2(t-t_0)^2})e^{-{\alpha^2(t-t_0)^2}/{4}}.   
\end{align}
Here, $\alpha>0$ is the central width of the Ricker wavelet controlling the gradient of the bell curve as shown in \cref{subfig:Ricker_TD}, $t_0$ is the peak arrival time, $\sigma_m$ is the source width, and $k_a$ is a scaling parameter. For the numerical modeling, we fix the peak-arrival of the seismic event as $t_0=k\pi/\alpha$, for an arbitrary $k\in \mathds{R}^+$ and $k\geq 3$ in view of preserving the symmetric shape of $\mathcal{q}(t)$ and keeping $\mathcal{q}(0)$ negligibly small; the latter ensures that the energy from the earthquake accumulates gradually rather than instantaneously.

Considering $\Nr$ many receivers, we finally define the seismograms as
\begin{equation}
    \OoI_i(t;\paramu) := \ell_i\big(\mathcal{u}(x,t;\paramu)\big), \quad  t\in [0,T],\quad 1\leq i\leq \Nr,\label{eq: output functional TD}
\end{equation}
where the linear functionals $\ell_i$ model the receiver characteristics that are located on the top of the subsurface.
\begin{remark}
    In practice, the Dirichlet boundary conditions in \cref{eq: Time Domain problem} are replaced by suitable absorbing boundary conditions or a perfectly matched layer \cite{komatitsch2003perfectly}. However, for the sake of simplicity, we have omitted absorbing boundary conditions. To reduce the influence of reflections on the seismograms, we assume a sufficiently large separation between the boundaries and our region of interest by extending the domain uniformly as done in the numerical results.  
\end{remark}
\subsection{Methodology and contributions of this paper}\label{sec:Methodlogy}
\begin{figure}[t]
\centering
\subfloat[]{\label{subfig:Ricker_TD}\includegraphics[width=.33\linewidth,trim={00cm 0cm 0.0cm 0cm},clip]{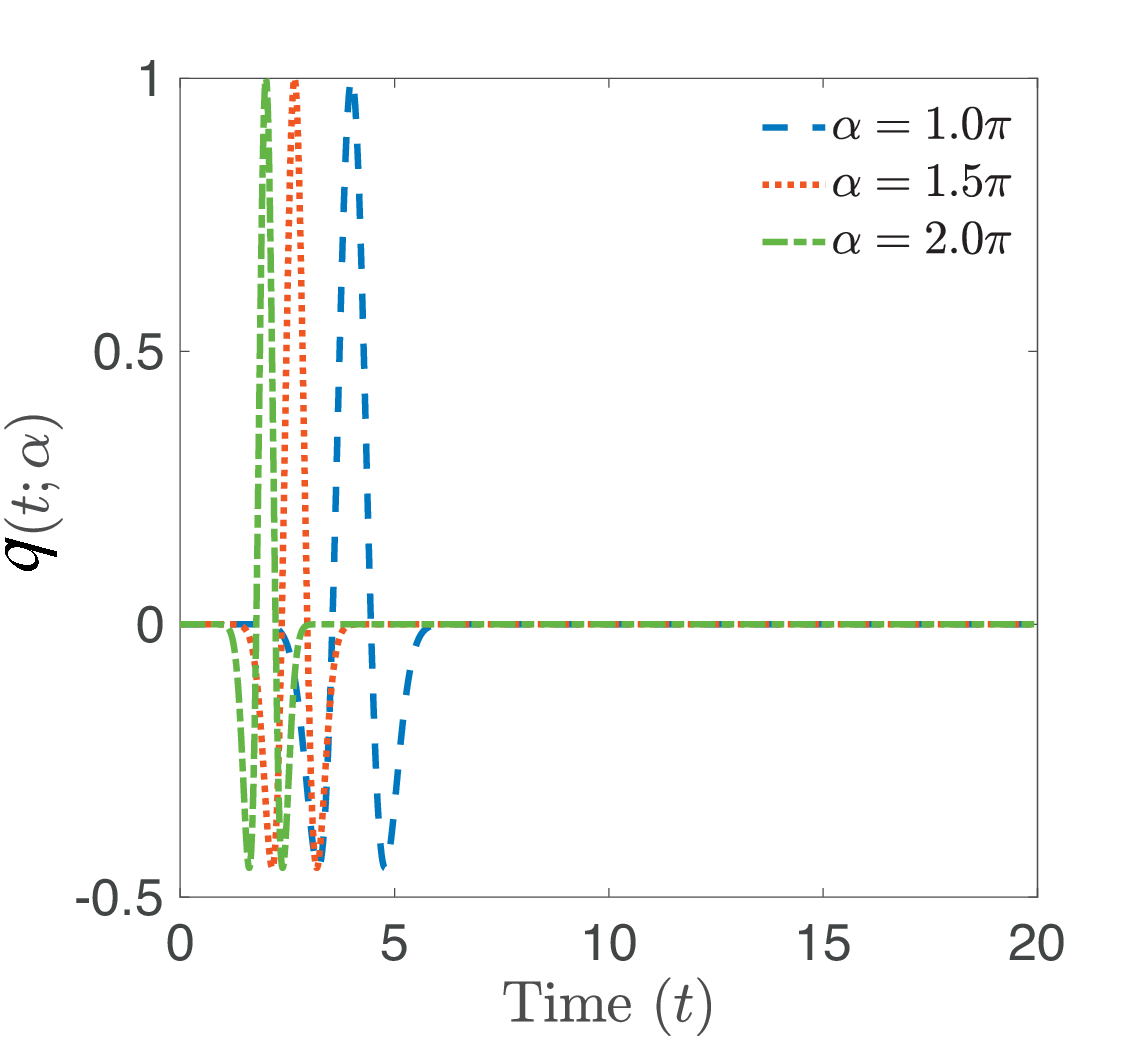}}
\subfloat[]{\label{subfig:Laplace_domain_Ricker_solution}\includegraphics[width=.33\linewidth,trim={00cm 0cm 0.0cm 0cm},clip]{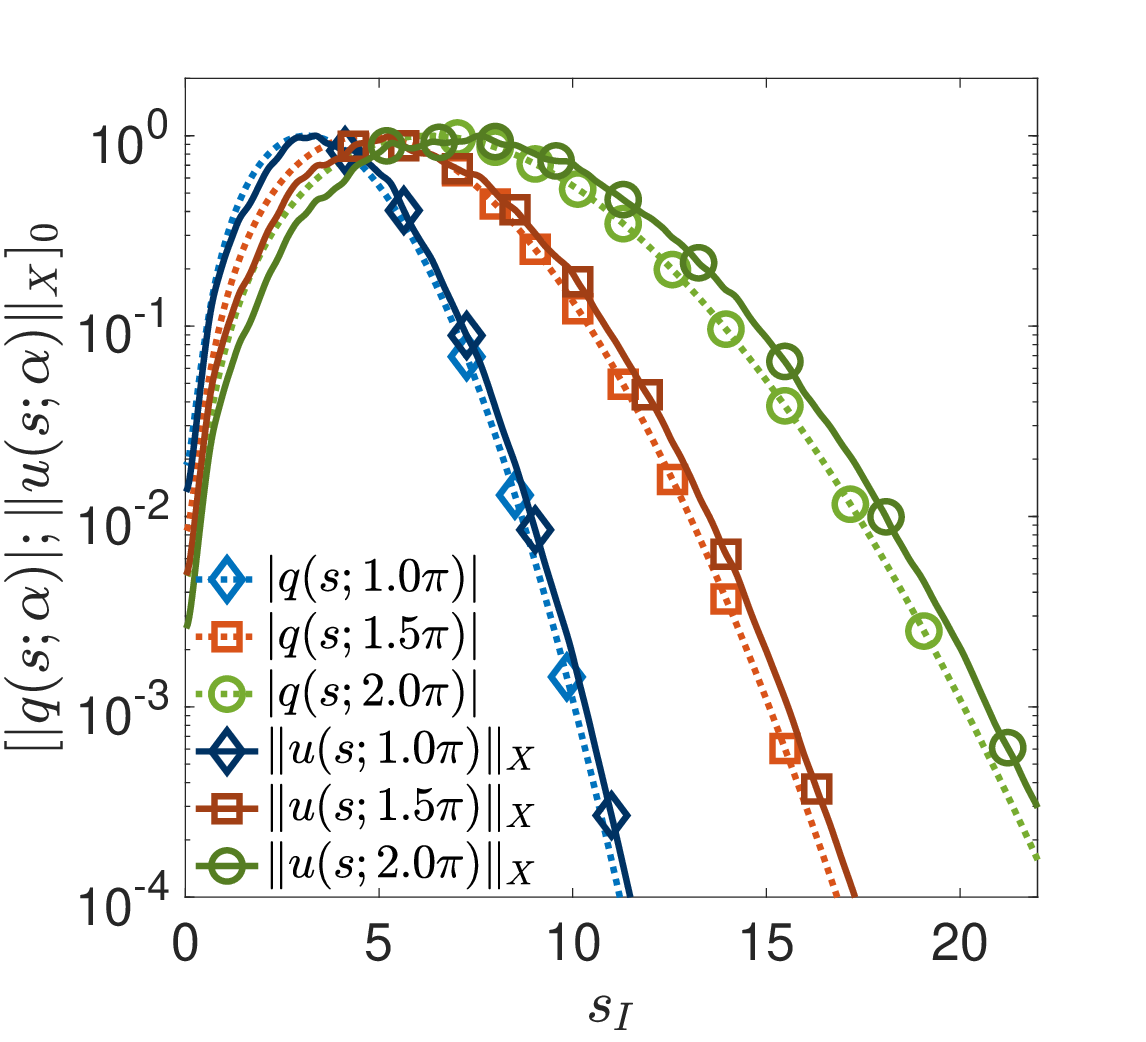}}
\subfloat[]{\label{subfig:Solution_behavior_smax}\includegraphics[width=.33\linewidth,trim={00cm 0cm 0.0cm 0cm},clip]{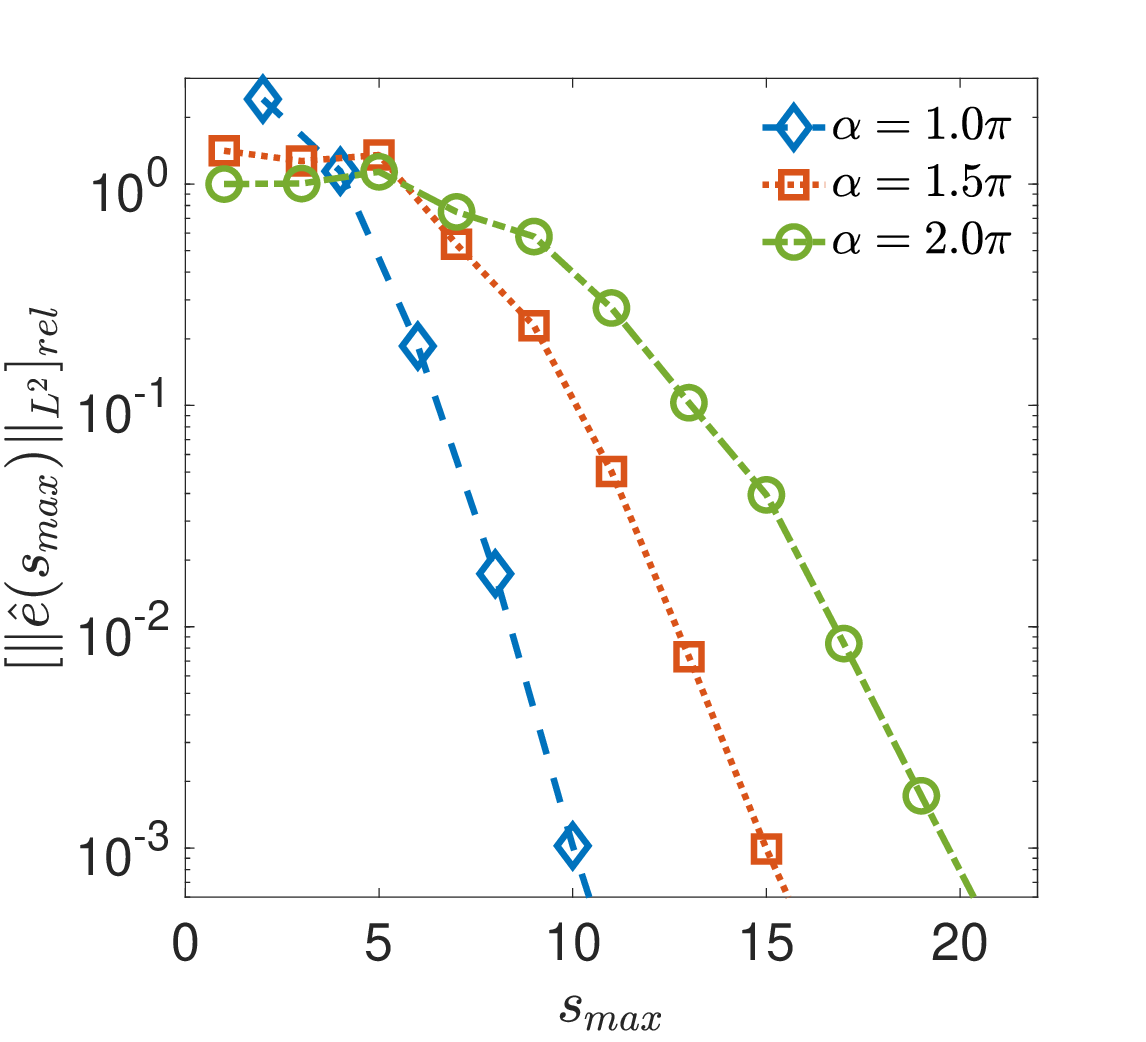}}
\caption{In (a), we plot the Ricker wavelet in the time domain for $\alpha\in\{1.0\pi,1.5\pi,2.0\pi\}$. In (b), we plot the Laplace transformed source and the frequency domain solutions for the considered problem in \cref{fig:model} for $s=0.26+\imagI\imags$ with $\imags\in[0.01,22]$. In (c), we plot the error $\hat{e}(s_{max}):=\OoI_{nbm} - \hat{\OoI}(s_{max})$ for increasing values of $s_{max}$. Here, $\OoI_{nbm}$ and $\hat{\OoI}(s_{max})$ denotes the approximation from the implicit Newmark-beta method approximation and and Weeks' method (cf., \cref{sec:NLI_Weeks}), respectively. }\label{fig:Ricker_wavelet_TD_FD_sol}
\end{figure}
 We use model order reduction via RB methods (see e.g., \cite{Rozza2008ReducedEquations}) to construct a rapid forward model solver for \cref{eq: Time Domain problem} without compromising the accuracy. The construction of the RB space for \cref{eq: Time Domain problem} is, however, challenging for the following reasons: 
 (i) Reducibility: The spatial solution trajectories experience (smooth) gradients across the interfaces, which can influence the Kolmogorov $n$-width\footnote{The Kolmogorov $n$-width is a measure of the best possible error attainable using elements of an $n$-dimensional linear space \cite{kolmogoroff1936uber}.} decay (cf., \cite{arbes2023kolmogorov}).
(ii) Stability: Classical RB methods for time-dependent problems, such as the POD-Greedy algorithm, do not generally guarantee stable RB approximations \cite{Peng2016SymplecticSystems}. 
(iii) Computational costs: Computation of the solution trajectories for a large parameter set can be expensive. 
Moreover, when using a time-stepping method also the solution of the ROM can be still computationally demanding.

We use the Laplace transformation to take advantage of the low-frequency restriction in practical seismological problems \cite{woodward1992wave,Pratt2012SeismicModel,du2000noise,sriyanto2023performance} and control the implicitly the frequency content by using the source function $\mathcal{q}(t)$ as in \cref{eq:Sources}. Let $s:=\reals + \imagI \imags \in \CC$, $\reals>0$, then, by applying the Laplace transform to $\mathcal{q}(t)$, we obtain:% the frequency domain source function:
\begin{equation}
    q(s) := e^{-\alpha^2t_0^2/4}(\alpha^3t_0 + 2\alpha s - 2\sqrt{\pi}s^2\text{erfc}(y)e^{y^2}),\quad y={(2s - \alpha^2t_0)}/{(2\alpha)}.
\end{equation}
Here, erfc$(\cdot)$ denotes the complementary error function, and $q(s)$ is analytic in the extended complex plane (cf., \cref{supplementary:Rickeranalytic}). We observe that for $\imags\gg \reals$, up to the numerical precision $\epsilon_m$ (given by the value of the Ricker wavelet at time zero), we have $|q(s)|\sim C_q|s|^2e^{- \imags ^2/\alpha^2}$ for some $C_q>0$. 
Here, we highlight that by varying $\alpha$, we can control the gradient of $\mathcal{q}(t)$ which influences the exponential decay of $|q(s)|$ (see \cref{subfig:Ricker_TD,subfig:Laplace_domain_Ricker_solution}) and, consequently, the frequency content used in the forward model.
A similar decay is observed in \cref{subfig:Laplace_domain_Ricker_solution} for the energy of the solutions $\ldu(\paramu)=\ldu(x,s;\paramu)$ of the elastic wave problem in the frequency domain:
\begin{alignat}{2}
    \nonumber  s^2 \rho \,\ldu(\paramu) - \nabla \cdot \sigma(\ldu(\paramu);\paramu)&=  f(x,s),  \quad &&\text{in } \Omega,\\
    \sigma(\ldu(\paramu);\paramu)\cdot \bm{n} &= 0,  &&\text{on } \Gamma_{N},\label{eq: Laplace Domain problem}\\
       \nonumber \ldu(\paramu) &= 0, \quad &&\text{on } \Gamma_{D},
\end{alignat}
where $f(x,s)=p(x;x_0)q(s)$, and the frequency domain seismograms read:
 \begin{equation}
    z_i(s;\paramu) := \ell_i\big(\ldu(\cdot,s;\paramu)\big) \quad \text{for }  i=1,...,\Nr. \label{eq: Output functional LD}
\end{equation} 
Thus, for some $\reals>0$ we can choose $ s_{max}>0$ depending on an acceptable error tolerance to define the frequency range of interest
\begin{equation}
     I_s:=\{ s=\reals + \imagI s_I \;:\; -s_{max}\leq s_I \leq  s_{max}\},\label{eq:definition of Is}
 \end{equation}
which acts as a band-pass filter to retain the important information (dominating energy) and neglect the rest; see \cref{remark:choice_smax}. 

Problem \cref{eq: Laplace Domain problem} is elliptic and, as we show in \cref{prop:infsup cont.}, uniformly inf-sup stable and well-posed for $s \in I_s$. As $q(s)$ is analytic we can thus show in \cref{prop:kolmogorov} that the Kolmogorov $n$-width decays exponentially. Together with the restriction to the low-frequency range $I_s$, this addresses the challenges of reducibility and stability. 

To approximate the time-dependent seismograms corresponding to \cref{eq: Output functional LD}, we perform numerical Laplace inversion via Weeks' method \cite{Weeks1966NumericalFunctions}. Similar to the decay of energy of the solutions, we observe exponential decay of the $L^2-$norm error in the time domain seismograms for increasing $s_{max}$ (see \cref{subfig:Solution_behavior_smax}). We use the structure of Weeks' method to derive an a posteriori error estimator for the error in the time domain seismograms (see \cref{prop:A_post_est}), which drives the Greedy or POD-Greedy algorithm in the Laplace domain. We highlight again that the uniform lower bound for the inf-sup constant in \cref{prop:infsup cont.} that is required for the a posteriori error estimator is explicitly computable and thus eliminates the need for expensive approximation methods such as the successive constraint method \cite{MR2367928}. 

As mentioned earlier, we emphasize that we do not need to solve the time-dependent problem at all and only need to solve problem \cref{eq: Laplace Domain problem} for few selected parameters $\paramu$ and $s$ for the Greedy algorithm and for few selected parameters $\paramu$ but in parallel for all parameters $s$ in a finite-dimensional training set that is dense in $I_s$ for the POD-Greedy algorithm. This makes our approach well-suited for modern computational architectures. Moreover, we observe in our numerical experiments that for a fixed parameter $\paramu$ the computational time for the construction of the RB space and reconstruction of time domain seismograms is smaller than their sequential computation using time-stepping methods. This is, for instance, useful for FWI since every grid point can be treated as a potential parameter in $\Pspace$, which might necessitate generating or updating RB spaces in every iteration of the optimization problem.

\section{Frequency and time domain approximations}\label{sec:Approximation and NLI}

%%%%%%%%%%%%%%%%%%%%%%%%%%%%%%%%%%%%%%%
\subsection{Weak formulation} \label{sec:Weak_form_LT}
 
 {\color{black} Let $\paramu\in \Pspace$ and $s:=\reals+\imagI \imags\in \CC$ with $\reals>0$.} We introduce the Sobolev space $\U:=\{\ldw\in H^1(\Omega;\mathds{C}^d) : \ldw|_{\Gamma_D} = 0\}$, which we equip with the inner product $(\ldu,\ldw)_{\U;\paramu} := (\rho \ldu,{\ldw})_{L^2} + b(\ldu,\ldw;\paramu)$, and denote by $\|\ldu\|_{\U;\paramu}:=(\ldu,\ldu)^{1/2}_{\U;\paramu}$ the induced norm. Here, $b(\ldu,\ldw;\paramu): \U \times \U \rightarrow \mathds{C}$ is defined as
\begin{equation}\label{eq:sesquilinearform_elasticity}
    b(\ldu,\ldw;\paramu) := \int_\Omega 2\mu(x) \varepsilon(u):\varepsilon(\bar{w})+\lambda(x) (\nabla\cdot u )(\nabla \cdot \bar{w}) dx.
\end{equation}
{\color{black} For a reference parameter $\paramu^*=(\lambda^*,\mu^*)\in\Pspace$ we set $\|\ldu\|_{\U}:=\|\ldu\|_{\U;\paramu^*}$. In view of \cref{eq:Lame_bounds}, we have for all $\paramu\in\Pspace$ the norm equivalence $c_1\|\ldu\|_{\U}\leq \|\ldu\|_{\U;\paramu} \leq c_2 \|\ldu\|_{\U}$ with constants $c_1 :=\min\{\|\lambda_0/\lambda^*\|_\infty^{-1},\|\mu_0/\mu^*\|_\infty^{-1}\}$ and $c_2:=\max \{\|\lambda_1/\lambda^*\|_\infty, \|\mu_1/\mu^*\|_\infty \}$.}

The weak formulation of \Cref{eq: Laplace Domain problem} then reads: Find $\ldu(s;\paramu) \in \U$ such that 
\begin{align}
    B(\ldu(s;\paramu),\ldw\,;s;\paramu) = q(s)(p,\ldw)_{L^2}, \quad \forall \; \ldw\in \U, \label{eq: Weak formulation Laplace domain}
\end{align}
where we define the parameterized sesquilinear form $ B(\cdot,\cdot;s;\paramu): \U \times \U \rightarrow \mathds{C}$ as 
\begin{align}\label{eq:sesquilinearform_LD}
    B(\ldu,\ldw\,;s;\paramu) := s^2(\rho\ldu,\ldw)_{L^2} + b(\ldu,\ldw;\paramu).
\end{align}
The next result ensures well-posedness of the weak formulation \cref{eq: Weak formulation Laplace domain}.
\begin{proposition}[Well-posedness]\label{prop:infsup cont.}
   {\color{black} For $s:=\reals + \imagI \imags \in I_s$ and $\paramu\in\Pspace$, we have}
\vspace{0.1cm}
\noindent (i) Inf-sup and transposed inf-sup conditions: 
\begin{align}
    \inf_{\ldu \in \U } \sup_{\ldw \in \U } \frac{|B(\ldu,\ldw;s;\paramu)|}{\|\ldu\|_{\U;\paramu} \| \ldw\|_{\U;\paramu} }&\geq \beta(s;\paramu),\quad %\label{eq:inf-sup constant def FOM cont.}
       \inf_{\ldw \in \U } \sup_{\ldu \in \U } \frac{|B(\ldu,\ldw;s;\paramu)|}{\|\ldu\|_{\U;\paramu} \| \ldw\|_{\U;\paramu} } \geq \beta(s;\paramu),\label{eq:transposed infsup cont.}
\end{align}
    where %for ${\tau}_*(s) := 1-\frac{c_R}{c_R^2+c_I^2}$, we have
    \begin{align}
    \beta(s;\paramu)=\begin{cases} |s^2(1-\tau_*(s)) + \tau_*(s)|, & \text{if } \tau_*(s) \in [\tau_{0}, 1], \\
      |s^2(1-\tau_0) + \tau_0| , & \text{if } \tau_{0}> \tau_*(s),\\
      1, & \text{if } \tau_*(s)> 1. 
        \end{cases}
        \label{eq:betaLB cases continous}
    \end{align} 
       {\color{black} Here, ${\tau}_*(s) := 1- {c_r}/{(c_r^2+c_i^2)}$ for $c_r := 1 - (\reals^2-\imags^2)$ and $c_i := 2\reals \imags$. Moreover, $\tau_0:=\CK/(1+\CK)>0$, where $\CK>0$ satisfies $|b(\ldw,\ldw;\paramu)|\geq \CK \|\ldw\|^2_{H^1_{\rho}},\forall\ldw\in\U$ in the weighted norm $\|\ldw\|^2_{H^1_{\rho}}:=\|\rho^{1/2}\ldw\|^2_{L^2} + \|\nabla\ldw\|^2_{L^2}$.
       }
         \vspace{0.1cm}
         
\noindent (ii)  Uniform bounds: Let $\underline{\beta}:=\frac{c_1}{c_2} \beta_*(\reals; s_{max})$, then 
    \begin{align}\label{eq:infsup cont. m0}
    \inf_{\ldu \in \U } \sup_{\ldw \in \U } \frac{|B(\ldu,\ldw;s;\paramu)|}{\|\ldu\|_{\U} \| \ldw\|_{\U} }&\geq \underline{\beta},\quad 
       \inf_{\ldw \in \U } \sup_{\ldu \in \U } \frac{|B(\ldu,\ldw;s;\paramu)|}{\|\ldu\|_{\U} \| \ldw\|_{\U} } \geq \underline{\beta},\quad \forall\;\paramu \in \Pspace,
    \end{align}
    where $\beta_*(\reals; s_{max}):=\min\{1, \reals^2,2\reals/\sqrt{4\reals^2+s^2_{max}}\}$ and $\underline{\beta}>0$.
    
\vspace{0.1cm}

\noindent (iii) Continuity condition: Let $C_B:=\max\{1,\reals^2+ s_{max}^2 \}c_2^2$, then for all $\paramu\in \Pspace$,
\begin{equation}\label{eq:continuity_B}
    |B(\ldu,\ldw;s;\paramu)| \leq C_B\|\ldu\|_{\U}\|\ldw\|_{\U} \quad \forall \ldu,\ldw \in \U.
\end{equation}
\end{proposition}

\begin{proof}
The proof is based on standard arguments, where we use the Riesz representation theorem to define a supremizer operator $T:\U\rightarrow\U$ by $(T(s;\paramu)\ldu,\ldw)_{\U;\paramu}=B(\ldu,\ldw;s;\paramu)$ for all $\ldu,\ldw\in\U$. We expand $B(\cdot,\cdot;s;\paramu)$ using eigenfunctions $v(\paramu)$ for the generalized eigenvalue problem $b(v(\paramu),\ldw;\paramu)=\tau(\paramu) (v(\paramu),\ldw)_{\U;\paramu}$ for all $\ldw\in\U$ and eigenvalues $\tau(\paramu)$.  
Next, by using the definition of $(\cdot,\cdot)_{\U;\paramu}$ we find that $B(v(\paramu),\ldw;\paramu)=\chi(s;\paramu)(v(\paramu),\ldw)_{\U;\paramu}$, where $\chi(s;\paramu)=s^2(1-\tau(\paramu)) + \tau(\paramu)$. Using the definition of the supremizer operator we note that $\chi(s;\paramu)$ are the eigenvalues of $T$.
Next, we treat $|\chi(s;\paramu)|^2$ as a quadratic polynomial in $\tau(\paramu)$ to find the minimizer $\tau_*(s)$ defined after \cref{eq:betaLB cases continous}. 
To find $\tau_0$ independent of $\paramu$, we use the definition of $(\cdot,\cdot)_{\U;\paramu}$ in $b(v(\paramu),\ldw;\paramu)=\tau(\paramu) (v(\paramu),\ldw)_{\U;\paramu}$, and after rearranging we get $b(v(\paramu),\ldw;\paramu)=\frac{\tau(\paramu)}{1-\tau(\paramu)}(\rho v(\paramu),\ldw)_{L^2}$. Next, following Korn's inequality \cite{horgan1995korn}, we use $|b(\ldw,\ldw;\paramu)|\geq \CK \|\ldw\|^2_{H^1_{\rho}}$ in the weighted norm and $(\rho v(\paramu),v(\paramu))_{L^2}\leq \| v(\paramu)\|_{H^1_{\rho}}^2$ to find $\tau_0=\CK/(1+\CK)>0$ such that $\tau_0\leq \tau(\paramu)\leq 1$ for all $\paramu\in \Pspace$.
Finally, standard  calculus yields $|\chi(s;\paramu)|\geq \beta(s;\paramu)$ for all $s\in I_s$ and $\paramu\in\Pspace$. The transposed inf-sup condition, continuity constant and uniform bounds are straightforward. For the readers convenience, we have added a detailed proof in \cref{supplementary:Wellposedness}. 
\end{proof}

\begin{remark}
 We obtain the same lower-bound $\underline{\beta}$ as in \cref{eq:infsup cont. m0} for the discrete counterpart of the inf-sup constant, which will be used in the a posteriori error estimator derived in \cref{sec:A posteriori error estimate}. Numerical results for the considered test case show that the lower bound is very close to the inf-sup constant of the discrete problem; see \cref{supplementary:infsup_plots}. Moreover, the solutions at higher frequencies might become unstable due to the $\mathcal{O}(|{s}|^2)$ growth for $C_B$ and $\mathcal{O}(|{s}|^{-1})$ decay in the inf-sup constants.
\end{remark}
\subsection{Finite Element approximation}\label{sec:truth}
 Let $\Uh \subset \U$ be a conforming finite element space of dimension $N_h$. Then the Galerkin approximation of \cref{eq: Weak formulation Laplace domain} reads: For any $\paramu \in \Pspace$ and $s\in \CC$ with $\reals>0$, find $\ldu_h(s;\paramu)\in \Uh$ such that 
 \begin{equation}
B(\ldu_h(s;\paramu),\ldw_h;s;\paramu) = q(s)(p,\ldw_h)_{L^2}, \quad \forall \; \ldw_h \in \Uh. \label{eq: Weak formulation discrete}
\end{equation}
The well-posedness of problem \cref{eq: Weak formulation discrete} can be proved similarly to \cref{prop:infsup cont.} with the same stability constants.  
Next, we define the discrete counterpart of the frequency domain seismogram 
\begin{equation}
    z_{h,i}(s;\paramu) := \ell_i\big(\ldu_h(s;\paramu)\big) \quad \text{for } i=1,...,\Nr\;,\paramu\in\Pspace, \text{and } s\in I_s.\label{eq: output functional FD discrete}
\end{equation}

\begin{remark}[Low-frequency restriction using $ s_{max}$]\label{remark:choice_smax}
The standard stability estimate yields $|z_{h,i}(s;\paramu)|\leq \kappa(s)$ for all $s\in I_s$, $\paramu\in\Pspace$ and for $i=1,\ldots,N_r$, where ${\underline{\beta}}$ is  defined as in \cref{eq:infsup cont. m0}, $\kappa(s):=\|\ell_i\|_{\U_h^*}\|P\|_{\U_h^*}{|q(s)|}/{\underline{\beta}}$, $\|\ell_i\|_{\U_h^*}:=  \sup_{w_h\in\Uh} |\ell_i(w_h)|/\|w_h\|_{\U}$, and $\|P\|_{\U_h^*}:=  \sup_{w_h\in\Uh} |(p,w_h)_{L^2}|/\|w_h\|_{\U}$. 
For $ s_{max} \rightarrow \infty$, we 
observe that ${|q(s)|}/{\underline{\beta}}$ is of the same (negligibly small) order as $\epsilon_{0}:=\mathcal{q}(0)\cdot c_2/{(4\reals c_1)}$.\footnote{Since, $\lim_{ s_{max}\rightarrow\infty} {|q(s_R+\imagI s_{max})|}/{\underline{\beta}(\reals; s_{max})}=\mathcal{q}(0)\cdot c_2/{(4\reals c_1)}$.} 
Therefore, for a given desired accuracy $\epsilon_*\geq C\epsilon_{0} >0$ where $C>0$ depends on $\ell_i$ and $p$, we define $ s_{max}$ as the infimum of the set $\{\omega'>0: \kappa(\reals\pm\imagI \omega)\leq \epsilon_*\text{ for all } \omega>\omega'\}$,
which we also use for the time domain approximations in \cref{sec:NLI_Weeks}.
\end{remark} 

\subsection{Numerical Laplace inversion via Weeks' method \cite{Weeks1966NumericalFunctions}}\label{sec:NLI_Weeks}
Let $z(s)$ be an arbitrary Laplace-transformed function, then its time domain approximation via Weeks' method is defined as
\begin{align}
     \hat{\OoI}(t) : = e^{\wr t}\sum_{p=0}^{N_z-1} {a}_{p}(z;\wr,\wi) e^{-\wi t}L_p(2\wi \,t),\quad\label{eq:Weeks Laguerre sum v2} %
\end{align}
where $L_p(\cdot)$ denotes the Laguerre polynomial of degree $p$. Here, $\wr,\wi\in\mathds{R}^+$ are free parameters and the expansion coefficients are defined as
\begin{equation}\label{eq:a_p original}
    {a}_p(z;\wr,\wi) := \frac{\wi}{N_z} \sum_{j=-N_z}^{N_z-1}\frac{e^{-\imagI p\theta_{j+1/2}}}{1 - e^{\imagI\theta_{j+1/2}}} z(\hat{s}_j), \text{ for } \, \hat{s}_j= \wr + \imagI \wi \cot \frac{\theta_{j+1/2}}{2},
\end{equation}
where $\theta_j=j\pi/N_z$ for $j=-N_z,\ldots,N_z-1$ (cf., \cite{Brio2005ApplicationExponential,Weeks1966NumericalFunctions,Weideman1999AlgorithmsTransform}). 
In view of \cref{remark:choice_smax} and the negligibly small energy of the solutions after $s_{max}$ (see \cref{fig:Ricker_wavelet_TD_FD_sol}), we approximate $a_p(z;\wr;\wi)$ by considering only those $j$ for which $\Im(\hat{s}_j)<s_{max}$.  
We define
\begin{equation}\label{eq:shat_new_Week}
    I_{w} = \Big \{\wr + \imagI \wi \cot \frac{\theta}{2}\,:\,\theta\in[-\pi,\pi]\setminus \{0\}\text{ and } |\theta| \geq 2\arctan \frac{\wi}{ s_{max}}\,\Big\},
\end{equation}
and replace $a_p(z;\wr;\wi)$ in \cref{eq:Weeks Laguerre sum v2} with $ {a}^{\epsilon}_p(z;\wr,\wi)$ defined as
\begin{equation}
    {a}^{\epsilon}_p(z;\wr,\wi) := \frac{\wi}{N_z} \sum_{\hat{s}_j\in I_w} \frac{e^{-\imagI p\theta_{j+1/2}}}{1 - e^{\imagI\theta_{j+1/2}}} z(\hat{s}_j), \text{ for } \hat{s}_j= \wr + \imagI \wi \cot \frac{\theta_{j+1/2}}{2},\label{eq:a_p truncated}
\end{equation}
for $\theta_j=j\pi/N_z$, $j=-N_z,\ldots,N_z-1$ and $p=0,\ldots,N_z-1$. Here, ${a}^{\epsilon}_p(z;\wr,\wi)$ is the filtered approximation of $a_p(z;\wr;\wi)$.

We perform the series summation in \cref{eq:Weeks Laguerre sum v2} using the Clenshaw algorithm \cite{Clenshaw1955ASeries}. In addition, the wavefield at any $t\in[0,T]$ is reconstructed as
\begin{equation}\label{eq:Weeks_u_HFM}
    \utdc_h(t;\paramu) : = e^{\wr t}\sum_{p=0}^{N_z-1} {a}^{\epsilon}_{p}(u_h;\wr,\wi) e^{-\wi t}L_p(2\wi \,t).
\end{equation}

\begin{remark} 
 We choose Weeks' method due to its explicit expression for the time domain approximations, as also considered in \cite{Bigoni2020SimulationbasedMonitoring} for structural-health monitoring. 
Because of the distribution of the singularities over the whole imaginary axis, it is challenging to use methods based on contour deformation (cf., \cite{guglielmi2022contour,talbot1979accurate,weideman2007parabolic}). 
\end{remark}
The success of Weeks' method depends on the selection of the parameters $(\wr,\wi)$ which helps to balance the exponential growth in time and the behavior of $L_p$ for large $p$ 
\cite{Weeks1966NumericalFunctions,Weideman1999AlgorithmsTransform,Brio2005ApplicationExponential}. Next, we present an algorithm to select $(\wr,\wi)$. 
\subsubsection{Optimal parameter selection} \label{sec:parameter selection Week}
 Let $I_r$ and $I_i$ be heuristically selected intervals in which $(\wr,\wi)$ are likely located, and let $\epsilon_m$ denote the machine precision. Then, we find optimal $(\wr,\wi)$ by solving the following nested optimization problem (\cite[Algorithm 2]{Weideman1999AlgorithmsTransform}): Find $(\wr,\wi)\in I_r\times I_i$, such that 
\begin{equation}
\begin{split}
     &\wr = \underset{s_r \in I_r}{\text{arg}\min}\; e^{s_r  T} \Big(\sum_{p=\Nz}^{2\Nz-1} |{a}_p^{\epsilon}(z_h;s_r ;\wi(s_r ))| + \epsilon_m \sum_{p=0}^{\Nz - 1}|{a}_p^{\epsilon}(z_h;s_r ;\wi(s_r ))|   \Big), \\   & \text{where }  \wi(s_r ) =\underset{s_i \in I_i }{\text{arg}\min}\; \sum_{p=N_z}^{2N_z-1}|{a}_p^{\epsilon}(z_h;s_r ;s_i)|. \label{eq:min problem Weeks method sr}
\end{split}
\end{equation}
 As one novel contribution of this paper, we propose replacing the full order seismogram $z_h(\hat{s};\paramu)$ in \cref{eq:min problem Weeks method sr} by a reduced order approximation $z_k(\hat{s};\paramu)$, which can be evaluated much faster, cf., \cref{prop:kolmogorov}; for details on how to obtain $z_k(\hat{s};\paramu)$ we refer to \cref{sec:RB approximation and error estimation,sec:MOR for NLI,sec:parametric reduction}. 
\begin{remark}[Computational realization]\label{rem:opti_para_week}
    We use Brent's algorithm \cite{brent2013algorithms} to solve the nested optimization problem as suggested in \cite{Weideman1999AlgorithmsTransform}, with the choice $I_r= [0.12,0.62]$ and $I_i = [0.01,s_{max}]$ after following the considerations in \cite[Sec. 3.5]{Brio2005ApplicationExponential}. Brent's algorithm can potentially exhibit slower convergence if the minimum is on the boundary of the interval. However, we observed a rapid convergence for our problem. Moreover, in the numerical results we observe that the optimal parameters $(\wr,\wi)$ depend only weakly on $\paramu\in\Pspace$. Therefore, we compute the optimal parameters only once for the reference parameter $\paramu^*\in \Pspace$ and re-use them for all other $\paramu\in\Pspace$.
\end{remark}
%%%%%%%%%%%%%%%%%%
%%%%%%%%%%%%%%%%%%
%%%%%%%%%%%%%%%%%%
%%%%%%%%%%%%%%%%%%
\section{Reduced basis method} \label{sec:RB approximation and error estimation}
In this section, we present a Galerkin RB formulation of the parametric elastic wave equation for a computationally efficient approximation of seismograms. We highlight the practical significance of the low-frequency restriction, which is crucial for ensuring stability and achieving (provably) exponential convergence for the RB approximation errors. Furthermore, we derive a novel and computationally efficient a posteriori error estimator for the time domain RB approximation error in the seismograms which we use in \cref{sec:parametric reduction} to construct the RB space without computing any time-dependent counterparts.
\subsection{Galerkin RB approximations}\label{sec:RB approx FD}
Let $\Urb\subset \Uh$ be a subspace of dimension $N_k\ll \Nh={\rm dim}(\Uh)$. The Galerkin RB approximation in the frequency domain is defined as follows: For any $\paramu\in\Pspace$ and $s\in I_s$, find $\ldu_k(s;\paramu)\in\Urb$ such that 
\begin{equation}
    B(\ldu_k(s;\paramu),\ldw_k;s;\paramu) = q(s)(p,\ldw_k)_{L^2}\quad \forall \, \ldw_k\in\Urb. \label{eq:weak form RB}
\end{equation}
We obtain well-posedness of \cref{eq:weak form RB} analogously to \eqref{eq: Weak formulation discrete} for all $s\in I_s$ and $\paramu\in\Pspace$.

Next, for any receiver location index $i=1,...,\Nr$, we define the RB approximation of the frequency domain and the time domain seismograms as
\begin{align}
     z_{k,i}(s;\paramu) &:= \ell_i\big(\ldu_k(s;\paramu)\big),\quad s\in I_s,\label{eq: output functional FD discrete RB}\\ 
    {\hat{\mathcal{z}}}_{k,i}(t;\paramu) &:= e^{\wr t}\sum_{p=0}^{N_z-1} {a}_p^{\epsilon}(z_{k,i}(\paramu);\wr;\wi) e^{-\wi t}L_p(2\wi \,t),
\end{align}
where ${a}_p^{\epsilon}(\mathcal{z}_{k,i}(\paramu);\wr;\wi)$ is defined as in \cref{eq:a_p truncated} using RB approximations $z_{k,i}(\hat{s};\paramu)$ for $\hat{s}\in I_w$. In addition, the time domain wavefield can be reconstructed as
\begin{align}\label{eq:Weeks_u_RBM}
    {\hat{\mathcal{u}}}_{k}(t;\paramu) &:= e^{\wr t}\sum_{p=0}^{N_z-1} {a}_p^{\epsilon}(u_{k}(\paramu);\wr;\wi) e^{-\wi t}L_p(2\wi \,t).
\end{align} 
The next proposition establishes the exponential decay of the Kolmogorov $n-$width for approximating $s\mapsto u(s;\paramu)$, $s\in I_s$, for a fixed $\paramu$, similar to the procedure in \cite[Theorem 3.1]{ohlberger2015reduced} for coercive elliptic problems.
%%%%%%%%%%%%%%%%%%%%%%%%%%%%%%%%%%%%%%%%%%%%%%%%%%%%%%
%%%%%%%%%%%%%%%%%%%%%%%%%%%
%%%%%%%%%%%%%%%%%%%%%%%%%%%
%%%%%%%%%%%%%%%%%%%%%%%%%%%
%%%%%%%%%%%%%%%%%%%%%%%%%%%

\begin{proposition}[Exponential decay of Kolmogorov $n$-width]\label{prop:kolmogorov}
    Let $\paramu\in\Pspace$ be fixed
    and define $\mathcal{M} := \{u(s;\paramu)\, :\,u(s;\paramu) \;\text{solves } \cref{eq: Weak formulation Laplace domain} \text{ for } s \in I_s \}$. Then, the Kolmogorov $n-$width $d_{n}(\mathcal{M}):= \inf_{\underset{\dim(\U_n)=n}{\U_n\subset \U}}\sup_{s\in I_s} \inf_{v_n\in \U_n}\|u(s;\paramu) - v_n\|_{\U}$ satisfies
\begin{equation}
    d_{n}(\mathcal{M}) \leq \frac{\hat{c}_1 s_{max}}{2s_R} e^{-\frac{\ln(2)  \reals}{ s_{max}} n}, \qquad n\in \mathbb{N},\label{eq:n_width_kolmo}
\end{equation}
for some fixed constant $\hat{c}_1>0$.
\end{proposition}
 \begin{proof}
We first observe that, since $q(s)$ is an analytic function of $s$ (see \cref{supplementary:Rickeranalytic}) and \cref{eq: Weak formulation Laplace domain} is well-posed for any $s$ in the right half-plane, $u(s)$ is an analytic function in an open neighborhood of $I_s$. Hence, $u(s)$ can be expanded as a power series locally around any $s\in I_s$ with convergence radius larger than $s_R/2$. 
Covering $I_s$ with $2s_{max}/s_R$ many open balls of radius $s_R/2$, and following the arguments given in the proof of \cite[Thm 2.1]{ohlberger2015reduced}, we arrive at the result.
\end{proof}

\subsection{A posteriori error estimator}\label{sec:A posteriori error estimate}
In this part, we present an a posteriori error estimator which is utilized in the construction of $\Urb$ in \cref{sec:MOR for NLI,sec:parametric reduction}.
First, we recall the well-known residual-based error estimator. Let $s:=\reals + \imagI \imags \in I_s$ and $\paramu\in\Pspace$, then for the residual
\begin{equation}
    r_h(\ldw_h;s;\paramu) := q(s)(p,\ldw_h)_{L^2} - B(\ldu_k(s;\paramu),\ldw_h;s;\paramu), \quad \forall \ldw_h\in \Uh,
\end{equation}
we exploit the error-residual relationship, similar as e.g., in \cite{Rozza2008ReducedEquations}, to obtain
\begin{align}\label{eq:standard_res_estimator}
    \|\ldu_h(s;\paramu) - \ldu_k(s;\paramu) \|_{\U} \leq {\Delta}_k(s;\paramu):= \frac{1}{\underline{\beta}} \|r_h(s;\paramu)\|_{\U_h^*},
\end{align}
where $\|\cdot\|_{\U_h^*}$ is the dual norm as in \cref{remark:choice_smax} and $\underline{\beta}>0$ is the lower bound of the (discrete) inf-sup constant from \cref{prop:infsup cont.}. In addition, the effectivity $\eta(s;\paramu)$ of the a posteriori error estimator satisfies
\begin{equation}\label{eq:effectivity}
    1 \leq \eta(s;\paramu):=\frac{{\Delta}_k({s};\paramu)}{\|u_h({s};\paramu) - u_k({s};\paramu)\|_{\U}} \leq \frac{C_B}{\underline{\beta}},\quad \forall\,{s}\in I_s,\paramu\in\Pspace,
\end{equation}
where $C_B$ is as in \cref{eq:continuity_B}. 
We stress that $C_b/\underline{\beta}$ behaves like $\mathcal{O}(|s_{max}|^3)$. 

Next, by using \cref{eq:Weeks_u_HFM} we get a bound for the time domain RB approximation error in the wavefield:
\begin{equation}\label{eq:RB_approx_error_u}
    \|\hat{\utdc}_h(\paramu) - \hat{\utdc}_k(\paramu)\|_{L^2(0,T;\U)}\leq C_{\Week} \sum_{\hat{s}_j\in I_w} C_j\|\ldu_h(\hat{s}_j;\paramu) - \ldu_k(\hat{s}_j;\paramu)\|_{\U}.
\end{equation}
Here, $C_{\Week}:= \wi/N_z({\sum_{p,q=0}^{N_z-1}|\int_0^T e^{2(\wr-\wi)t}L_p(2\wi t)L_q(2\wi t)dt |})^{1/2}$ and $C_j=1/|1 - e^{\imagI \theta_{j+1/2}}|$ for $j=-N_z,\ldots,N_z-1$. Moreover, we obtain the following a posteriori error estimator for the RB approximation error in the seismograms.
\begin{proposition}[A posteriori error estimator]\label{prop:A_post_est}
Let $\paramu\in \Pspace$ and $\hat{s}_j\in I_w$ for $j=-N_z,\ldots,N_z-1$. Then, the frequency domain RB approximation error satisfies
    \begin{equation} 
    |z_{h,i}(\hat{s}_j;\paramu) - z_{k,i}(\hat{s}_j;\paramu) |\leq \Delta_{k,f}(\hat{s}_j;\paramu):=\|\ell_i \|_{\U_h^*} {\Delta}_k(\hat{s}_j;\paramu), \;i=1,\ldots,N_r.\label{eq:error estimate FD output} 
   \end{equation}
 Moreover, the RB approximation error in the time domain seismograms satisfies
   \begin{equation}
   \|\hat{\OoI}_{h,i}(\paramu) - \hat{\OoI}_{k,i}(\paramu)\|_{L^2(0,T)} \leq  \Delta_{k,t}(\paramu) :=C_\Week \sum_{\hat{s}_j\in I_w} C_j\Delta_{k,f}(\hat{s}_j;\paramu).\label{eq:OoI estimate TD}
    \end{equation}  
\end{proposition}

\begin{proof}
    We bound the time domain RB approximation errors in the seismograms using the dual norm of the bounded linear output functional and use \cref{eq:RB_approx_error_u}. The rest follows directly using \cref{eq:standard_res_estimator}. 
\end{proof}
Compared to the a posteriori error estimators for wave problems that use time-dependent evaluations of a semi-discretized problem \cite{glas2020reduced,amsallem2014error}, we avoid computing any time-dependent counterparts and numerically compute $C_{\Week}$ depending on the final time $T$ only once.  In addition, the derived lower-bound $\underline{\beta}$ for the discrete inf-sup constant is very close to the actual values shown in numerical experiments for the considered test case in \cref{supplementary:infsup_plots} and eliminates the need for expensive approximation methods such as the successive constraint method \cite{MR2367928}. After a few additional solutions of \eqref{eq: Weak formulation discrete}, one can calculate the a posteriori error estimators in \cref{prop:A_post_est} in a computational complexity independent of $N_h$; see e.g., \cite{Rozza2008ReducedEquations}.

%%%%%%%%%%%%%%%%%%%%%%%%%%%%%%%%%%%%%%%%%%%%%%%%%%%%%%
%%%%%%%%%%%%%%%%%%%%%%%%%%%
%%%%%%%%%%%%%%%%%%%%%%%%%%%
%%%%%%%%%%%%%%%%%%%%%%%%%%%
%%%%%%%%%%%%%%%%%%%%%%%%%%%
\section{Reduction in the $s-$parameter with \textit{m} fixed} \label{sec:MOR for NLI}
In this section, we construct $\Urb \subset \Uh$ to accurately compute the time-dependent seismograms, for a fixed $\paramu \in \Pspace$, in a fraction of the wall-clock time of the FOM relying on an implicit Newmark-beta method. 
This is achieved by employing the well-known POD method \cite{MR1204279,MR1868765,MR0910462} and the Greedy algorithm \cite{veroy2003posteriori} in an elliptic setting. We take advantage of the exponential Kolmogorov $n-$width decay under low-frequency restrictions (\cref{prop:kolmogorov}) to ensure that $\Urb$ is of very small dimension. Moreover, the construction of $\Urb$ can be achieved in a notably small elapsed time, especially when many processors are used for the computation of the snapshots for the POD method. This makes our approach well-suited for goal-oriented model reduction presented in \cref{sec:parametric reduction} and for seismic applications such as FWI, where the RB space might need to be updated in every iteration.

\subsection{Proper orthogonal decomposition} \label{sec:CPOD}
 The objective of the POD method \cite{MR1204279,MR1868765,MR0910462}, also known as Karhunen-Lo{\`e}ve expansion 
 or Principal Component Analysis, is to construct an ``optimal'' set of orthogonal basis functions $\{\xi_k\}_{k=1}^{N_k}$: 
 The POD basis functions $\xi_k:=\xi_k^R+\imagI \xi_k^I\in\Urb$, $1\leq k\leq N_k$ are obtained by solving the following optimization problem
\begin{align}
   \underset{\U_{k} }{\min}\frac{1}{N_s}\sum_{s\in \tilde{I}_s} \Big\|u_h(s) - \sum_{k=1}^{\Nk}\Big(u_h(s),\xi_k\Big)_{\U}\xi_k\Big\|^2_{\U},\label{eq:CPOD min problem}\; \text{s.t. } (\xi_i,\xi_j)_{\U}=\delta_{ij},\; 1\leq i,j\leq \Nk,
\end{align}
where $\U_{k} :=\text{span}\{\xi_1,\ldots,\xi_k\}\subset {\text{span}}\{u_h(s)\,:\,s\in \tilde{I}_s\}$, and $\tilde{I}_s\subset I_s$ is a finite-dimensional training set of cardinality $N_s$. 

We determine the POD basis by applying the method of snapshots \cite{MR0910462}. To that end, we introduce the Gramian matrix $G\in\mathds{C}^{N_s\times N_s}$  
\begin{equation}
G_{ij}:=\Big(u_h(s_j),u_h(s_i) \Big)_{\U}\quad  \text{for } 1\leq i,j\leq N_s, 
\end{equation}
and solve the eigenvalue problem: Find $(\boldsymbol{w}_k,\gamma_k)\in(\mathds{C}^{N_h},\mathds{R}^+)$ such that 
\begin{equation}
    G\boldsymbol{w}_k = \gamma_k \boldsymbol{w}_k\quad  \text{for } 1\leq k \leq N_k. \label{eq:CPOD correlation matrix EVP}
\end{equation}
Here, we choose the dimension of the RB space $N_k$ as the smallest integer satisfying ${\sum_{i=1}^{\Nk}\gamma_i^2}/{\sum_{i=1}^{N_s}\gamma_i^2} \geq 1 - \varepsilon_{tol}$ for some prescribed tolerance $\varepsilon_{tol}$. The POD basis functions are then given as 
\begin{equation}
    \xi_k:=\frac{1}{\sqrt{N_s\gamma_k}}\sum_{m=1}^{N_s}(\boldsymbol{w}_k)_m u_h(s_m),\quad \text{for } 1\leq k \leq N_k.
\end{equation} 
 The POD basis offers an optimal representation of the full order solutions in the $L^2$/$l_2$-sense. 
 We construct the POD basis functions using $I_s=I_r\times I_i$, which are then used to calculate $(\wr,\wi)$ by solving \cref{eq:min problem Weeks method sr}. Moreover, the same POD basis functions can be used for the time domain RB approximation of seismograms and wavefields. 
 \begin{remark}
      We recall (\cref{rem:opti_para_week}) that we compute $(\wr,\wi)$ only once for a fixed $\paramu\in\Pspace$ and reuse them when $\paramu$ is varied. In our numerical results, we observe that by treating the (coarsely sampled) imaginary part of $\hat{s}\in I_w$ as the parameter, the number of full order snapshots and the number of POD basis functions to reach the desired tolerance are reduced. Therefore, the combined wall-clock time to construct ROMs and approximating seismograms is significantly smaller than solving the FOM relying on an implicit Newmark-beta method. We conjecture that this is a potential approach to expedite seismological problems, such as FWI, where the ROM might need to be updated at every iteration step due to many parameters involved. 
 \end{remark}

\begin{remark} \label{remark: POD vs SPOD}
In contrast to other POD methods that construct separate POD basis spaces for the real and imaginary parts \cite{tonn2011comparison} or extend them into a single real-valued matrix \cite{Bigoni2020SimulationbasedMonitoring,Peng2016SymplecticSystems} to form a real-valued POD basis space, we construct a complex-valued POD basis from complex-valued snapshots. As demonstrated in \cref{fig:SPOD},
other POD-based approaches \cite{Bigoni2020SimulationbasedMonitoring,Peng2016SymplecticSystems} may only partially capture the interdependence relation, which can be crucial in heterogeneous mediums with sharp solution variations, potentially leading to a high-dimensional POD space. The use of a complex-valued POD basis, as introduced for optimal control problems in \cite{volkwein2001optimal}, is standard, and has also been successfully used, e.g., for wave motion analysis in \cite{feeny2008complex}, and Hamiltonian systems in \cite{Peng2016SymplecticSystems}. 
\end{remark}

\begin{remark}[Computational realization] \label{rem:PODvsGreedy}
 The POD method requires solving the FOM $N_s$ times and approximating  \cref{eq:CPOD correlation matrix EVP}, which can be costly for large $N_s$. Randomized methods can be utilized to efficiently approximate \cref{eq:CPOD correlation matrix EVP} (see, e.g., \cite{halko2011finding}). Moreover, by making use of parallel computations, the computation of the POD basis and the RB approximation of seismograms can be achieved in a fraction of the wall-clock time of an implicit Newmark-beta method \cite{newmark1959method}. Alternative to the POD method, the Greedy algorithm \cite{veroy2003posteriori}, presented in \cref{sec: Greedy algorithm}, is a sequential algorithm requiring only $\mathcal{O}(N_k)$ many full order solutions. 
\end{remark}

\subsection{Greedy algorithm} 
\label{sec: Greedy algorithm}
 The objective of the Greedy algorithm \cite{veroy2003posteriori} is to iteratively construct an RB space $\Urb:=\text{span}\{u_h(s_1),\ldots,u_h(s_{k}) \}$ and a corresponding set of frequency parameters $S_k:=\{s_1,\ldots,s_k\}$, where we add in each iteration the snapshot $u_h(s)$ worst approximated by $\U_{k-1}$ to the RB space. 

 In a generic $k^{\text{th}}$ iteration, we identify the parameter $s_{k+1}\in \tilde{I}_s\subset I_s$ as 
\begin{equation}
    s_{k+1}:=\text{arg max}_{s\in \Tilde{I}_s} \;{\Delta}_{k,f}(s), \label{eq:Greedy s max}
\end{equation}
and enrich the RB space $\Urb$ with the full order solution $u_h(s_{k+1})$. 
We terminate the algorithm when ${\Delta}_{k,f}(s_{k+1})$ is below the required tolerance or the maximum number of RB functions is reached. For stability reasons, we orthonormalize the basis functions in $\Urb$ using the modified Gram-Schmidt procedure. The Greedy algorithm constructs an RB space which is quasi-optimal with respect to the maximum norm over $\tilde{I}_s$ \cite{MR2821591,buffa2012priori,devore2013greedy}. We highlight that thanks to the latter (in particular \cite[Theorem 3.5]{MR2821591} and \cite[Theorem 3.2]{devore2013greedy}) and \cref{prop:kolmogorov}, we obtain that the RB space constructed with the Greedy algorithm will yield an exponentially converging approximation.

%%%%%%%%%%%%%%%%%%%%%%%%%%%%%%%%%%%%%%%%%%%%%%%%%%%%%%
%%%%%%%%%%%%%%%%%%%%%%%%%%%
%%%%%%%%%%%%%%%%%%%%%%%%%%%
%%%%%%%%%%%%%%%%%%%%%%%%%%%
%%%%%%%%%%%%%%%%%%%%%%%%%%%

\section{Goal-oriented parametric model order reduction}
\label{sec:parametric reduction}
In this section, we construct $\Urb\subset\Uh$ by taking into account the variation of $\paramu\in\Pspace$ such that the RB approximations can be used in, e.g., seismic tomography problems where the earth model is for example updated as a result of an outer optimization procedure (cf., \cite[sec. 7.6]{hawkins2023model}) or the development of monitoring tools of seismicity using simulation data. 
Inspired by the time domain POD-Greedy algorithm \cite{haasdonk2008reduced}, we present a frequency domain POD-Greedy algorithm and a Greedy algorithm that are driven by the a posteriori error estimator \cref{eq:OoI estimate TD} for the time domain RB approximation error in the seismograms. Our approach is novel in that we use low-frequency restriction to ensure rapid convergence of the algorithms while not computing any time domain solutions. 

\subsection{POD-Greedy algorithm}\label{sec:PODGreedy_method}
The objective of the POD-Greedy algorithm is to iteratively construct $\Urb:={\rm span}\{\xi_1(\paramu_1),\ldots,\xi_k(\paramu_k)\}$ by combining the Greedy algorithm parameter search in the $\paramu-$parameter with the POD method for the $s-$parameter.
In a generic $k^{\text{th}}$ iteration, this is realized by identifying $\paramu_{k+1}$ from a dense training set $\Xi_m\subset \Pspace$ of cardinality $N_p$ as 
\begin{equation}
    \paramu_{k+1}:= {\text{arg}\max}_{\paramu\in\Xi_m}\; \Delta_{k,t}(\paramu),
\end{equation}
where $\Delta_{k,t}(\paramu)$ is defined in \cref{eq:OoI estimate TD}. Next, we enrich $\U_k$ by $\xi_{k+1}(\paramu_{k+1})$, which is the first POD basis function obtained by applying the POD method to
\begin{align}
   e_h(\hat{s};\paramu_{k+1}):= u_h(\hat{s};\paramu_{k+1}) - \sum_{l=1}^{k}\Big(u_h(\hat{s};\paramu_{k+1}),\xi_k(\paramu_l)\Big)_{\U}\xi_k(\paramu_l)\label{eq:POD-Greedy projection error}
\end{align}
 for $\hat{s}\in \tilde{I}_w:=\{\hat{s}_j\in I_w: 0\leq j\leq N_z-1\}$. 
The procedure terminates when $\Delta_{k,t}(\paramu_{k+1})$ is below the required tolerance or the maximum number of RB functions is reached.

\begin{remark}[Comparison with existing approaches] The analogy of the POD-Greedy algorithm above is the time domain POD-Greedy algorithm \cite{haasdonk2008reduced}, which combines the time domain POD with a Greedy algorithm over the parameter set. Compared to the time-dependent variants \cite{glas2020reduced,haasdonk2008reduced}, 
we avoid computing any time domain solutions and only require frequency domain solutions for a few low frequencies to ensure rapid convergence.
\end{remark}

%%%%
%%%%
%%%%
\subsection{Greedy algorithm}
In this subsection, we adapt the Greedy algorithm to iteratively construct an RB space $\Urb:=\text{span}\{u_h(\hat{s}_1;\paramu_1),\ldots,u_h(\hat{s}_k,\paramu_k) \}$ corresponding to the set of parameters $M_k:=\{\paramu_1,\ldots,\paramu_k\}$ and $S_k:=\{\hat{s}_1,\ldots,\hat{s}_k\}$. In contrast to the POD-Greedy algorithm in 
\cref{sec:PODGreedy_method}, we replace the application of the POD method in \cref{eq:POD-Greedy projection error} with a Greedy algorithm search to identify $S_k$ from $\tilde{I}_w$.
 
    That is, in a generic $k^{\text{th}}$ iteration, we identify $\paramu_{k+1}\in\Xi_m$ and $\hat{s}_{k+1}\in \tilde{I}_w$ as
\begin{equation}
    \paramu_{k+1}:= {\text{arg}\max}_{\paramu\in\Xi_m}\; \Delta_{k,t}(\paramu)\; {\rm and }\; \hat{s}_{k+1}:= {\text{arg}\max}_{\hat{s}\in \tilde{I}_w}\; \Delta_{k,f}(\hat{s};\paramu_{k+1}),\label{eq:Greedy_m_stop}
\end{equation}
and enrich the RB space $\Urb$ with the full order solution $u_h(\hat{s}_{k+1};\paramu_{k+1})$ of \cref{eq: Weak formulation discrete}. We terminate the algorithm when $\Delta_{k,t}(\paramu_{k+1})$ is below the required tolerance or the maximum number of RB functions is reached. 
%%%%%%%%%%%%%%%%%%%%%%%%%%%%%%%%%%%%%%%%%%%%%%%%%%%%%%
%%%%%%%%%%%%%%%%%%%%%%%%%%%
%%%%%%%%%%%%%%%%%%%%%%%%%%%
%%%%%%%%%%%%%%%%%%%%%%%%%%%
%%%%%%%%%%%%%%%%%%%%%%%%%%%

\section{Numerical experiments}\label{sec:Numerical_results}
In this section, we demonstrate the excellent approximation properties of the RB methods proposed in this paper for a realistic earth model of Groningen, the Netherlands, which is described in \cref{sec:Model problem numerical results}. 
In \cref{sec:reduction in the s parameter results} we fix $\paramu$ and demonstrate how the restriction to low frequencies helps to ensure exponential and rapid convergence of the RB approximation errors, where the maximum frequency given by $ s_{max}$ is controlled by the central width of the Ricker wavelet $\mathcal{q}(t)$ (see e.g., \cref{fig:Ricker_wavelet_TD_FD_sol}).  
Furthermore, we highlight the computational advantages (in light of the available computational resources) by demonstrating the RB approximation of the seismograms in a wall-clock time faster than the implicit Newmark-beta method. 
 In \cref{sec:results_lameparameters}, we perform reduction to accommodate changes in the Lam\'e parameters of the considered subsurface model. 
 We consider practically relevant examples, inspired by the seismological applications mentioned in \cref{sec:seismo_application}, to construct the training and test sets.  
We show that the POD-Greedy and Greedy algorithms inherit the good approximation properties of the methods proposed in \cref{sec:MOR for NLI}, which along with our a posteriori error estimator aids in the rapid convergence of the algorithms.
Furthermore, we show that the effectivity of the proposed error estimator, utilized in the POD-Greedy and Greedy algorithms, remains approximately constant for an increasing number of RB functions.

\subsection{Model problem and numerical setup}\label{sec:Model problem numerical results}
\begin{figure}[t]
   \centering
   \subfloat{\includegraphics[width=.4\linewidth]{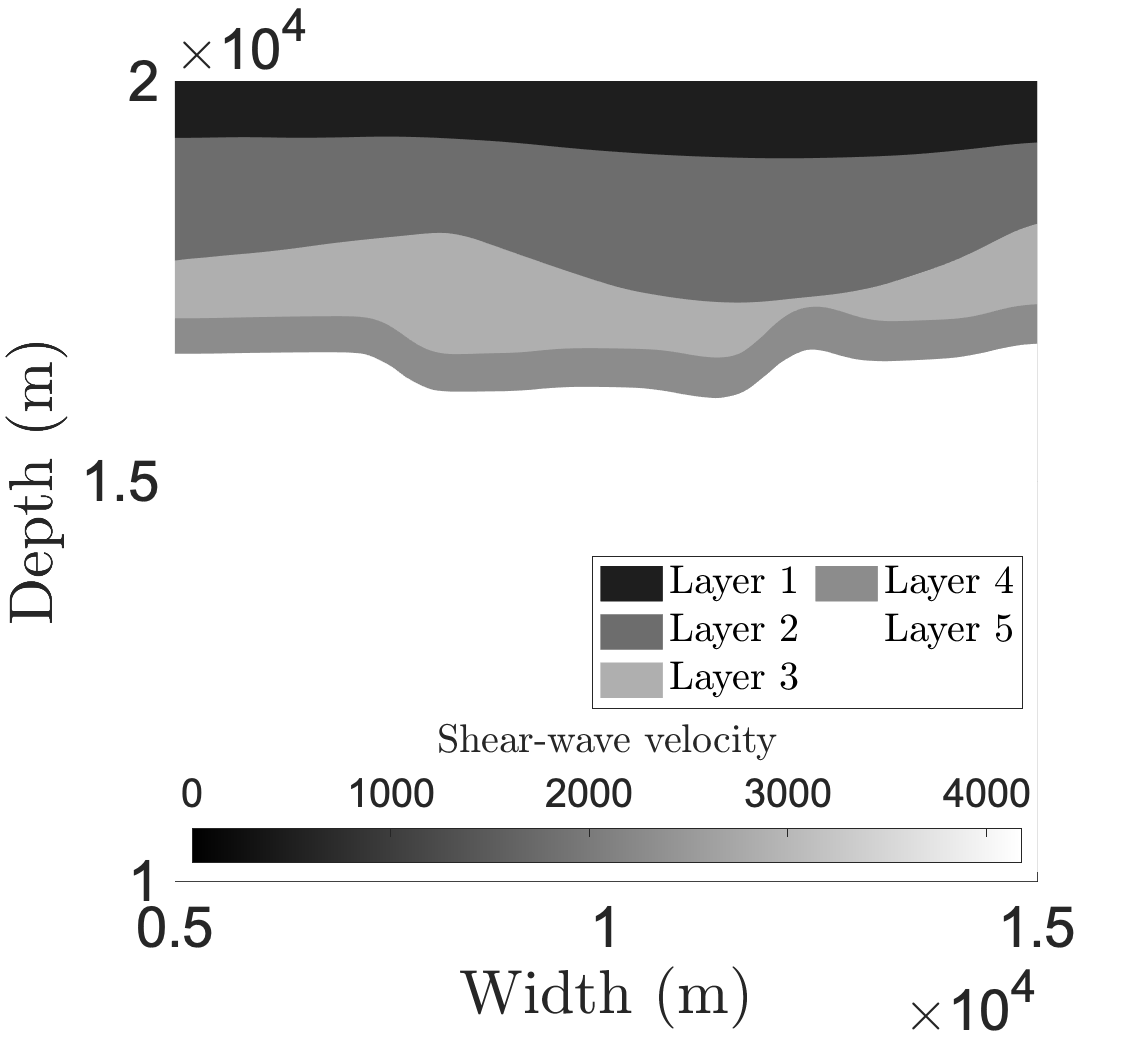}}
   \subfloat{\includegraphics[width=.5\linewidth,trim={0cm -4cm 0 0},clip]{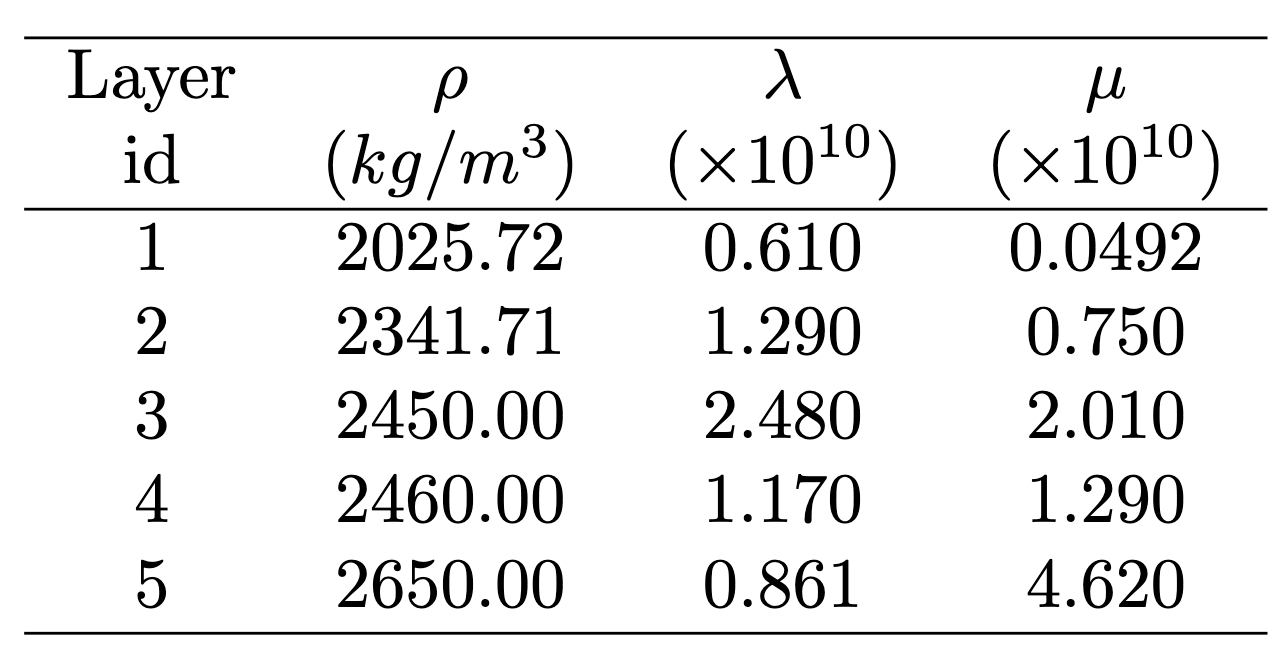}}
   \caption{A 2D semi-realistic model problem of Groningen, the Netherlands \cite{Kruiver2017A, Romijn2017A}, where the source location is assumed to be in layer 4, and receiver locations are kept at the top boundary.}
   \label{fig:model}
\end{figure}

We consider a 2D layered subsurface model of Groningen, the Netherlands, from \cite{Kruiver2017A, Romijn2017A}. The model exhibits pronounced changes in the elastic properties between the layers, as illustrated through shading in \cref{fig:model}, representing the shear-wave velocities. Notably, we observe a small shear-wave velocity near the subsurface. 
To minimize the influence of reflections, we extend the model domain by $5000 \,$meters horizontally in both directions and $10^{4}\,$meters vertically on the lower section of the subsurface. Moreover, we extend the interfaces within the model using straight lines. We use GMSH \cite{geuzaine2009gmsh} to obtain a quasi-uniform triangulated mesh of the model and the assembly routine as in \cite{koko2016} for the discrete problem. The associated conforming finite element space $\Uh$ of continuous piecewise linear functions has dimension $N_h=233,318$. 
The parameter values for the source introduced in \cref{eq:Sources} are set to $\sigma_m=80$, and $k_a=160$, and the source is located in layer 4 at $x_0=(9915.8\,m,1641.7\,m)$. Moreover, we assume that the receiver is positioned at the surface at $( 9475.2\,m, 20000\,m)$, unless stated otherwise.
For time approximations, we choose $N_z=608$ to define the contour as in \cref{eq:a_p truncated}. We obtain reference time domain solutions for comparisons using the implicit Newmark-beta method \cite{newmark1959method} with time-step $10^{-3}$ for $t\in[0,20]$ yielding $20001$ time steps. Finally, the results presented have been obtained using MATLAB 2023b on a GLNXA64 machine with an AMD EPYC 7742 with 2$\times$64-Core processor and 1TB memory, and the source code is available online at {\cite{CODES}}. 

\subsection{Reduction for the $s-$parameter with \textit{m} fixed}\label{sec:reduction in the s parameter results}
In this subsection, we treat the real and imaginary parts of the frequency $s:=\reals+\imagI \imags$ as parameters, and use the material properties described in \cref{fig:model}. First, we discuss the influence of maximum frequency content on the dimension of the RB spaces. We construct a finite-dimensional training set $(\reals,\imags)\in \tilde{I}_s:=[0.12,0.62]\times [0.01, s_{max}]$ of cardinality 1024 using a uniform random sampling.
 We recall our observation from \cref{fig:Ricker_wavelet_TD_FD_sol}, where we noted that increasing $\alpha$ leads to a steeper gradient of the time domain source function which was reflected in a wider frequency spectrum for the decay of $\|u(\reals+\imagI \imags)\|_{\U}$ with increasing $\imags$.
We consider $\alpha$ from the set $\{1.0\pi,1.5\pi,2.0\pi\}$, and find $ s_{max}$ for each choice of $\alpha$ as $\{11.75,17.66,23.24\}$ respectively, where $ s_{max}$ is the smallest value to satisfy $\kappa(\reals + \imagI  s_{max})\leq 10^{-4}$; see \cref{remark:choice_smax}.
 Here, we have chosen $\reals=0.35$ and note that similar values for $s_{max}$ are obtained for arbitrary $\reals\in [0.12,0.62]$, which might be explained by the exponential decay of $|q(\reals+\imagI \imags)|$ in the considered frequency region. 
 In our figures, we use a semi-logarithmic scale, and the notation $[\cdot]_{0}=\cdot/{\max_{s\in \tilde{I}_s}\|z_k(s)\|_2}$ for the estimator and  the true error. 

 \begin{figure}[t!]
   \centering
   \subfloat[]{\label{subfig:singular_valuedecay}\includegraphics[width=.33\linewidth]{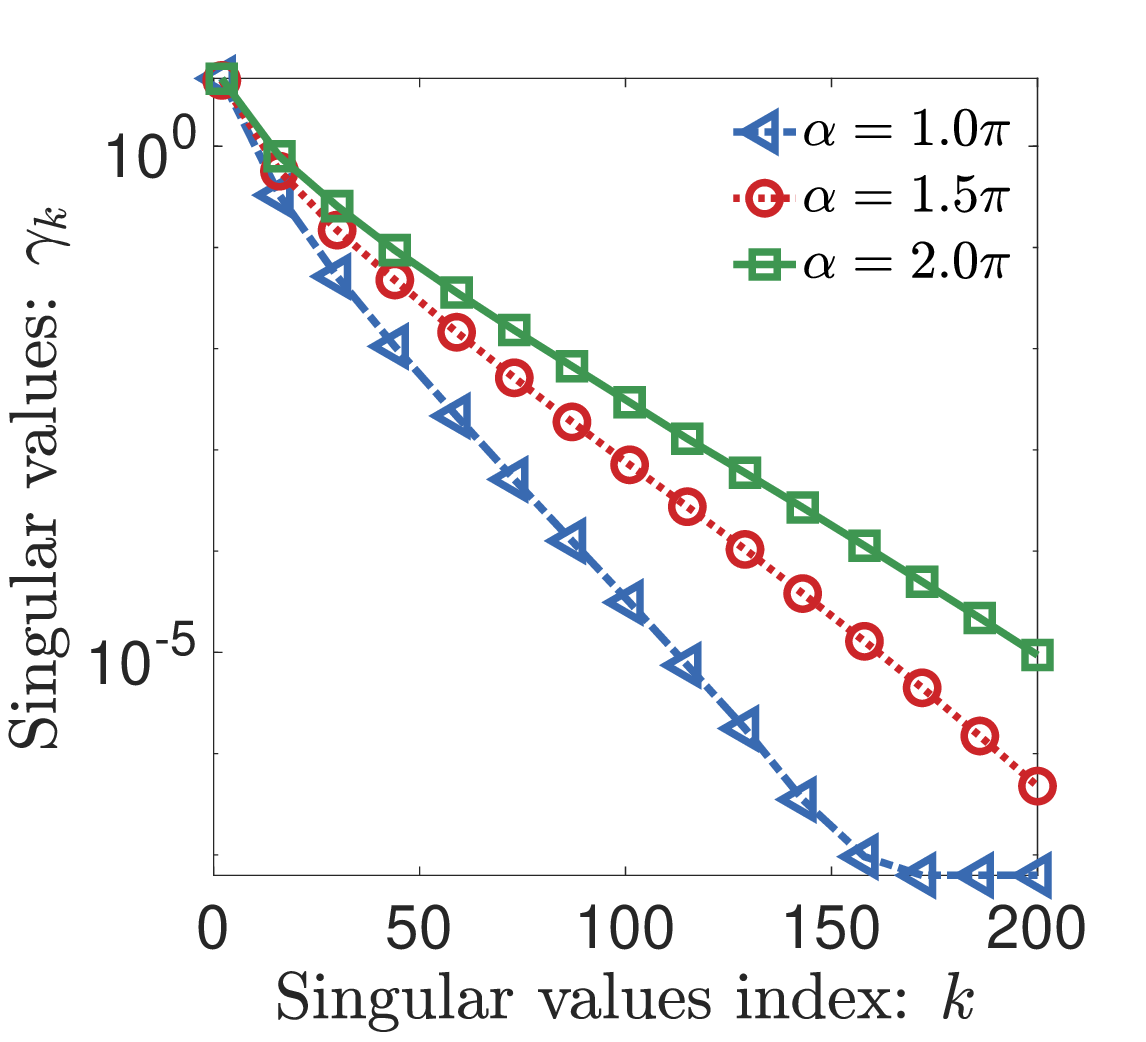}}
   \subfloat[]{\label{subfig:Greedy_fixedm_Est}\includegraphics[width=.33\linewidth]{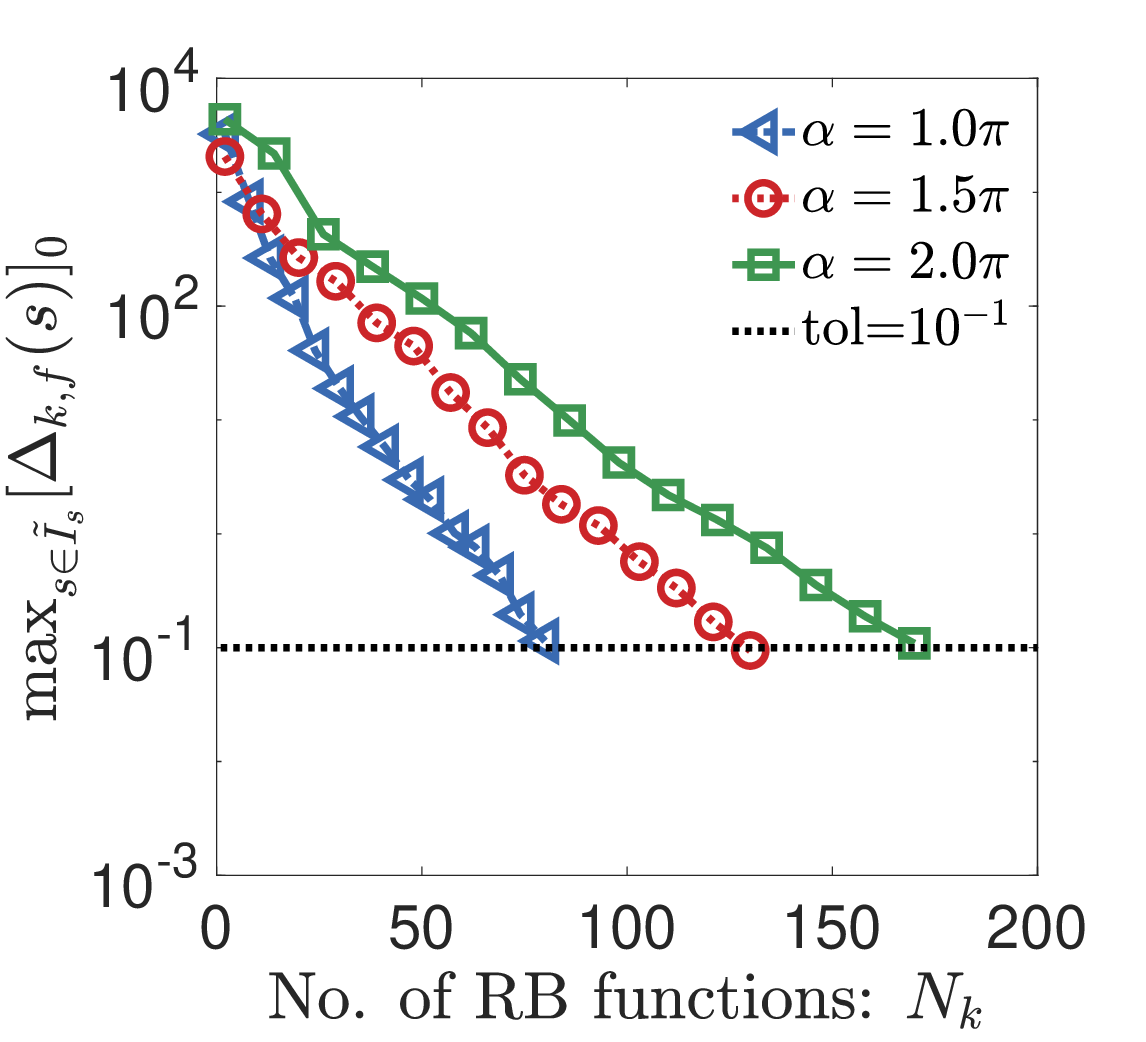}}
   \subfloat[]{\label{subfig:Greedy True_fixedm}\includegraphics[width=.33\linewidth]{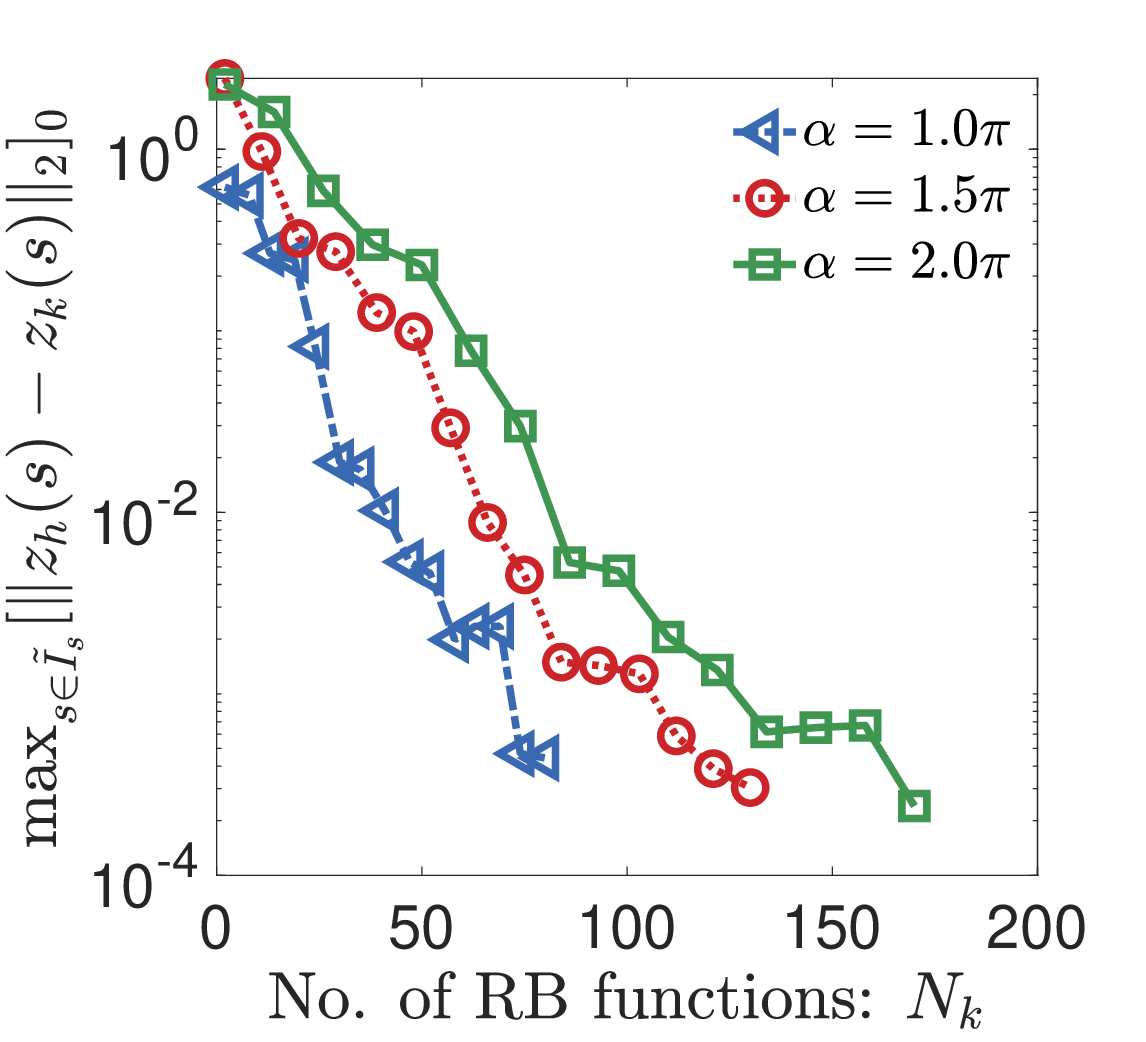}} \caption{In (a), we show the singular value decay for the POD method. In (b) and (c), we plot the convergence of the Greedy algorithm for the error estimator and the true errors, respectively.}
   \label{fig:Fixed_m_offlinephase}
 \end{figure}

In \cref{subfig:singular_valuedecay,subfig:Greedy_fixedm_Est}, we observe an exponential decay of the singular values for the POD method and the error estimator during the Greedy algorithm. Moreover, in \cref{subfig:Greedy True_fixedm}, we observe that the convergence of the true error in the frequency domain seismogram for the Greedy algorithm is very similar to the error estimator. However, as we increase $\alpha$, both the decay of the singular values and the convergence of the Greedy algorithm becomes relatively slower, as expected, see \cref{eq:n_width_kolmo}. Furthermore, we emphasize that we would potentially require a large number of RB functions if the desired frequency spectrum is further increased. Therefore, we have used the uniform lower-bound for the discrete inf-sup constant to avoid misleading selection of $s\in\tilde{I}_s$, which can happen if the linear decay in the discrete inf-sup constant starts dominating in the error estimator for higher $\alpha$ (cf., \cref{eq:effectivity}).  

 In \cref{subfig:Meanerror_CPOD,subfig:Meanerror_Greedy}, we observe a similar behavior of the error in the frequency domain RB approximations of the output over a uniform random test set $\tilde{I}_t$ of cardinality 512. Moreover, we note that the ROMs obtained from the POD method and the Greedy algorithm converge approximately at the same rate. Here, targeting the construction of the RB space to the seismograms does not facilitate any further reduction. We conjecture that the correlation between the seismograms and the solutions is stronger than the symmetries present in the solution. Therefore, the RB functions that contribute most to the seismograms are often also dominant in the solutions, which results in a very similar performance for both methods. 

 \begin{figure}[t!]
   \centering
   \subfloat[POD method]{\label{subfig:Meanerror_CPOD}\includegraphics[width=.33\linewidth,trim={1.2cm 0cm 16.3cm 0.0cm},clip]{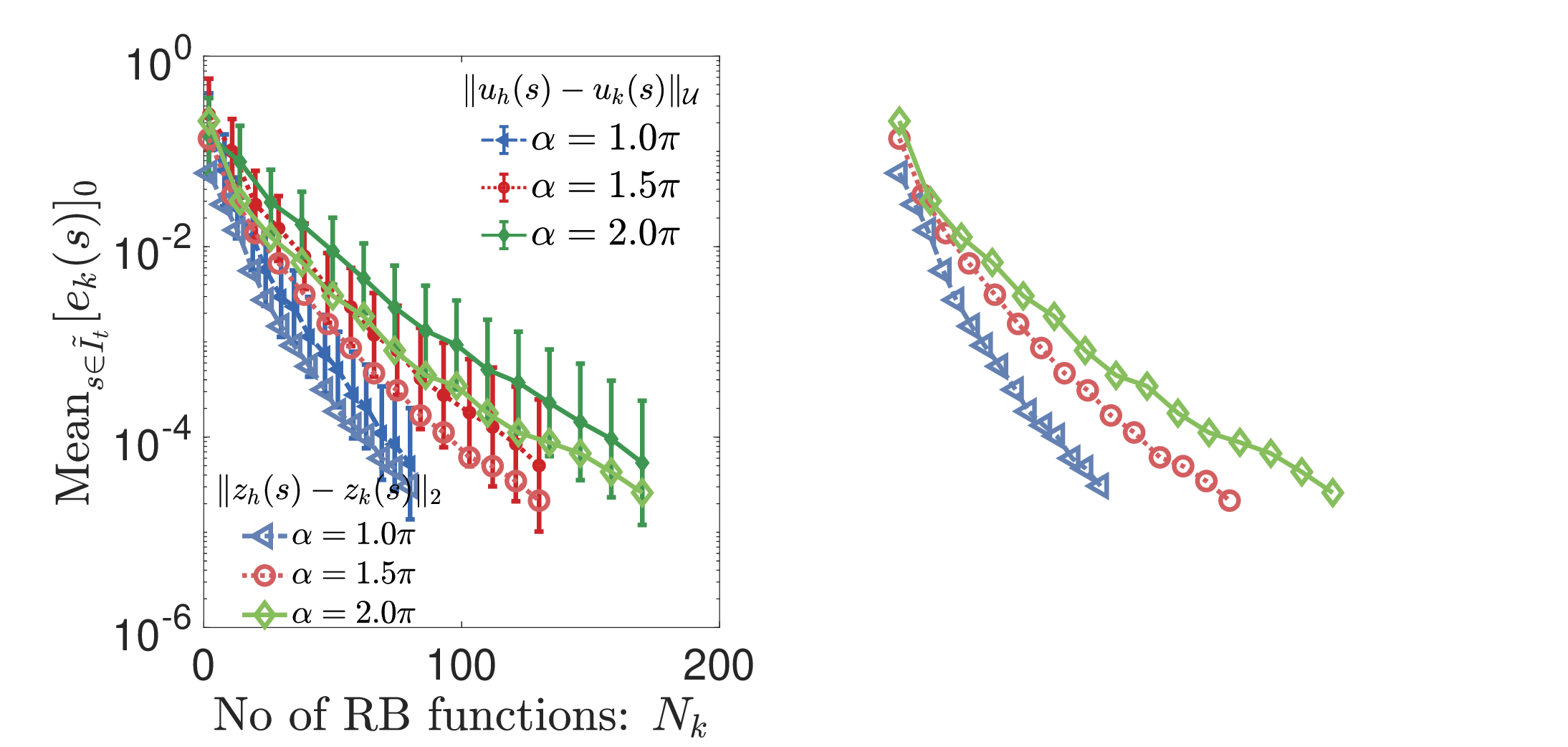}}
   \subfloat[Greedy algorithm]{\label{subfig:Meanerror_Greedy}\includegraphics[width=.33\linewidth,trim={1.2cm 0cm 16.3cm 0.0cm},clip]{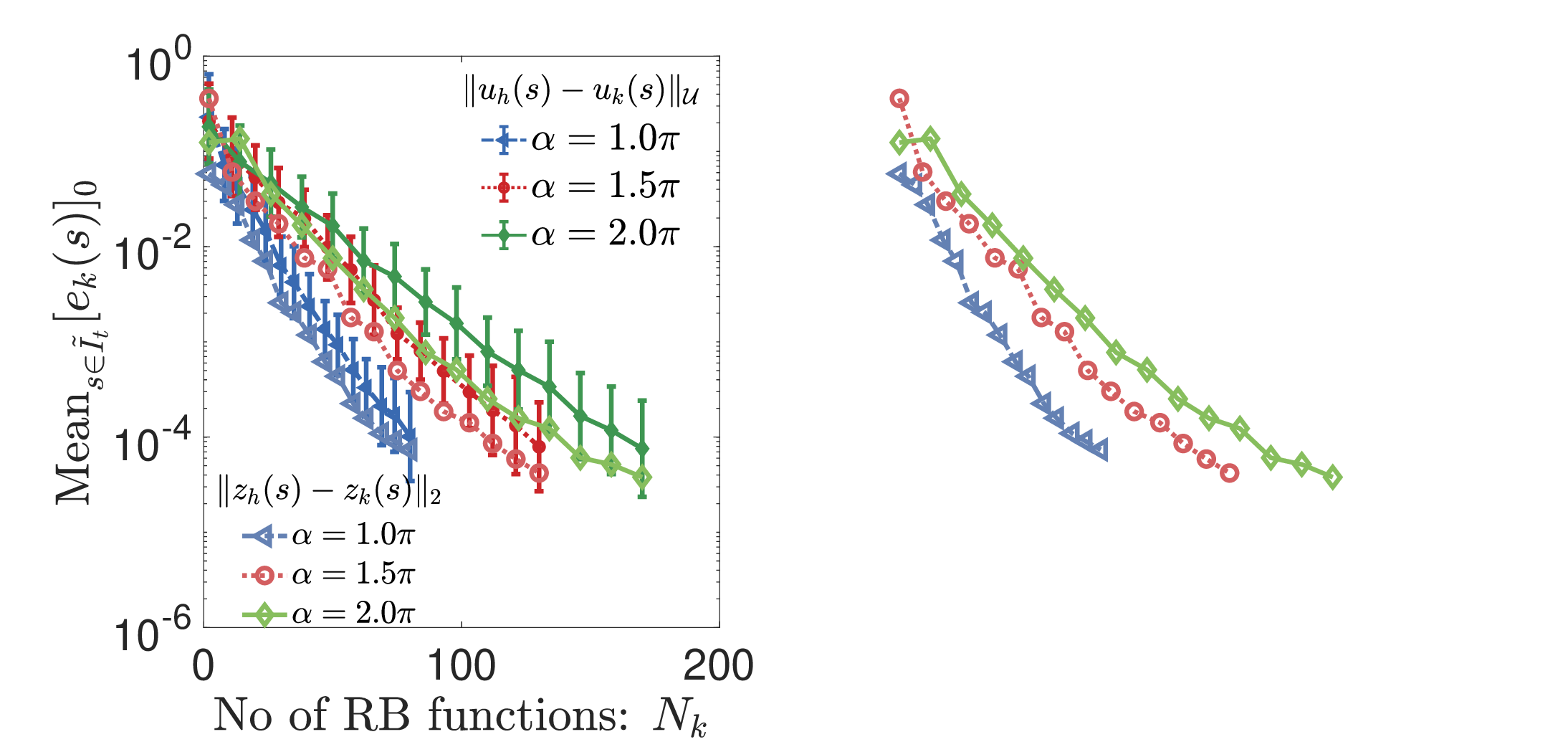}}
   \subfloat[Time domain]%
   {\label{subfig:Meanerror_TD}\includegraphics[width=.33\linewidth,trim={1.2cm 0cm 16.3cm 1.2cm},clip]{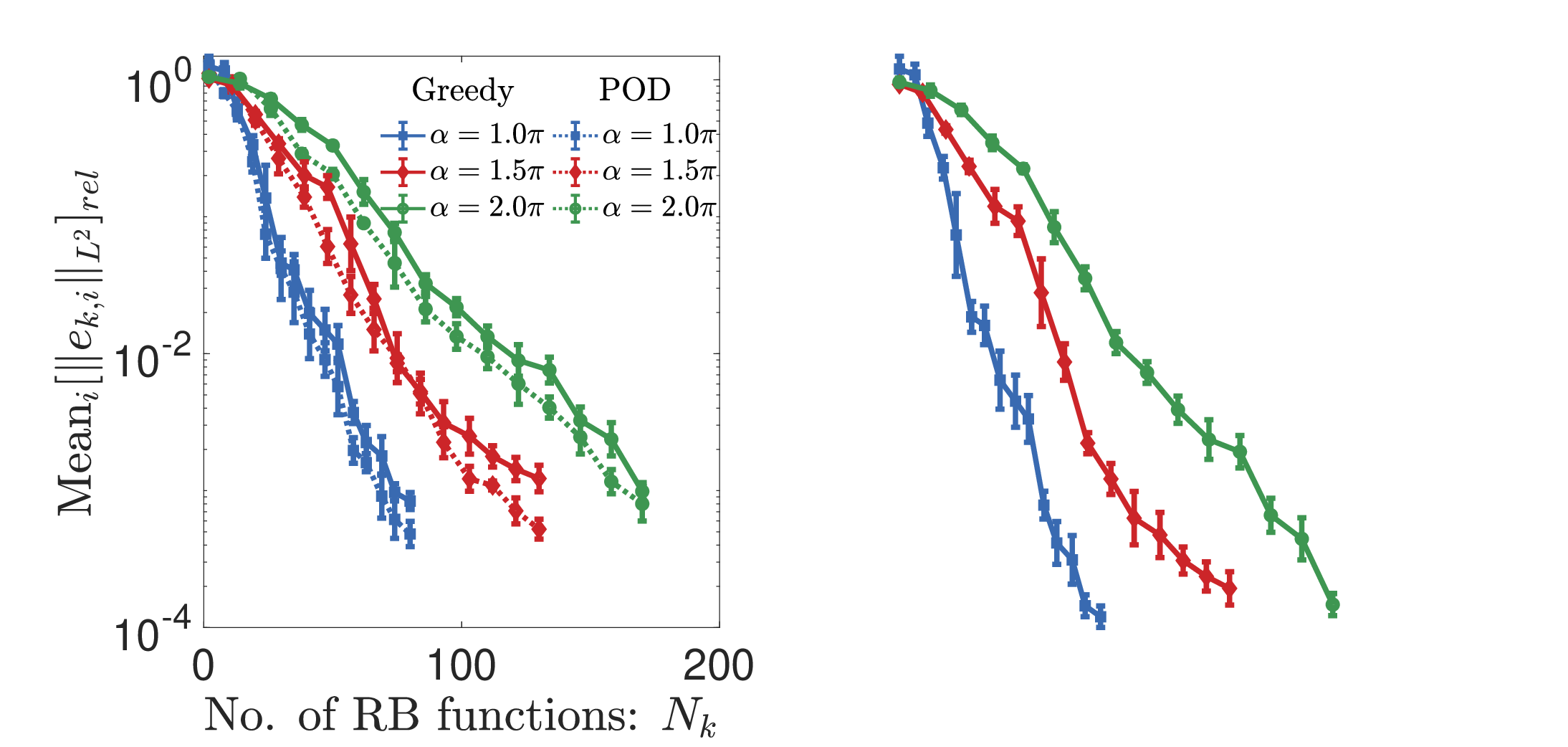}}
   \caption{In (a) and (b), we plot the frequency domain mean error of the RB approximations in the full domain and seismograms over a test set for the POD and the Greedy algorithm, respectively. In (c), we plot the time domain mean error $[\|e_{k,i}\|_{L^2}]_{rel}={\|\OoI_i - \hat{\OoI}_{k,i}\|_{L^2(0,T;\mathds{R}^2)}}/{\|\OoI_i\|_{L^2(0,T;\mathds{R}^2)}}$ over a set $I_r$ containing 5 uniformly placed receiver locations at the surface. The length of the error bars is the standard deviation.}
   \label{fig:Fixedm_online}
 \end{figure}

 Next, the free parameters required to define the contour $\hat{s}$ as in \cref{eq:a_p truncated}, are calculated for $\alpha=2\pi$ using a POD-based ROM of dimension $N_k=176$. Here, we have benefited from the guaranteed accuracy of RB approximations observed in \cref{subfig:Meanerror_CPOD} to calculate $(\wr,\wi)=(0.26,15.20)$. We chose the most dominant case of $\alpha$ for the selection of free parameters, which allows to reuse them for all other values of $\alpha$ as noted below. In practice, the free parameters are computed for the relevant problem of interest. {\color{black}To calculate $(\wr,\wi)$, we have used the code from \cite{Brio2005ApplicationExponential} available under the BSD 3-Clause license. The code uses MATLAB's function \texttt{fminbnd} to compute $(\wr,\wi)$ for an arbitrary Laplace transformed function. We have adapted the code to our problem by replacing the computation of the expansion coefficients for general functions with a ROM of the elastic wave equation. The rest of the algorithm has been left unchanged, and converges in a few iterations.}

The RB constructed earlier retains its good approximation properties in the time domain. In \cref{subfig:Meanerror_TD}, we plot the relative $L^2([0,20];\mathds{R}^2)$ norm errors for time domain RB approximation of seismograms over increasing dimension of ROMs. We note that the RB approximation error converges at approximately the same rate as observed during the construction of RB for different choices of $\alpha$. Furthermore, in \cref{fig:Seismo_fixedm}, we plot the horizontal and vertical component of seismograms approximated using ROMs of sizes $N_k=76,126,167$ for $\alpha=1\pi,1.5\pi,2\pi$, respectively. Here, the errors are computed using time domain reference solutions obtained from the implicit Newmark-beta method. We observe a good approximation of the time domain seismograms for all cases, where we require more RB functions for relatively large values of $\alpha$ to guarantee an accuracy of $10^{-3}$ due to the highly oscillatory response noted in the seismograms. This observation aligns with our earlier discussion on frequency content, where we note that the general characteristics of the time domain seismograms stays the same for all $\alpha$, and each increment in $\alpha$ add fine scale details to the existing signal. 

\begin{figure}[t!]
   \centering
\subfloat{\label{subfig:Hcompalpha1pi}\includegraphics[width=.33\linewidth]{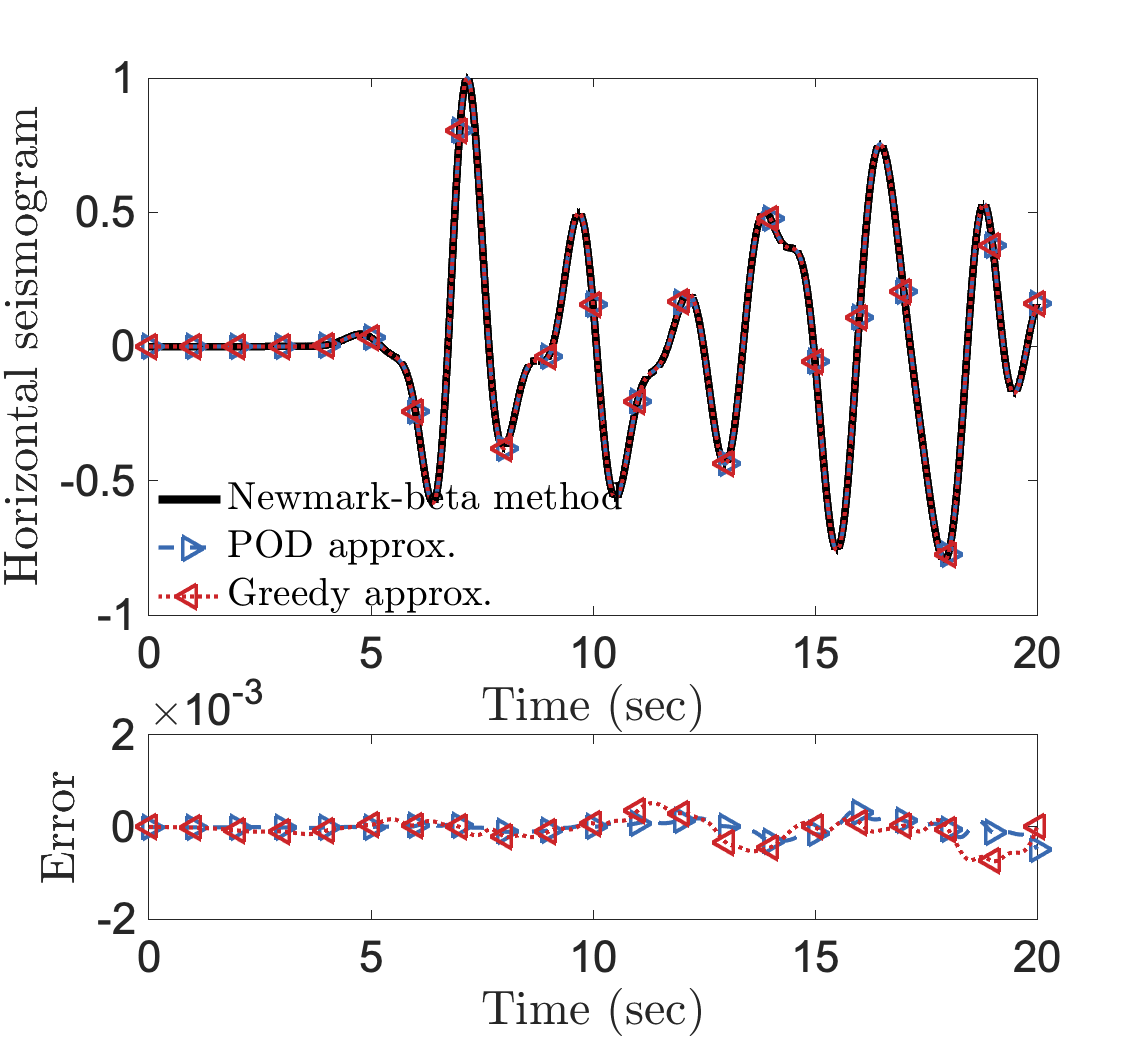}}   \subfloat{\label{subfig:Hcompalpha1_5pi}\includegraphics[width=.33\linewidth]{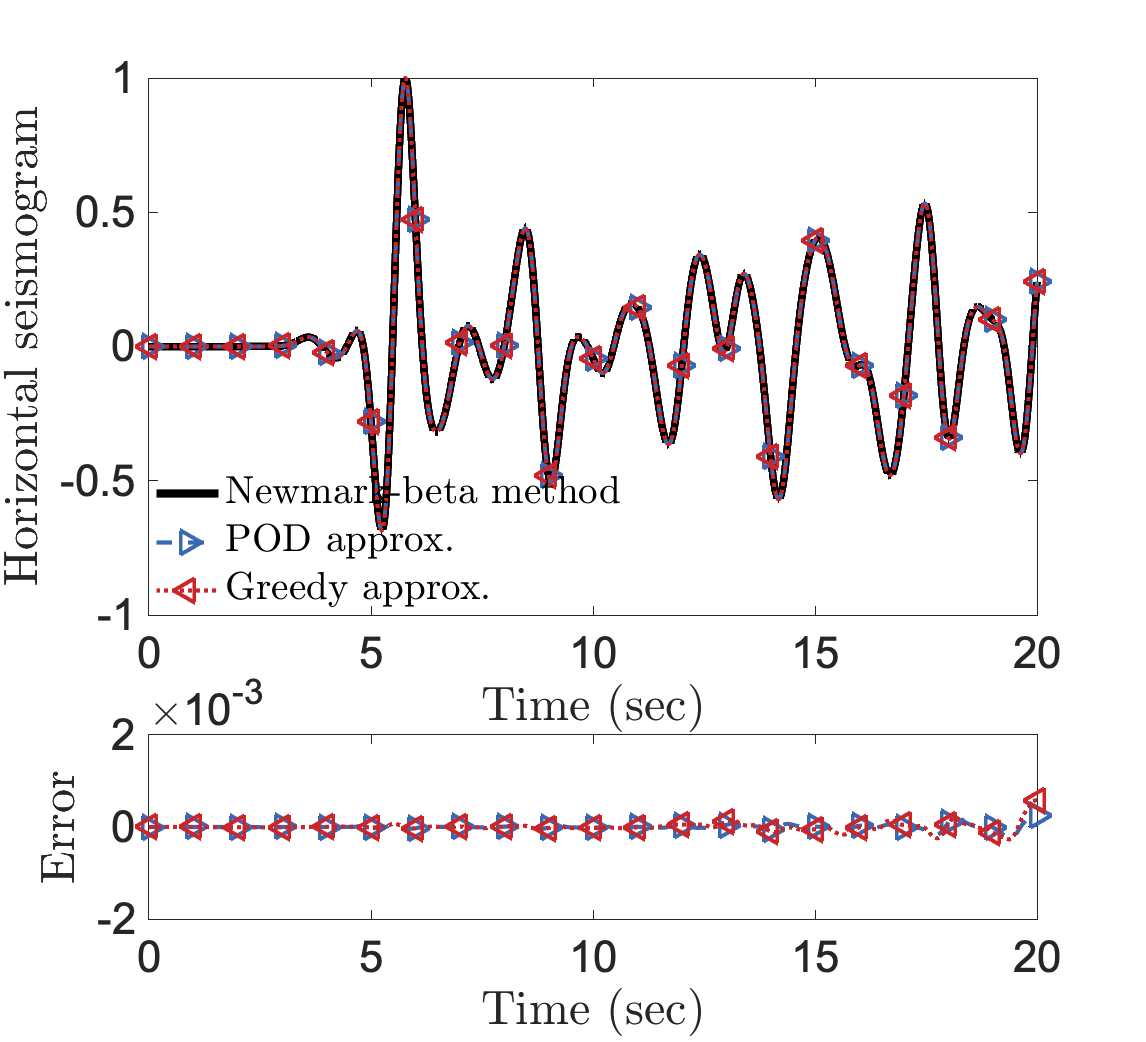}} \subfloat{\label{subfig:Hcompalpha2pi}\includegraphics[width=.33\linewidth]{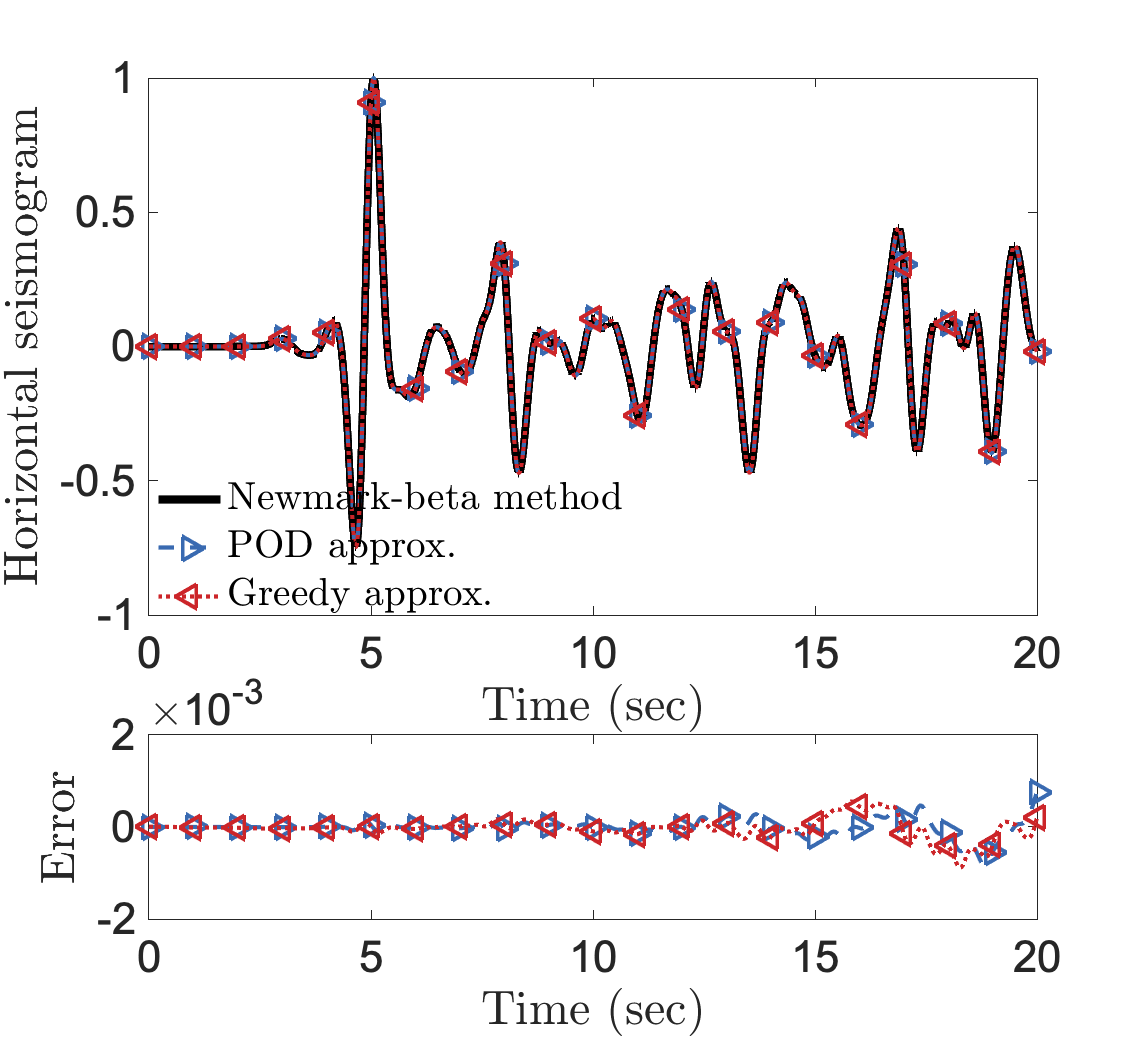}} \\ \vspace{-0.35cm}
\setcounter{subfigure}{0}
   \subfloat[$\alpha=1.0\pi$]{\label{subfig:Vcompalpha1pi}\includegraphics[width=.33\linewidth]{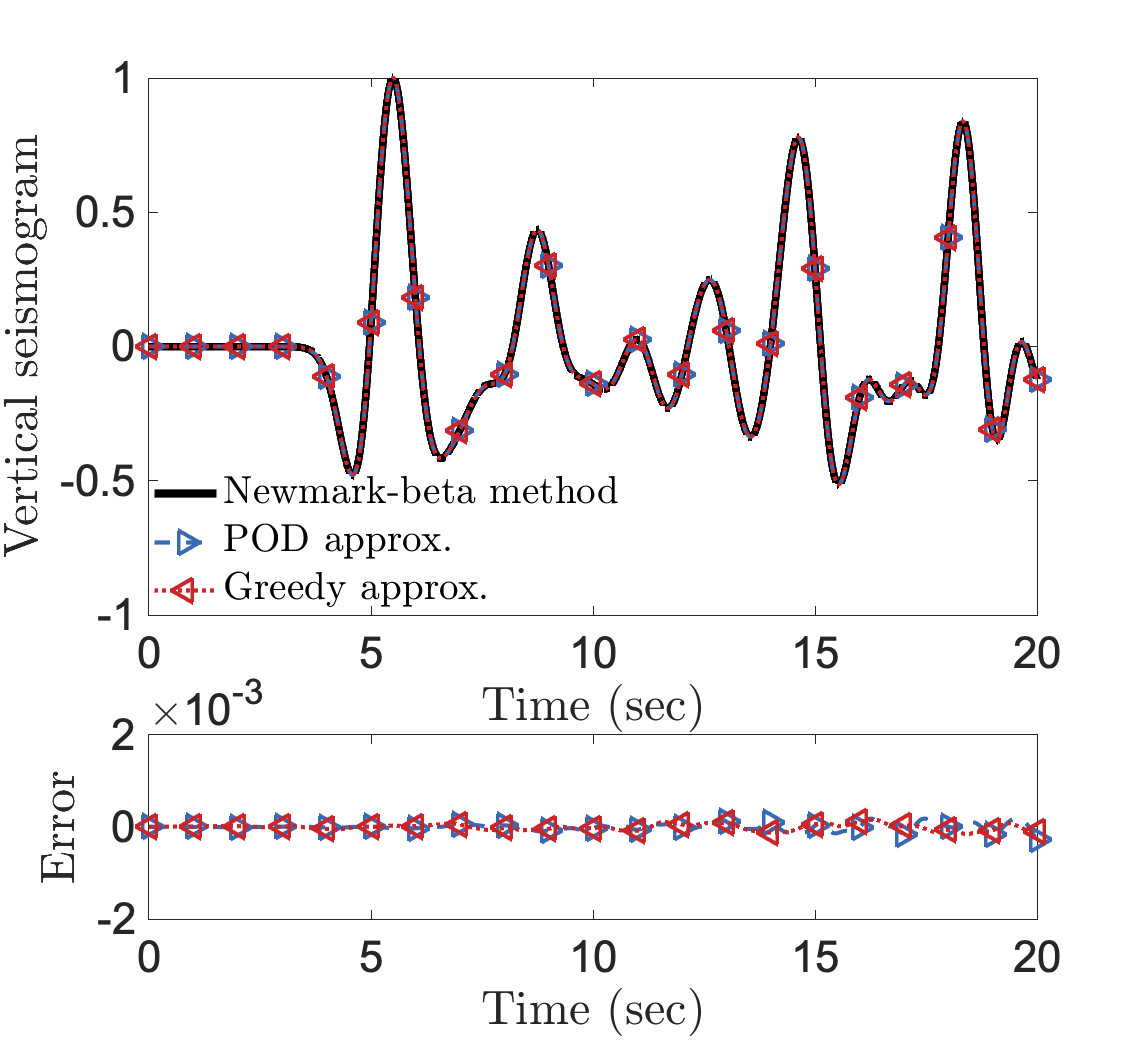}}
   \subfloat[$\alpha=1.5\pi$]{\label{subfig:Vcompalpha1_5pi}\includegraphics[width=.33\linewidth]{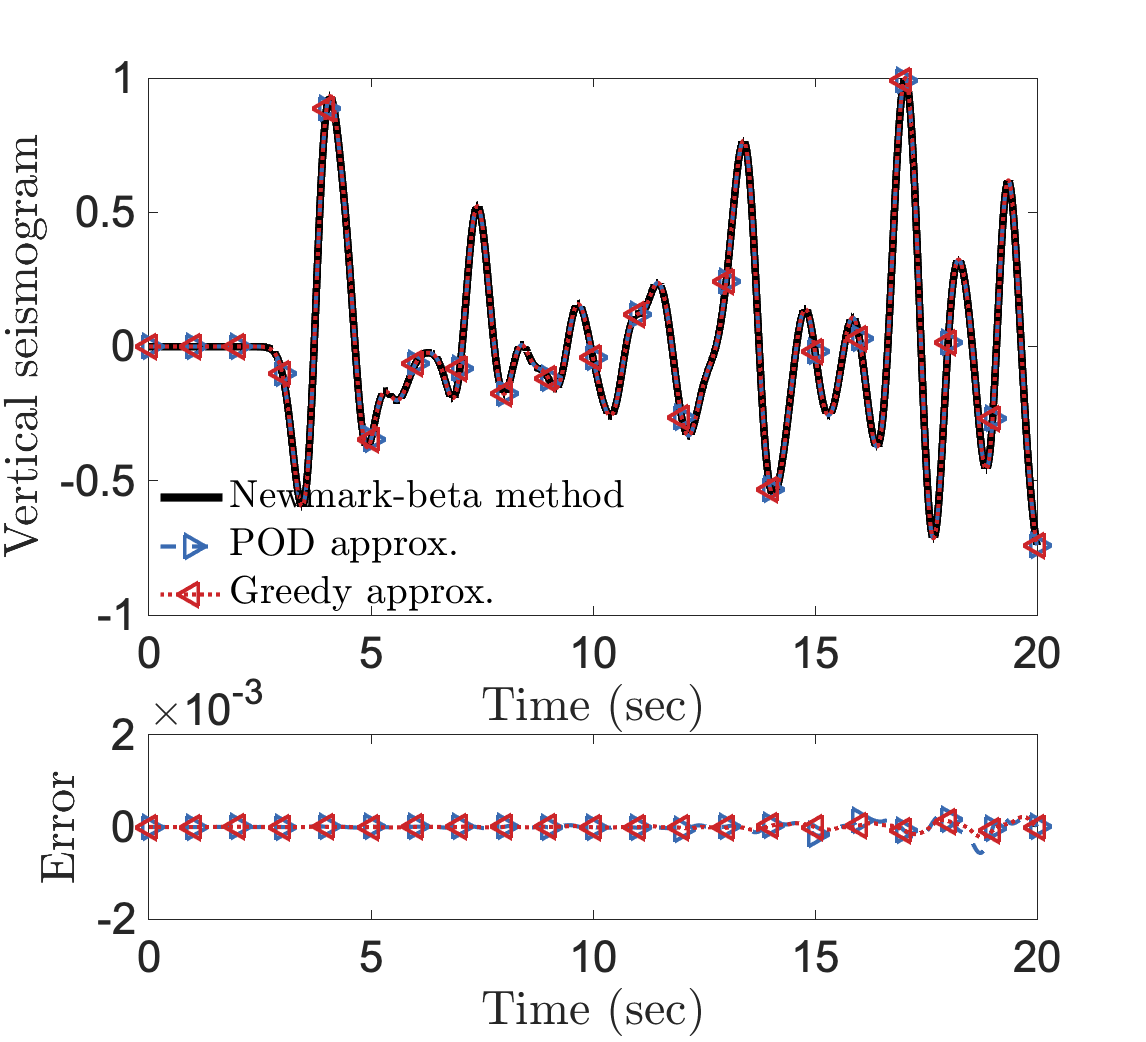}}
   \subfloat[$\alpha=2.0\pi$]{\label{subfig:Vcompalpha2pi}\includegraphics[width=.33\linewidth]{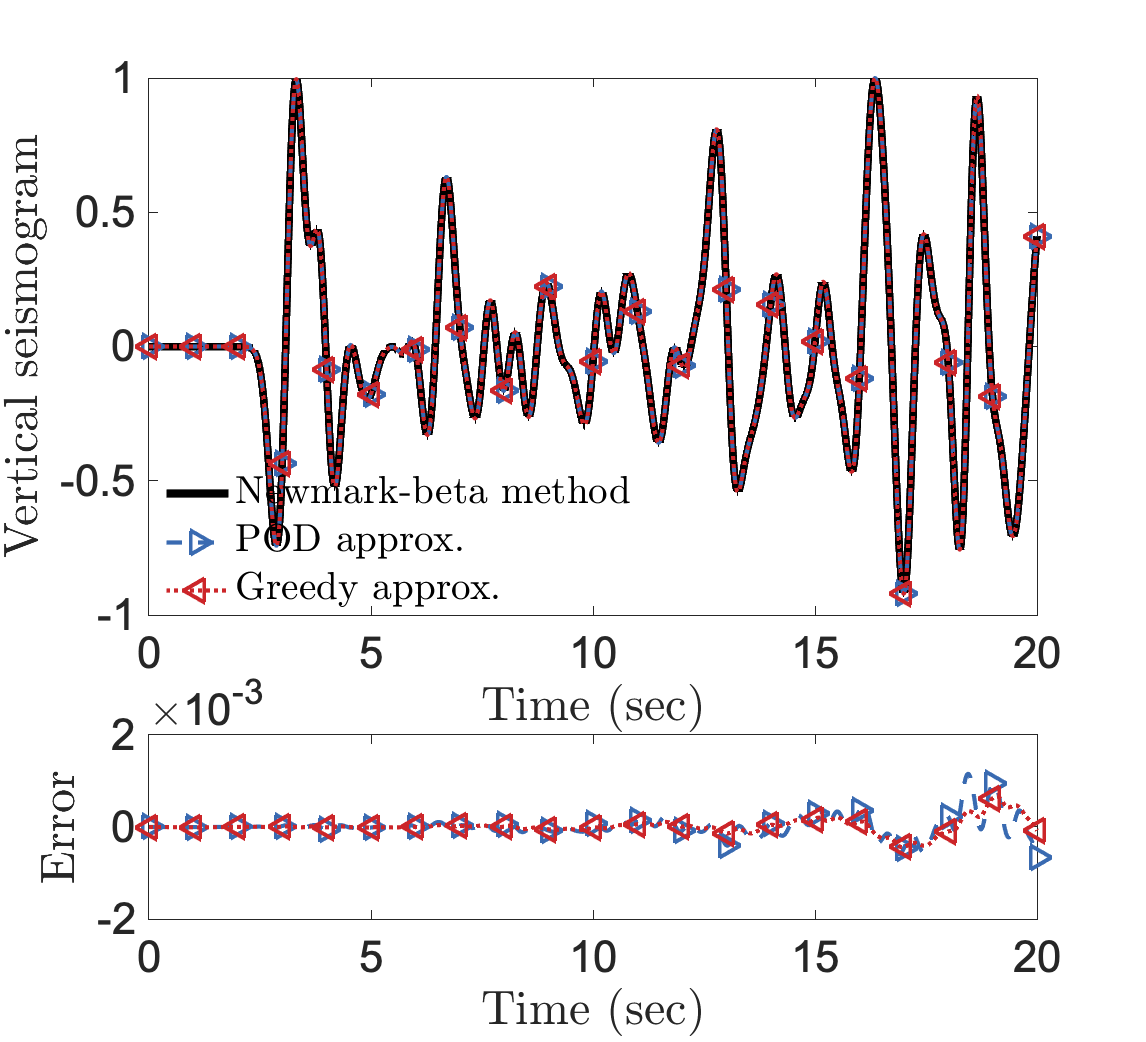}}
   \caption{Error comparison of the time domain horizontal and vertical components of the seismograms computed via ROMs and implicit Newmark-beta method. We scale the seismograms with the maximum absolute value.}
   \label{fig:Seismo_fixedm}
 \end{figure}
  
\begin{figure}[t!]
   \centering
   \subfloat[Offline: Construction]{\label{subfig:wall_clockOffline}\includegraphics[width=.33\linewidth]{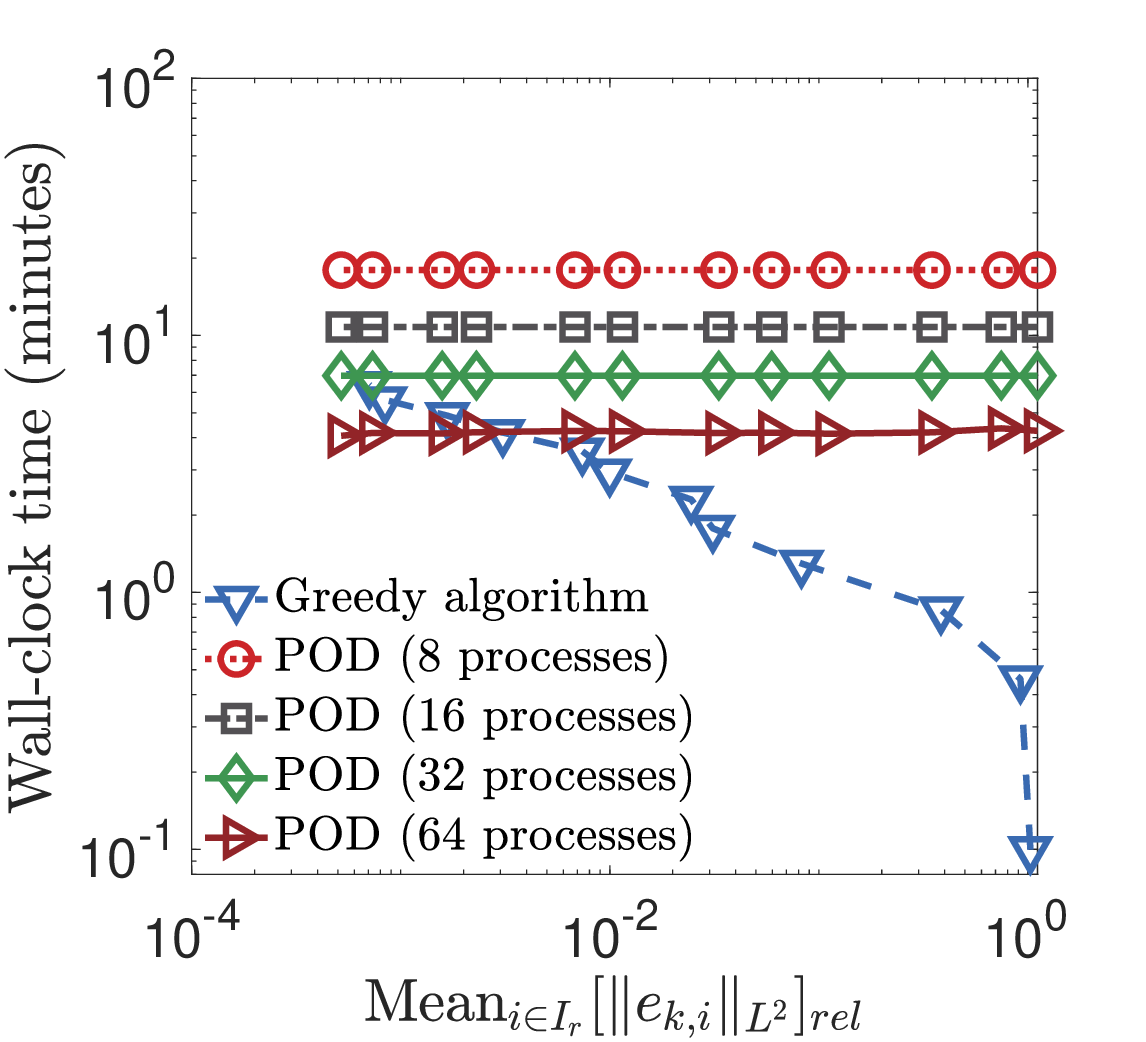}}
   \subfloat[Online: Increment $N_k$]{\label{subfig:wall_clockonline}\includegraphics[width=.33\linewidth]{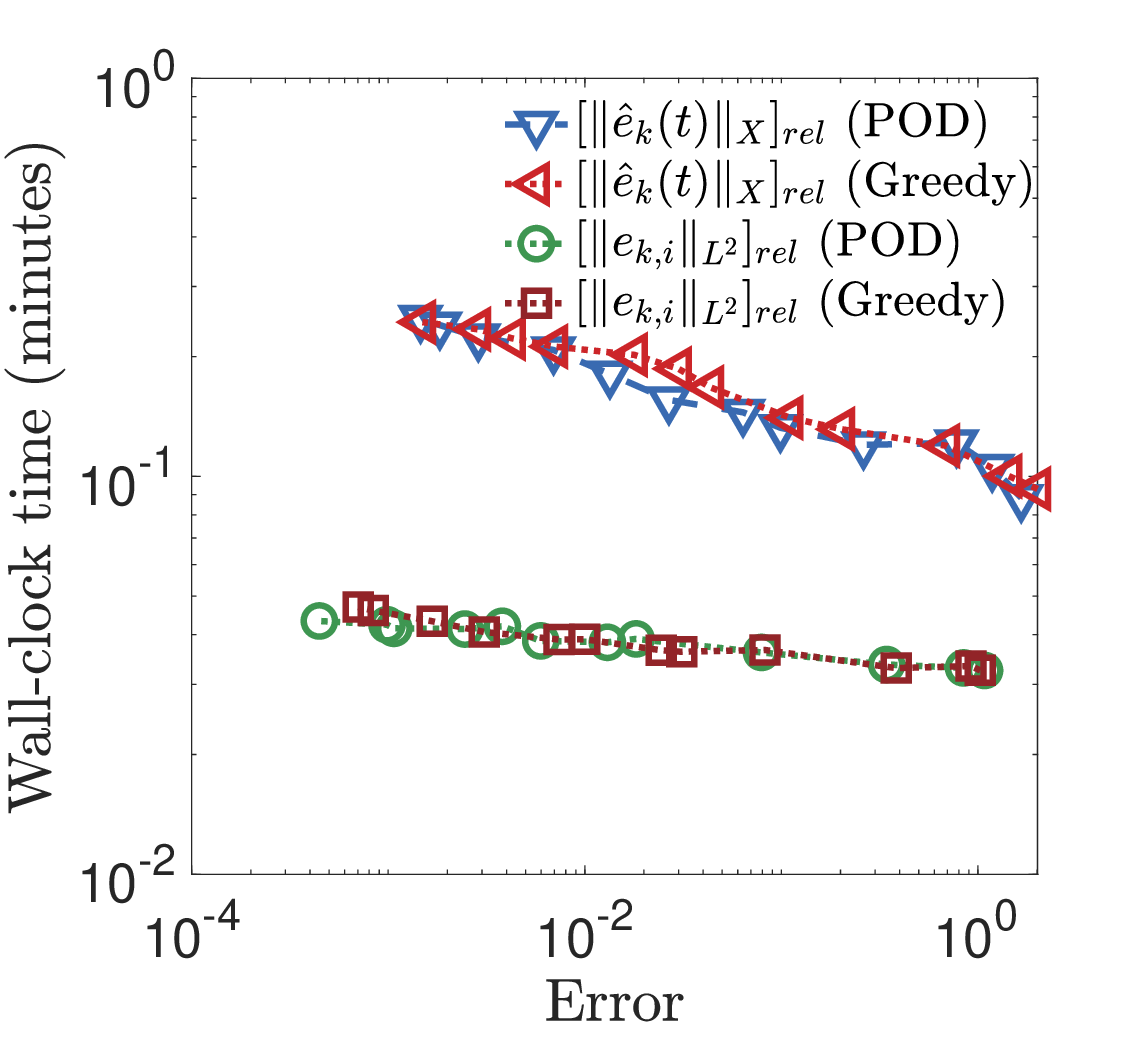}}
   \subfloat[Online: Increment time-step]{\label{subfig:wall_clockNBM_FOM_ROM}\includegraphics[width=.33\linewidth]{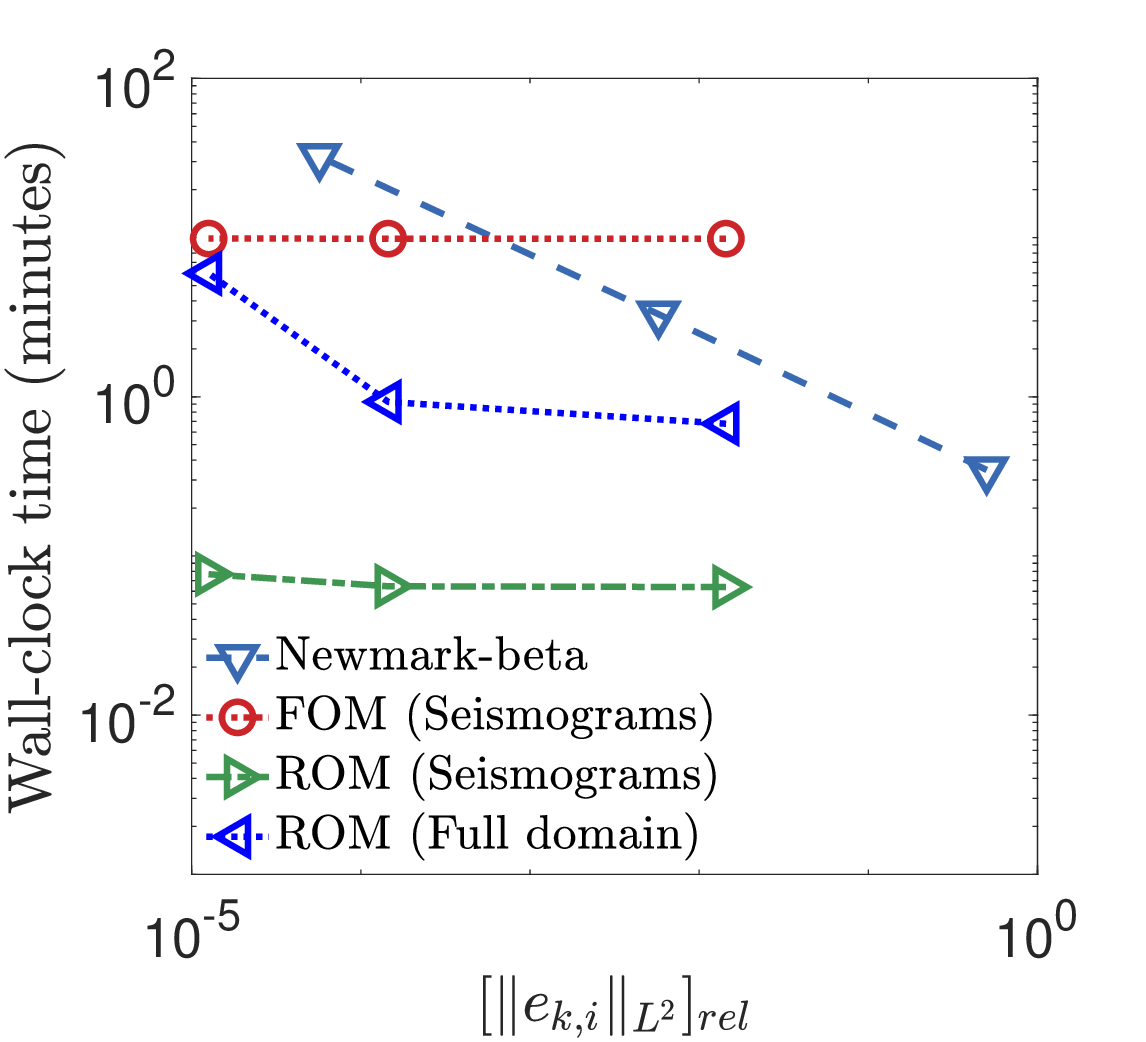}}
   \caption{In (a), we compare the offline time in a sequential and a multi-processor setting. In (b) and (c), we illustrate the respective error quantities and online wall-clock times to compute seismograms and full wavefield using a ROM or an implicit Newmark-beta method. We use the notation $[\|\hat{e}_{k}(t)\|_{X}]_{rel}={\|\utdc_h(t) - \hat{\utdc}_{k}(t)\|_{\U}}/{\|\utdc_h(t)\|_{\U}}$, where $\utdc_h(t)$ is the time domain wavefield computed using the implicit Newmark-beta method and $\hat{\utdc}_{k}(t)$ is the ROM wavefield approximation. 
}
   \label{fig:wallclock_time}
 \end{figure}

Next, we discuss the computational realization of the proposed approach. To that end, we illustrate the wall-clock times incurred in the offline and online phases in \cref{fig:wallclock_time}. We use the terminology offline phase and online phase to highlight the elapsed times for the RB spaces construction and the RB approximations of time-dependent seismograms or wavefields, respectively. Here, we consider the case of $\alpha=1.0\pi$, and the wall-clock times for the offline and online phases are averaged over 50 runs. For the implicit Newmark-beta method, we use the elapsed time for a single run. As discussed in \cref{rem:PODvsGreedy}, the choice between the POD method and the Greedy algorithm can be made depending on the available computational resources. As we observe in \cref{subfig:wall_clockOffline}, the wall-clock time for the POD method in a sequential setting can be significantly higher compared to the cases where we increase the number of processors. Meanwhile, the wall-clock time for the Greedy algorithm increases linearly and is comparable with the POD approach utilizing 64 processors. The significant computational cost in the POD method is in the generation of the data set containing solutions at random training parameters. In the Greedy algorithm, the dominating costs are the computation of the full order solutions for the RB space and setting up the a posteriori error estimator; the latter requires a few solutions of the parameter-independent high-dimensional linear system of equations in order to evaluate the error estimator for all frequencies efficiently (cf., \cite{Rozza2008ReducedEquations}).
 
In the online phase, we compare the wall-clock times to compute the time domain seismograms and the full wavefield for $t\in[0,20]$. 
Comparing the respective error quantities for an increasing dimension of the ROMs in \cref{subfig:wall_clockonline}, we observe that  the wall-clock time to approximate the seismograms or wavefield at any given time instance is under a fraction of a minute. Here, the wall-clock time when increasing the accuracy of the RB approximation of the seismograms grows only slightly. In the results, we have used the final time instance, i.e., $t=20$, to compare the error in the entire domain. Given that our objective is a fast approximation of the time domain seismograms, the approximation error in the whole domain can be relatively more.

\begin{figure}[t!]
   \centering
   \subfloat[Wall-clock time]{\label{subfig:Wall_clock_coarsePOD}\includegraphics[width=.33\linewidth]{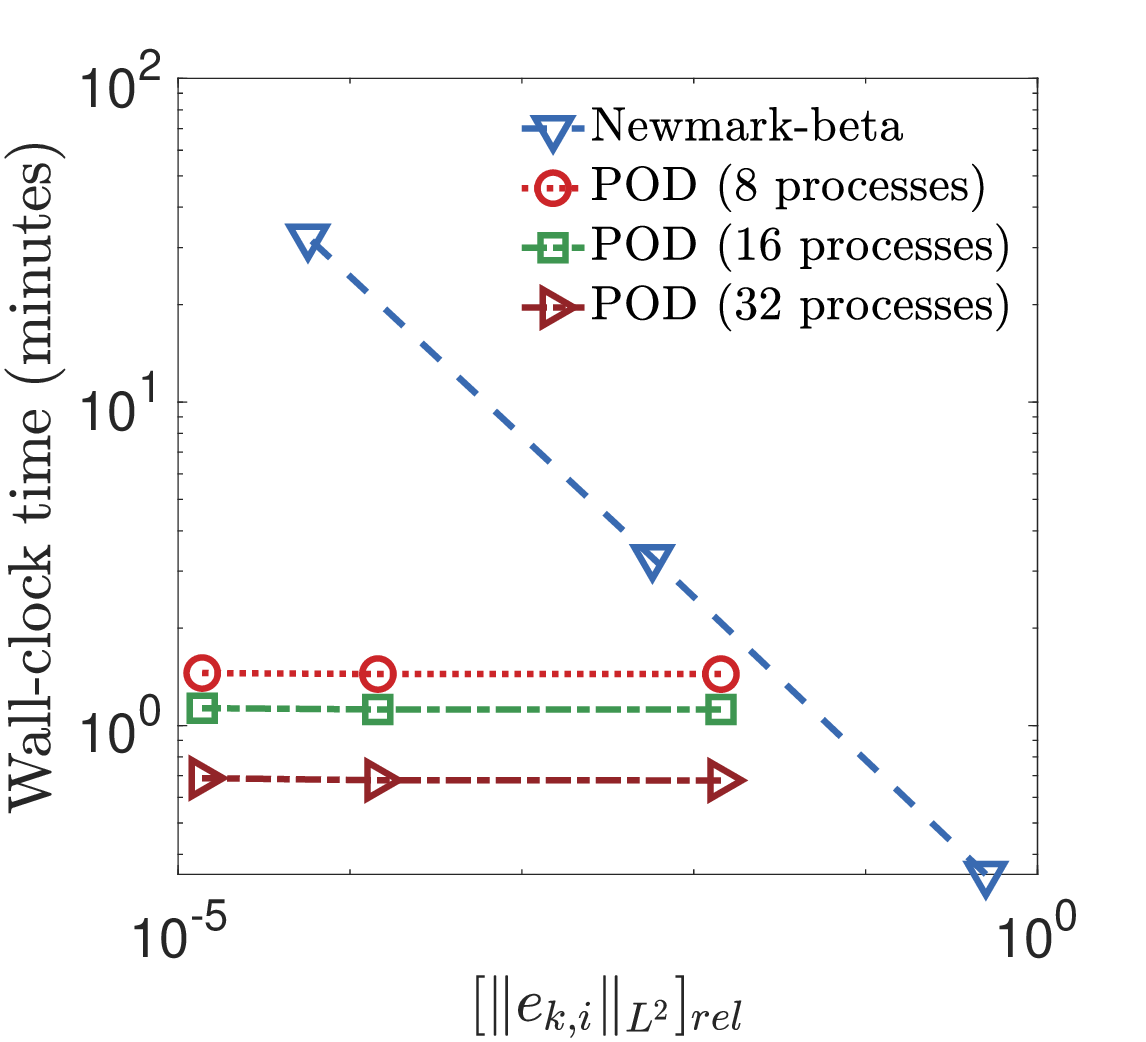}}
   \subfloat[Offline phase]{\label{subfig:POD_decay_scoarse}\includegraphics[width=.33\linewidth]{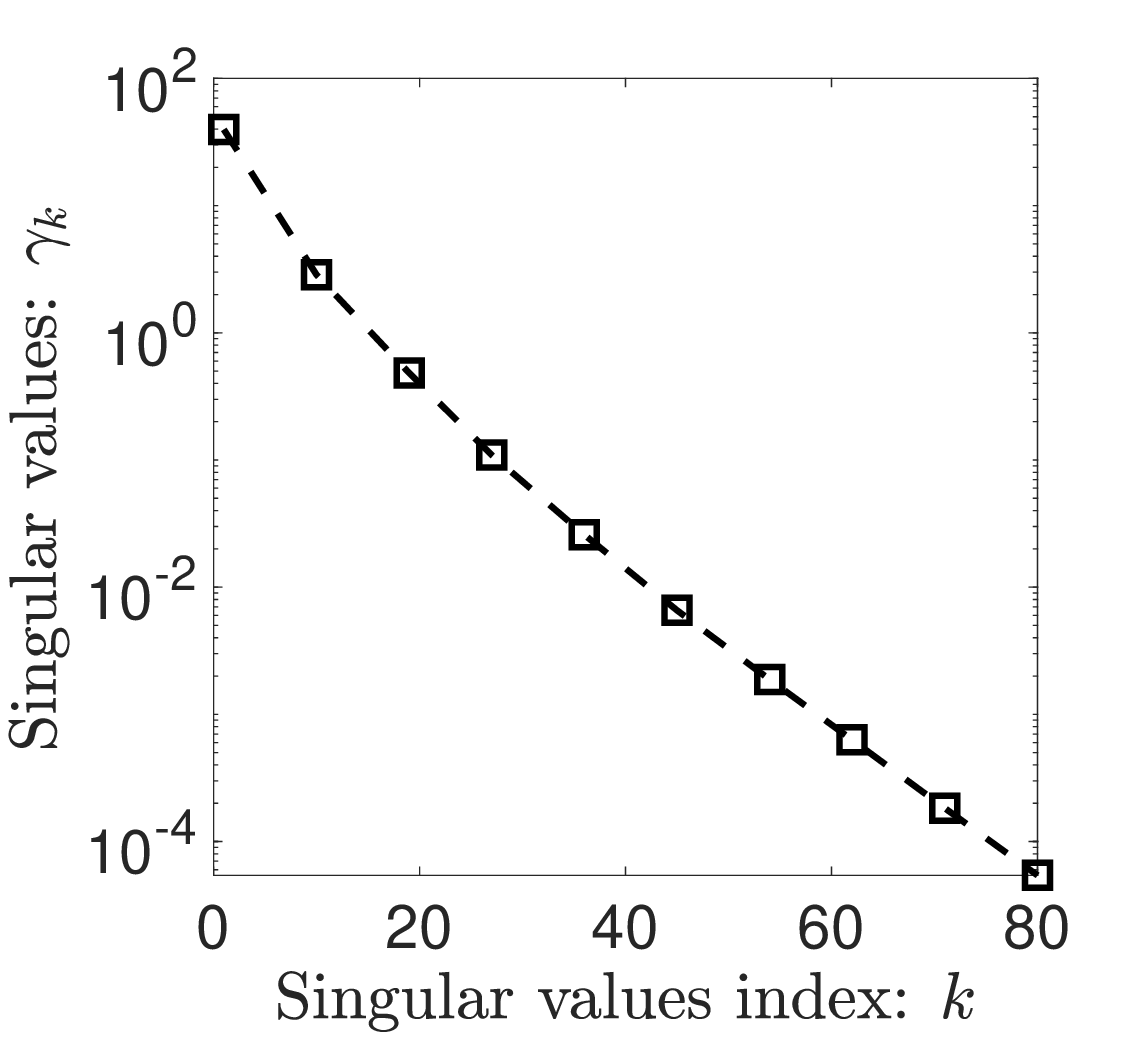}}
   \subfloat[Online phase]{\label{subfig:POD_errorTD_scoarse}\includegraphics[width=.33\linewidth]{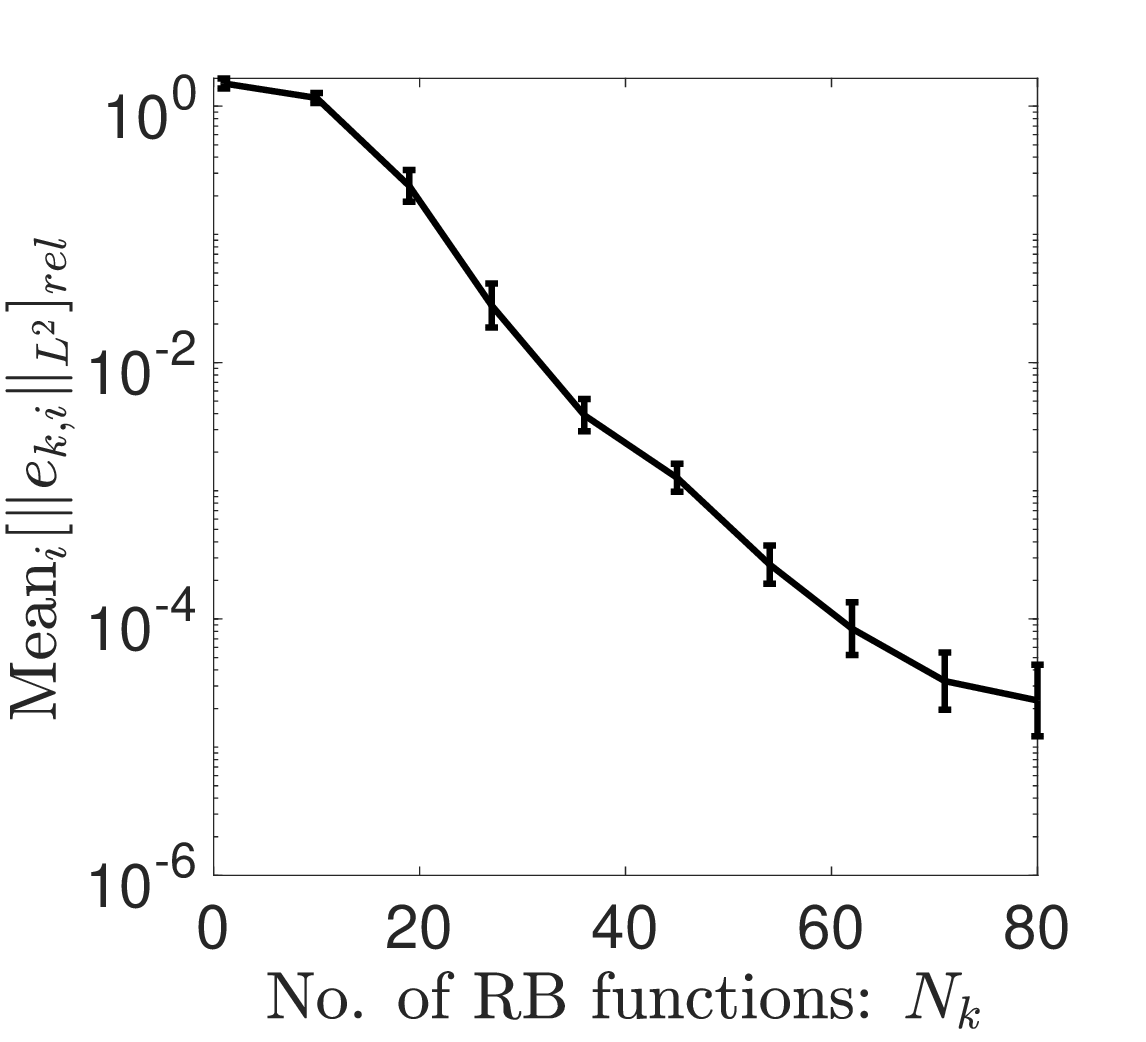}}
   \caption{ In (a), we plot the combined wall-clock time to construct the POD basis and approximating seismograms. Here, we have used 100 POD-basis functions to compute the approximation errors for decreasing time steps as $10^{-1},10^{-2}$ and $10^{-3}$. 
   The training set contains every second $\hat{s}$ specified in \cref{eq:a_p truncated}, resulting in a cardinality of 130. 
   The decay of the singular values is shown in (b) and the RB approximation error in the time domain seismograms is shown in (c).}
   \label{fig:Reduction_coarsePOD}
 \end{figure}

\cref{subfig:wall_clockNBM_FOM_ROM} shows a comparison of the wall-clock time for the RB approximation of the seismograms with the implicit Newmark-beta method. In this setting, we keep the dimension of the ROM, obtained using the Greedy algorithm, fixed at $N_k=150$ and consider time step sizes as $10^{-1},10^{-2}$ and $10^{-3}$. Here, we use more RB functions to avoid stagnation after $10^{-3}$ resulting from RB approximation errors. The reference solution is computed using the implicit Newmark-beta method with a step size of  $10^{-4}$. As noted in \cref{subfig:wall_clockNBM_FOM_ROM}, the time domain seismogram approximations for the ROM are not affected by the decrease of the time-step compared to square root growth in wall-clock time observed for the Newmark-beta method. However, the wall-clock time for the computation of wavefields at all time-steps using Weeks' method grows as 1/4. Here, the Clenshaw algorithm, which uses recursive computation of the Laguerre polynomials at all time instances, incurs the dominating computational cost. 

 \cref{subfig:Wall_clock_coarsePOD} demonstrates that by considering $\hat{s}$ from \cref{eq:a_p truncated} as the training parameter (\cref{rem:opti_para_week}), the combined wall-clock time for the offline and online phases is much smaller than for the implicit Newmark-beta method. 
 This approach is possible because the optimal parameters in $\hat{s}$ are computed only once (\cref{rem:opti_para_week}). We observe that this approach preserves the excellent approximation properties discussed earlier, as shown by the rapid decay of the singular values in \cref{subfig:POD_decay_scoarse} and the rapid decay in the time domain RB approximation errors in \cref{subfig:POD_errorTD_scoarse}. 
 
\subsection{Reduction for the \protect$ \mathcal{m}-$parameters}\label{sec:results_lameparameters}
 In this subsection, we apply the RB methods presented in \cref{sec:parametric reduction} to construct ROMs accommodating for variations in $\paramu=(\lambda,\mu)$ upto the maximum bounds given by $\lambda_{min},\lambda_{max},\mu_{min},\mu_{max}$.
 Here, we introduce a relative change in the reference values of $\paramu$ mentioned in \cref{fig:model} to find the maximum and minimum bounds of $\paramu$.
 In the following results, we consider the case $\alpha=1.0\pi$ for simplicity, while the results for increasing values of $\alpha$ can be understood in the context of frequency content as discussed in \cref{sec:reduction in the s parameter results}. 
 We use the notation $[\cdot]_{rel}={\cdot}/{\|\hat{\OoI}_k(\paramu)\|_{L^2} }$ for the true error or the a posteriori error estimator and $\hat{\OoI}_k(\paramu)$ denotes the time domain seismogram computed using $k$ RB functions. We compute the RB approximation errors using a reference solution obtained from Weeks' method utilizing full order solution. 
 In addition, we calculate $C_{\Week}=21.57$, as in \cref{eq:OoI estimate TD}, using the optimal parameters $(\wr,\wi)$ and the cut-off frequency $ s_{max}=11.74$ from the previous subsection. 

\begin{figure}[t!]
  \centering
   \subfloat[POD-Greedy algorithm]{\label{subfig:PODGreedy30Change}\includegraphics[width=.33\linewidth]{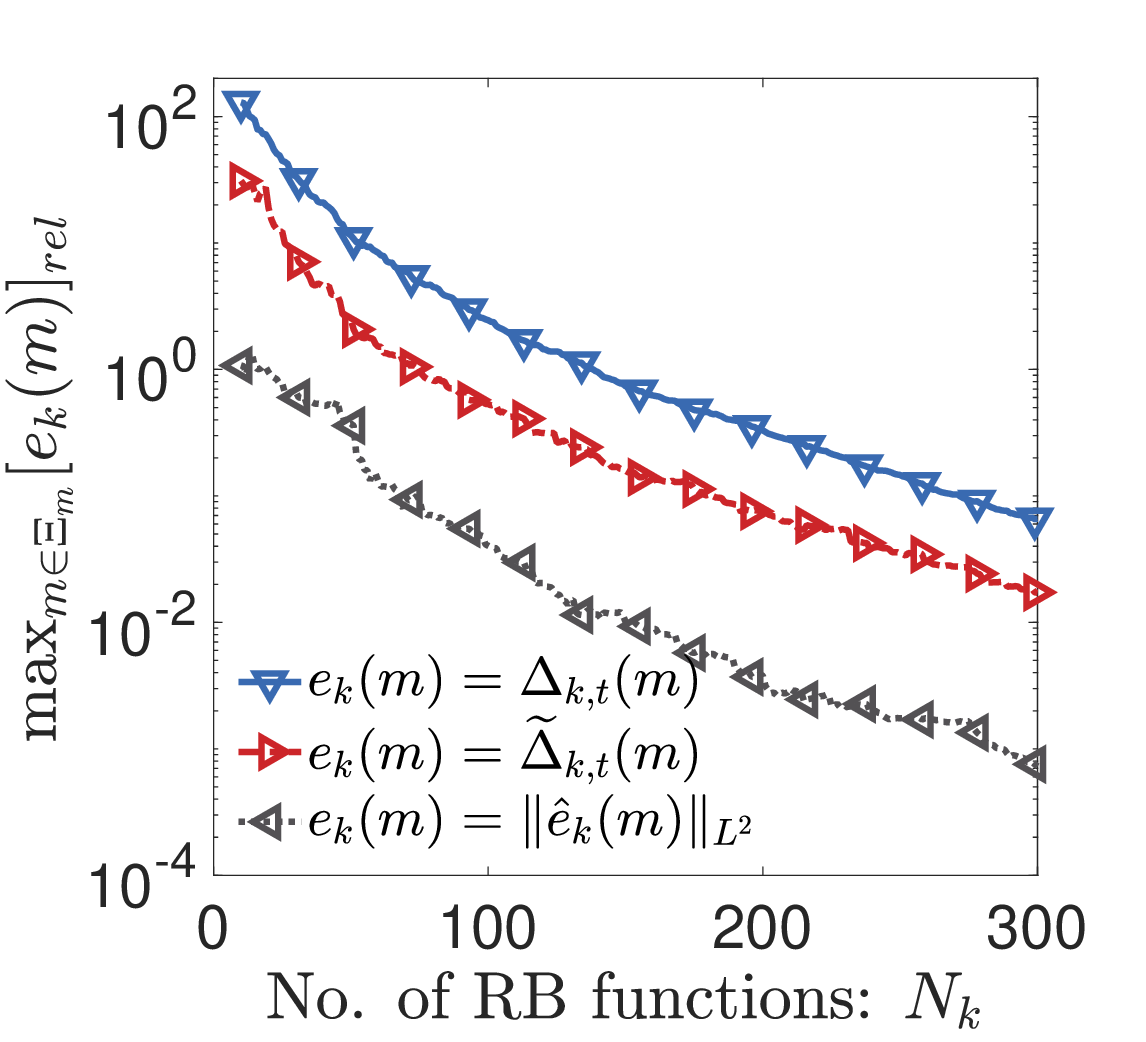}}
   \subfloat[Effectivity of the estimator]{\label{subfig:Effectivity30Change}\includegraphics[width=.33\linewidth]{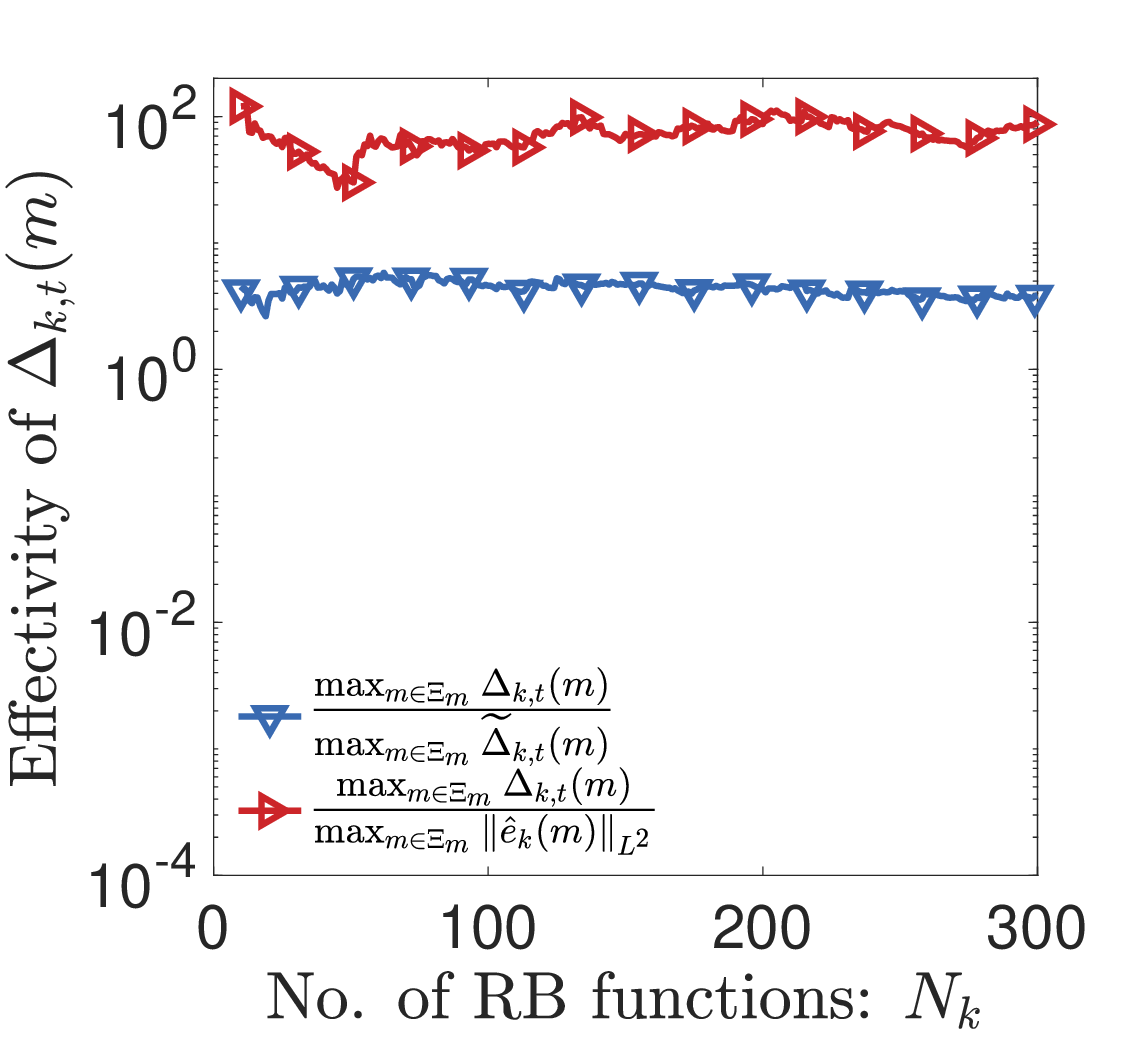}}
   \subfloat[Greedy algorithm]{\label{subfig:Greedyalgorithm30Change}\includegraphics[width=.33\linewidth]{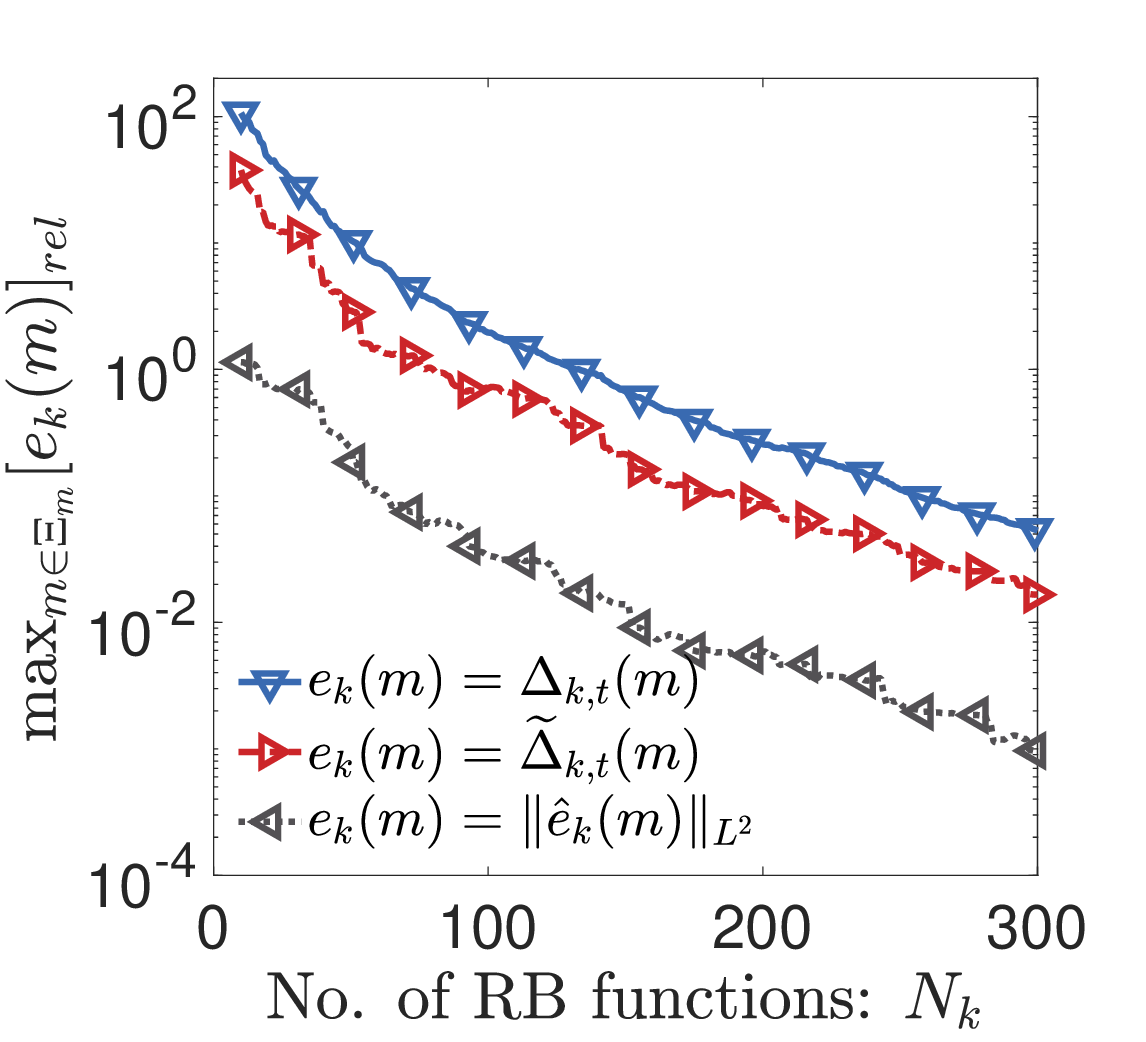}}  
   \caption{Test case 1: Construction of RB space with 30\% change in the material properties. Here, $\widetilde{\Delta}_{k,t}(\paramu):= C_{\Week} \sum_{j=1}^{N_z} |C_j| \|z_h(\hat{s}_j;\paramu) - z_h(\hat{s}_j;\paramu)\|_2$ is the upper bound for the true error $\hat{e}_k(\paramu)=\|\hat{\OoI}_h(\paramu) - \hat{\OoI}_k(\paramu)\|_{L^2}$. We scale $\Delta_{k,t}(\paramu)$ with $10^{-4}$ and the upper bound using $20^{-3}$.}
   \label{fig:Lame parameter RB construction 30Change}
 \end{figure}

\paragraph{Perturbations in $(\lambda,\mu)$} 
In the first test case, we perturb the material properties by a small amount, cf., \cite[Sec. 4.7]{hawkins2023model}. Similar perturbations occur, e.g., when solving an inverse problem to update the earth model during the optimization iterations \cite[Sec. 4.7]{hawkins2023model}. To this end, we introduce a 30$\%$ change in the material properties to establish the relative bounds $\lambda_{min}=0.70\lambda,\lambda_{max}=1.30\lambda,\mu_{min}=0.70\mu,$ and $\mu_{max}=1.30\mu$. To construct a training set $\Xi_m$, we draw 512 independent uniform random perturbations $\delta_{\lambda}^i\in[0.70,1.30]$ and $\delta_{\mu}^i\in[0.70,1.30]$ to select training samples $(\lambda^i_m,\mu^i_m)\in\Xi_m$ such that $\lambda^i_m=\delta^i_\lambda \lambda$ and $\mu^i_m=\delta^i_\mu \mu$ for $i=1,\ldots,512$.

In \cref{subfig:PODGreedy30Change}, we observe a rapid convergence of the POD-Greedy algorithm; here the algorithm terminates either if a maximum number of 300 RB functions or a tolerance of $10^{-1}$ is reached.
The moderate overestimation in $\Delta_{k,t}(\paramu)$ compared to \cref{subfig:Greedy_fixedm_Est} is due to the upper bound of the true error $\|\hat{\OoI}_h(\paramu) - \hat{\OoI}_k(\paramu)\|_{L^2}$. In \cref{subfig:Effectivity30Change}, we observe that the upper bound also converges at a very similar rate to the true error in the seismograms and helps to maintain an approximately constant effectivity of the a posteriori error estimator for an increasing number of RB functions. Comparing \cref{subfig:PODGreedy30Change,subfig:Greedyalgorithm30Change}, we observe that the POD-Greedy and Greedy algorithms converge rapidly. Since we construct the RB space using low-frequency domain solutions, both methods end up exploring almost similar subspaces of the full solution space, yielding a very similar rapid error decay analogous to the behaviour observed in \cref{sec:reduction in the s parameter results} for the POD and the Greedy algorithm.

In \cref{subfig:TestPGandG30}, we observe that the RB approximations obtained by both algorithms converge at nearly the same rate for a uniform random test set $\Xi_t$  of cardinality 128. Furthermore, in \cref{subfig:TestSeismoH30,subfig:TestSeismoV30}, we observe an accuracy of approximately $10^{-3}$  when using 300 RB functions to approximate time domain seismograms over a uniform random test set of size 128. Here, the reference solution is calculated with the implicit Newmark-beta method. This also confirms the discussion in \cref{sec:parameter selection Week}, that the optimal parameters computed for a reference $\paramu\in\Pspace$ can be re-used while ensuring a good approximation of the time domain seismograms.

In \cref{subfig:Lame_1030change}, we observe that by reducing the complexity of the problem to a $10\%$ relative change, we obtain an $L^2(0,20;\mathds{R}^2)-$norm error of $10^{-3}$ with 100 RB functions less than which were needed in the $30\%$ change. Here, we have used a uniform random training set of cardinality 512 with $10\%$ relative change in the minimum and maximum bounds for $\paramu$. In \cref{subfig:Test10change}, we observe that the RB approximation error converges with the same rate for a uniform random test set of cardinality 128. 

\begin{figure}[t!]
  \centering
   \subfloat[Online phase]{\label{subfig:TestPGandG30}\includegraphics[width=.33\linewidth]{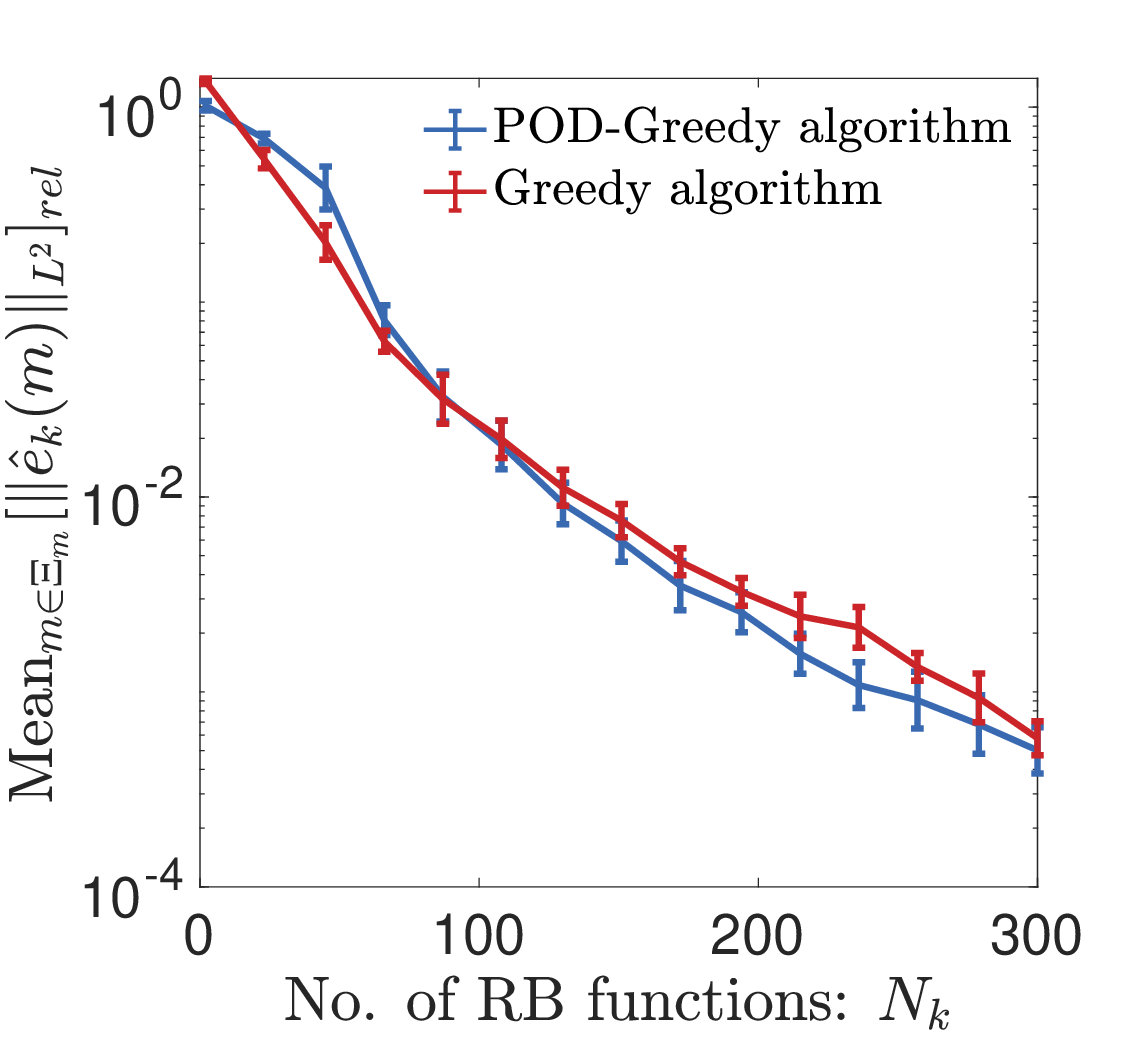}}  
   \subfloat[Horizontal component]{\label{subfig:TestSeismoH30}\includegraphics[width=.33\linewidth]{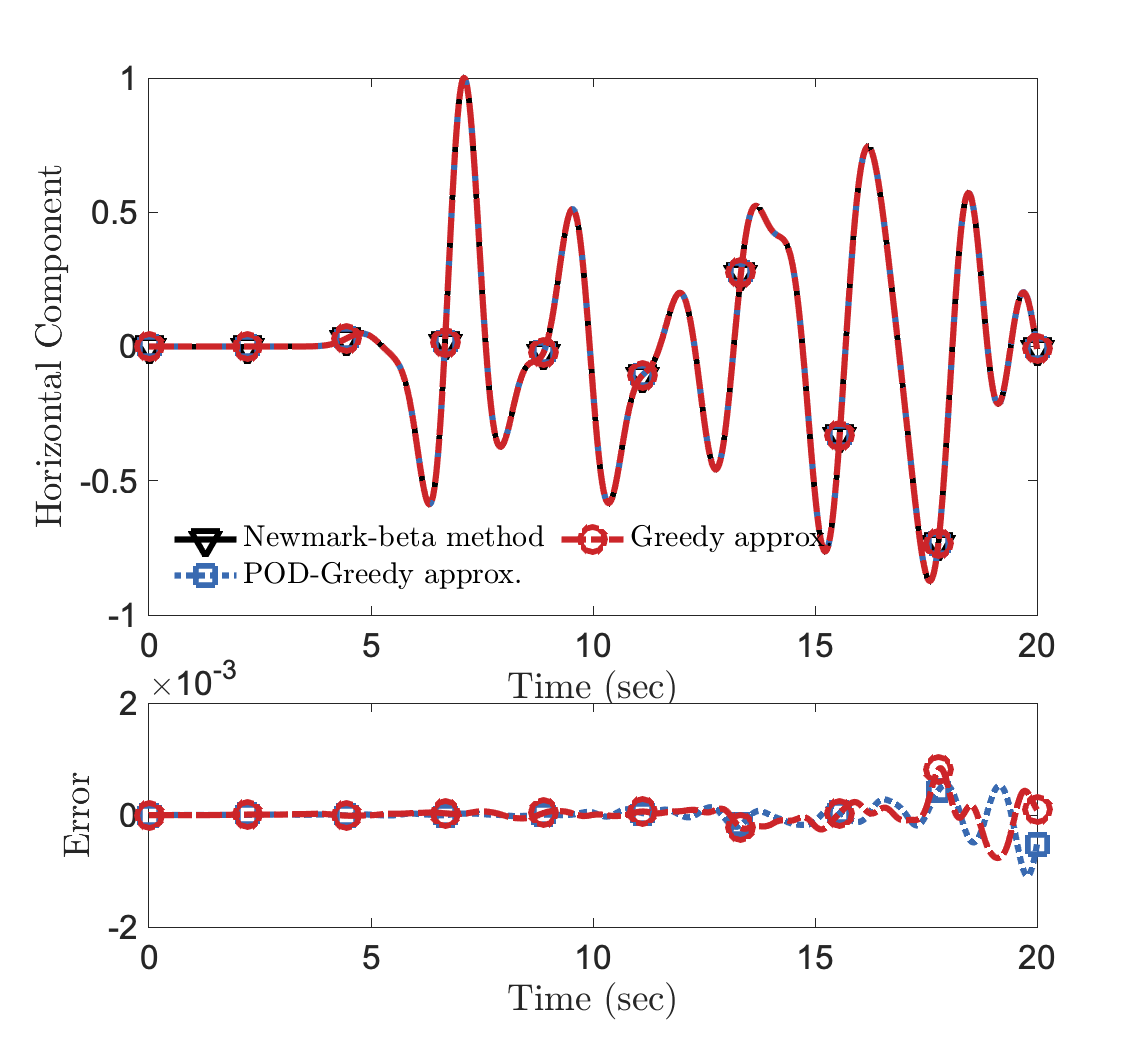}}
   \subfloat[Vertical component]{\label{subfig:TestSeismoV30}\includegraphics[width=.33\linewidth]{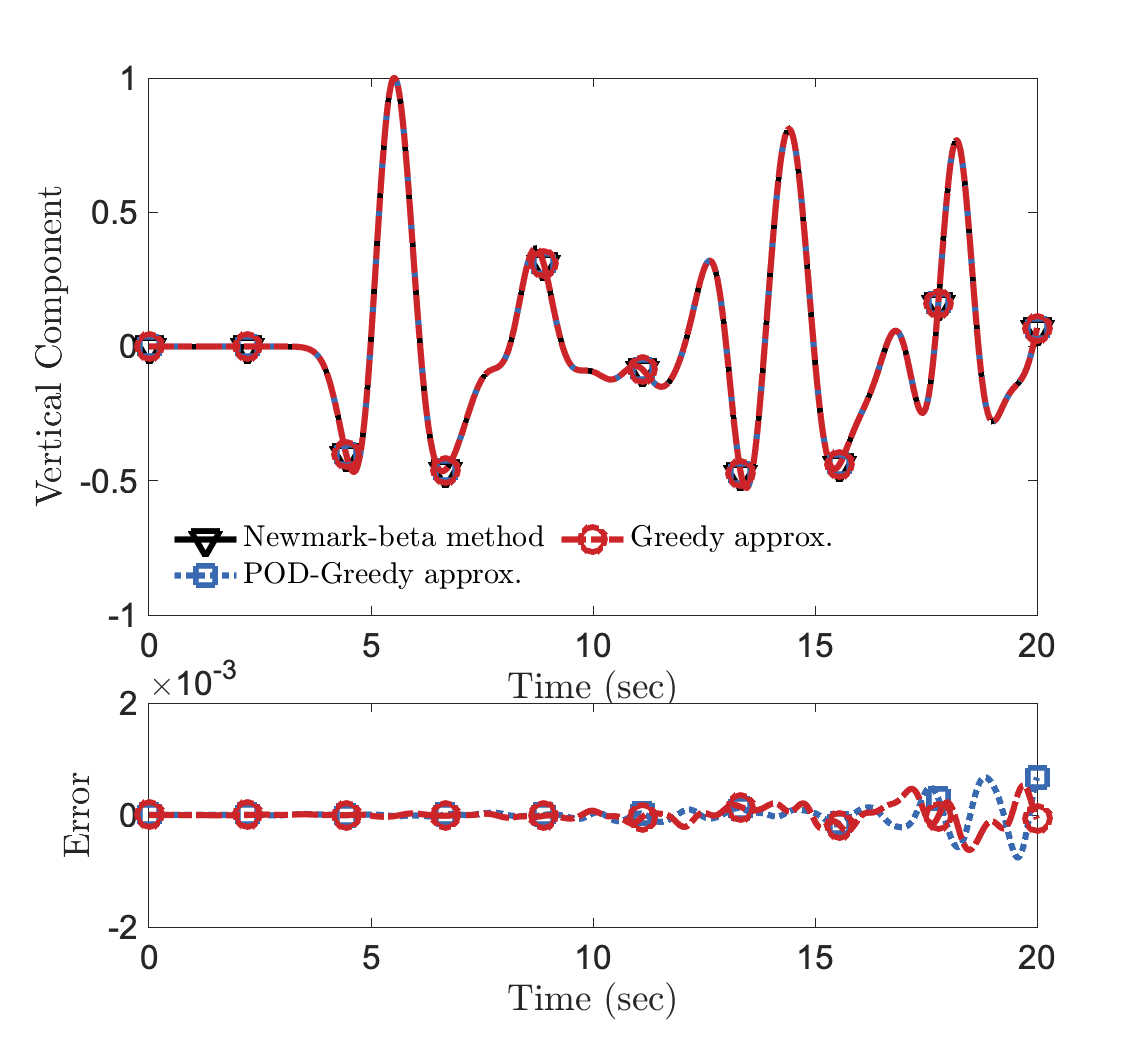}}
   \caption{Test case 1: (a) Comparison of relative $L^2(0,20;\mathds{R}^2)$ error norm convergence of the RB approximation of the seismograms over a test set $\Xi_t$ with 30\% change in the material properties. In (a), we have used full order solution as a reference to compute the mean error over $\paramu\in \Xi_t$, while in (b) and (c), we have used the implicit Newmark-beta method as a reference solution.}%\vspace*{-.5cm}
   \label{fig:Lame parameter RB test 30Change}
 \end{figure}

 \paragraph{Independent variations in the layers} 
 In the second test case, we introduce a uniform random variation that operates independently in each layer which is understood in the sense of piecewise constant perturbation per layer. Here, we use the same the uniform random piecewise constant perturbations for $\lambda\in [\lambda_{min},\lambda_{max}]$ and $\mu\in [\mu_{min},\mu_{max}]$ to avoid having a large set of parameters for the construction of the RB spaces. We introduce a $5\%$ and $10\%$ change to establish the upper and lower bounds for the variation in the material properties. 
 \begin{figure}[t!]
  \centering
   \subfloat[Case 1]{\label{subfig:Lame_1030change}\includegraphics[width=.33\linewidth,trim={0.6cm 0cm 19.1cm 0cm},clip]{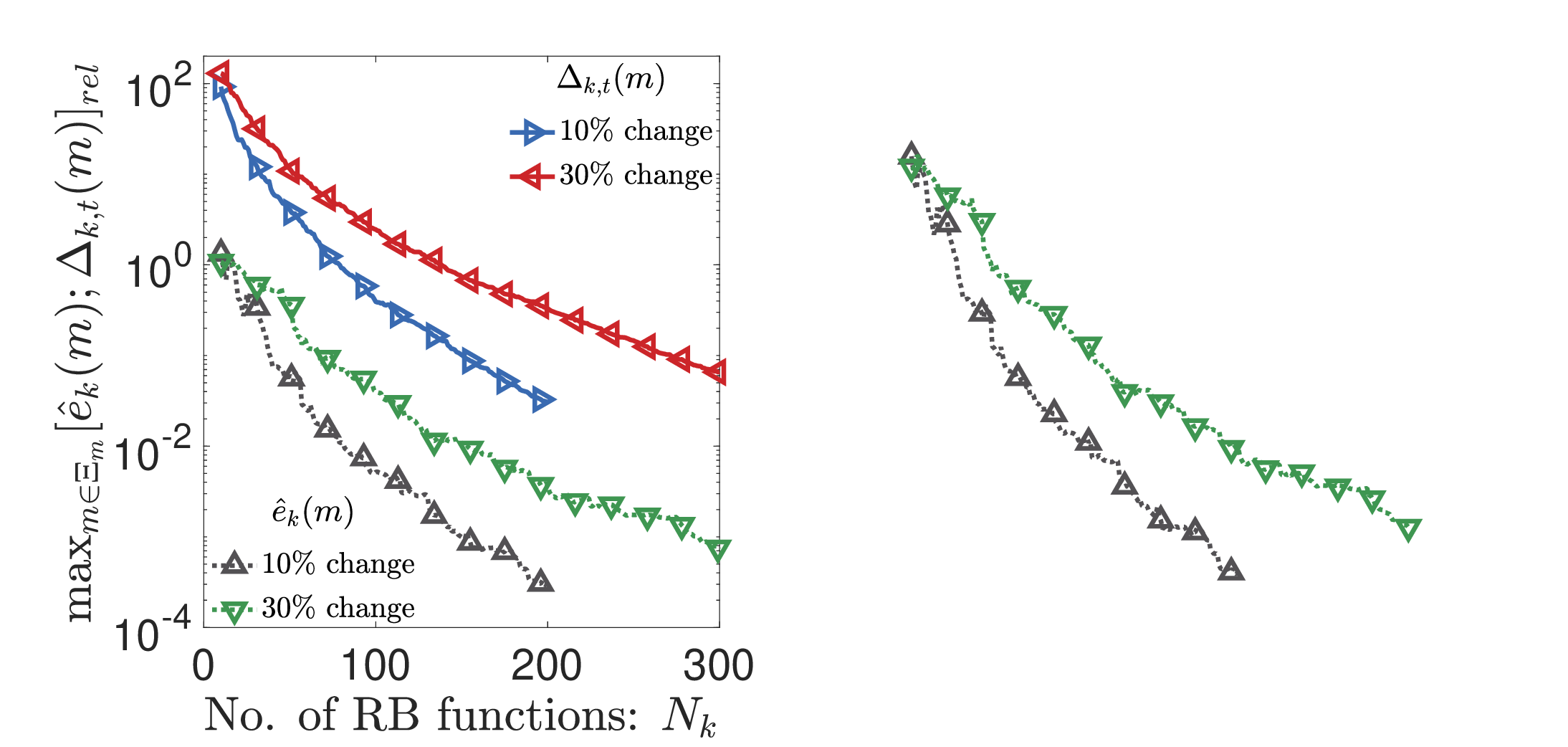}}
   \subfloat[Case 2]{\label{subfig:5and10allchange}\includegraphics[width=.33\linewidth,trim={0.5cm 0cm 19.1cm 0cm},clip]{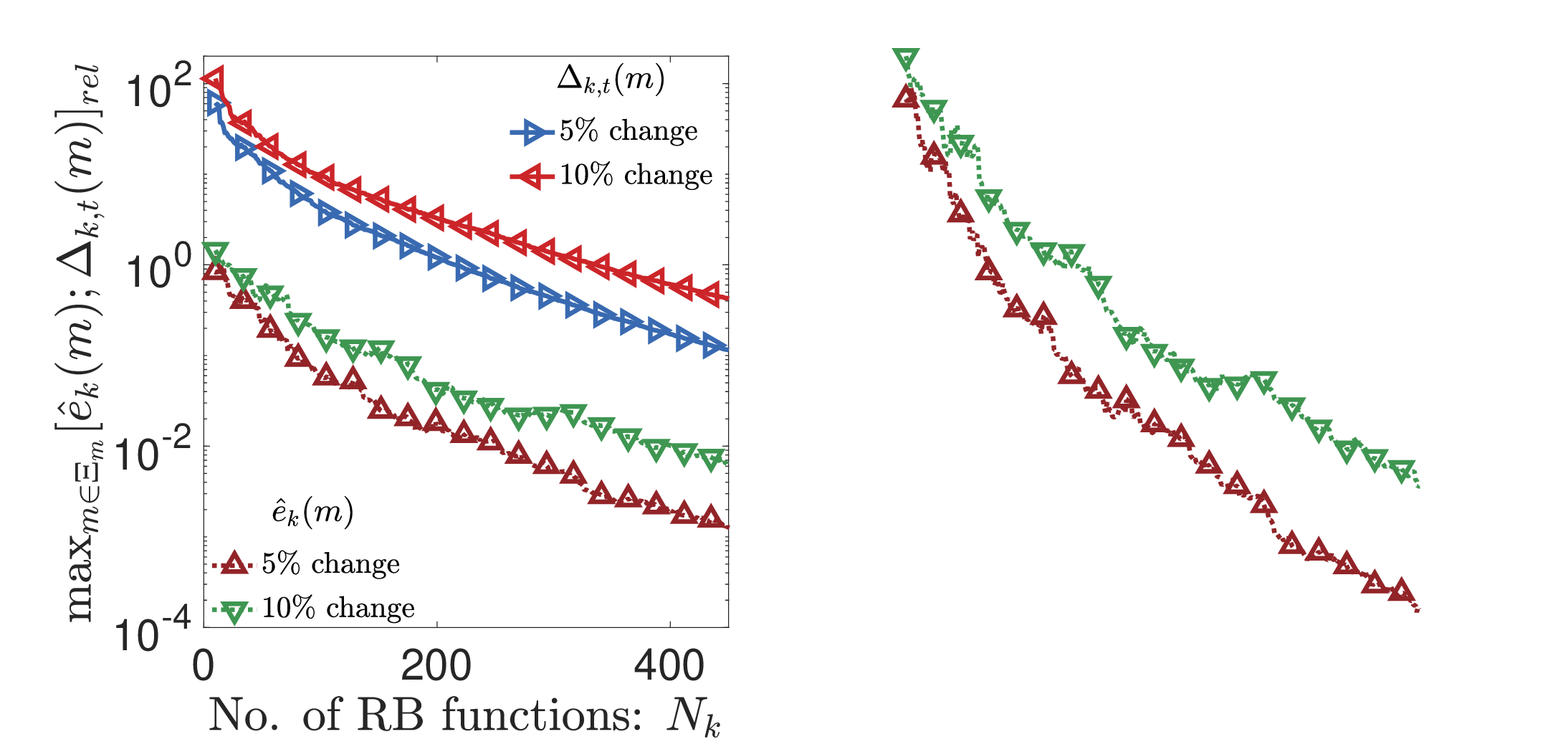}}
   \subfloat[Online phase]{\label{subfig:Test10change}\includegraphics[width=.35\linewidth]{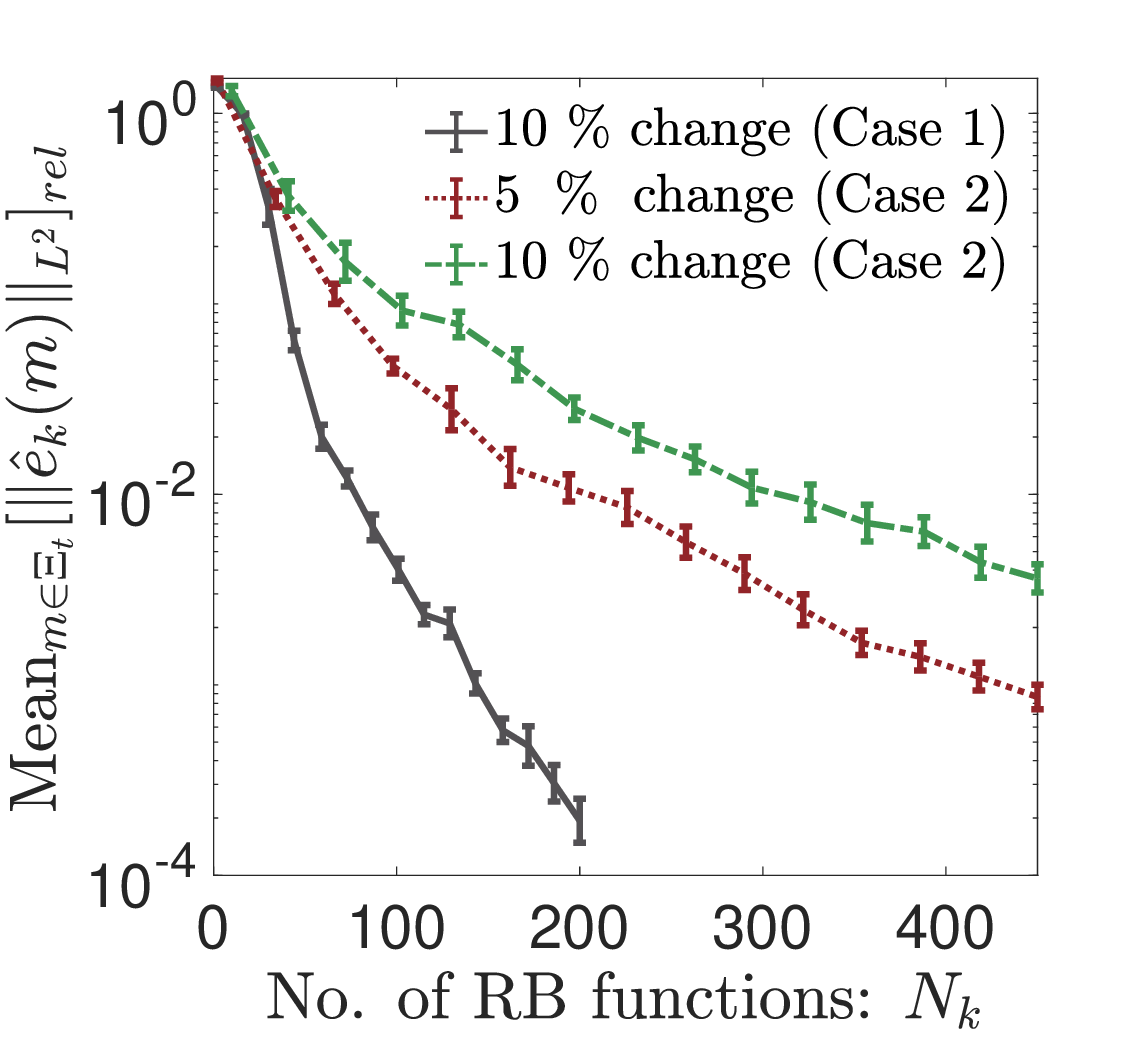}}
   \caption{ In (a) we illustrate the impact of variation in the material properties on the dimension of the ROM. In (b) we show the convergence of the POD-Greedy algorithm when all layers undergo an independent variation and its mean error over a test set in (c).  We scale the estimate with $10^{-4}$. }
   \label{fig:diffchanges}
 \end{figure}
 
 Subsequently, we construct a training set $\Xi_m$ of cardinality 512 and a test set $\Xi_t$ of cardinality 128 using a uniform random selection within the respective bounds of the Lam\'e parameters for both cases. In \cref{subfig:5and10allchange,subfig:Test10change}, we observe a very rapid convergence of the POD-Greedy algorithm and a similar rapid convergence in the mean errors over the corresponding uniform random test set. Compared to the finite element problem of dimension $N_h=233,318$, we can maintain an impressive accuracy of approximately $10^{-3}$ and $10^{-2}$ with only 450 RB functions in the cases of $5\%$ and $10\%$ changes, respectively.

\section{Conclusion}\label{sec:conclusions}
We have presented a goal-oriented model order reduction approach relying on RB methods to expedite simulations of the seismic wave propagation problem needed for many different parameters. 
We exploit the fact that the observed output of interest - seismograms - is low-pass filtered, and thus, we restrict ourselves to the low frequency region of the Laplace-transformed problem. 
Therefore, we obtain a stable ROM with provably exponential convergence of the RB approximation errors for a fixed subsurface model (cf., \cref{prop:infsup cont.,prop:kolmogorov}).
In our numerical results for a 2D realistic Groningen subsurface model, we observe exponential and rapid convergence of the RB approximation error in the time domain seismograms, even when wave speeds are relatively varied. 
 The a posteriori error estimator, which estimates the $L^2-$norm error in the time domain seismograms and drives the (POD)-Greedy algorithm in the parametric case, decays with approximately the same rate as the actual error (cf., \cref{subfig:Effectivity30Change}). 
We note (\cref{fig:Fixed_m_offlinephase} and \cref{fig:Seismo_fixedm}) that by increasing the frequency content of the Ricker wavelet, we need more RB functions to reach the desired tolerance. 
We thus conjecture that the convergence rates for approximating solutions - up to the desired accuracy - gets (significantly) worse in the high frequency regime. 
We also highlight that by using many processors, the wall-clock time to construct the ROMs for a fixed subsurface model and approximating the seismograms is significantly smaller than solving the FOM via an implicit Newmark-beta method; see, 
\cref{fig:wallclock_time,fig:Reduction_coarsePOD}. 
In the future work, we will explore absorbing boundary conditions to handle reflections from the boundaries and localized model order reduction \cite{buhr2localized} for large-scale 3D problems. 

\section*{Acknowledgments}
The authors thank Professor Jeannot Trampert for 
insightful discussions on the seismological applications. 
\FloatBarrier

\bibliographystyle{siamplain}

\begin{thebibliography}{10}
\bibitem{amsallem2014error}
{\sc D. ~Amsallem and U. ~Hetmaniuk}, {\em Error estimates for Galerkin reduced-order models of the semi-discrete wave equation}, ESAIM Math. Model Numer. Anal., 48 (2014), pp. 135–163.

\bibitem{arbes2023kolmogorov}
{\sc F. ~Arbes, C. ~Greif, and ~K. Urban}, {\em The Kolmogorov $N$-width for linear transport: Exact representation and the influence of the data}, 
arXiv preprint arXiv:2305.00066, (2023).

\bibitem{baydoun2020greedy}
{\sc K. ~Baydoun, M. ~Voigt, C. ~Jelich, and S. ~Marburg}, {\em A greedy reduced basis scheme for multifrequency solution of structural acoustic systems},
Int. J. Numer. Methods. Eng., 121 (2020), pp. 187–200.

\bibitem{benner2021frequency}
{\sc P. Benner and S. W. Werner}, {\em Frequency-and time-limited balanced truncation for large-scale second-order systems},
Int. J. Numer. Methods. Eng., 121 (2020), pp. 187–200.

\bibitem{MR1204279}
{\sc G. Berkooz, P. Holmes, and J. L. Lumley},
{\em The proper orthogonal decomposition in the analysis of turbulent flows},
Annu. Rev. Fluid Mech., 25, (1993), pp. 539–575.

\bibitem{bhouri2021two}
{\sc M. A. Bhouri and A. T. Patera},
{\em A two-level parameterized model-order reduction approach for time-domain elastodynamics},
Comput. Methods Appl. Mech. Engrg., 385 (2021), pp. Paper No. 114004, 27.

\bibitem{Bigoni2020SimulationbasedMonitoring}
{\sc C. Bigoni and J. S. Hesthaven},
{\em Simulation-based anomaly detection and damage localization: an application to structural health monitoring},
Comput. Methods Appl. Mech. Engr., 363 (2020), pp. 112896, 30.

\bibitem{MR2821591}
{\sc P. Binev, A. Cohen, W. Dahmen, R. DeVore, G. Petrova, and P. Wojtaszczyk},
{\em Convergence rates for greedy algorithms in reduced basis methods},
SIAM J. Math. Anal., 43 (2011), pp. 1457–1472.

\bibitem{borcea2021reduced}
{\sc L. Borcea, J. Garnier, A. V. Mamonov, and J. Zimmerling},
{\em Reduced order model approach for imaging with waves},
Inverse Probl., 38 (2021), p. 025004.

\bibitem{bozdaug2011misfit}
{\sc E. Bozda{\u{g}}, J. Trampert, and J. Tromp},
{\em Misfit functions for full waveform inversion based on instantaneous phase and envelope measurements},
Geophys. J. Int., 185 (2011), pp. 845–870.

\bibitem{brent2013algorithms}
{\sc R. P. Brent},
{\em Algorithms for minimization without derivatives},
Courier Corporation, 2013.

\bibitem{brunken2019parametrized}
{\sc J. Brunken, K. Smetana, and K. Urban},
{\em (Parametrized) First order transport equations: Realization of optimally stable Petrov–Galerkin methods},
SIAM J. Sci. Comput., 41 (2019), pp. A592–A621.

\bibitem{buffa2012priori}
{\sc A. Buffa, Y. Maday, A. T. Patera, C. Prud’homme, and G. Turinici},
{\em A priori convergence of the greedy algorithm for the parametrized reduced basis method},
ESAIM Math. Model Numer. Anal., 46 (2012), pp. 595–603.

\bibitem{buhr2localized}
{\sc A. Buhr, L. Iapichino, M. Ohlberger, S. Rave, F. Schindler, and K. Smetana},
{\em Snapshot-Based Methods and Algorithms},
vol. 2, De Gruyter, 2021, ch. Localized model reduction for parameterized problems.

\bibitem{bui2013constructively}
{\sc T. Bui-Thanh, L. Demkowicz, and O. Ghattas},
{\em Constructively well-posed approximation methods with unity inf–sup and continuity constants for partial differential equations}, Math. Comput., 82 (2013), pp. 1923–1952.

\bibitem{Clenshaw1955ASeries}
{\sc C. W. Clenshaw},
{\em A note on the summation of Chebyshev series},
Math. Tables Aids Comput., 9 (1955), pp. 118–120.

\bibitem{Dahmen2012}
{\sc W. Dahmen, C. Huang, C. Schwab, and G. Welper},
{\em Adaptive Petrov-Galerkin methods for first order transport equations},
SIAM J. Numer. Anal., 50 (2012), pp. 2420–2445.

\bibitem{dahmen2014double}
{\sc W. Dahmen, C. Plesken, and G. Welper},
{\em Double greedy algorithms: Reduced basis methods for transport dominated problems}, 
ESAIM Math. Model. Num., 48 (2014), pp. 623–663.

\bibitem{demkowicz2011class}
{\sc L. Demkowicz and J. Gopalakrishnan},
{\em A class of discontinuous Petrov-Galerkin methods. II. Optimal test functions},
Numer. Methods Partial Differential Equations, 27 (2011), pp. 70–105.

\bibitem{devore2013greedy}
{\sc R. DeVore, G. Petrova, and P. Wojtaszczyk},
{\em Greedy algorithms for reduced bases in Banach spaces},
Constr. Approx., 37 (2013), pp. 455–466.

\bibitem{druskin2017multiscale}
{\sc V. Druskin, A. V. Mamonov, and M. Zaslavsky},
{\em Multiscale s-fraction reduced-order models for massive wavefield simulations},
Multiscale Model. Simul., 15 (2017), pp. 445–475.

\bibitem{druskin2014extended}
{\sc V. Druskin, R. Remis, and M. Zaslavsky},
{\em An extended Krylov subspace model-order reduction technique to simulate wave propagation in unbounded domains},
J. Comput. Phys., 272 (2014), pp. 608–618.

\bibitem{druskin2018compressing}
{\sc V. Druskin, R. F. Remis, M. Zaslavsky, and J. T. Zimmerling},
{\em Compressing large-scale wave propagation models via phase-preconditioned rational Krylov subspaces},
Multiscale Model. Simul., 16 (2018), pp. 1486–1518.

\bibitem{du2000noise}
{\sc Z. Du, G. Foulger, and W. Mao},
{\em Noise reduction for broad-band, three-component seismograms using data-adaptive polarization filters}, 
Geophys. J. Int., 141 (2000), pp. 820–828.

\bibitem{feeny2008complex}
{\sc B. Feeny},
{\em A complex orthogonal decomposition for wave motion analysis}, JSV, 310 (2008), pp. 77–90.

\bibitem{fichtner2010full}
{\sc A. Fichtner},
{\em Full seismic waveform modelling and inversion}, 
Springer Science \& Business Media, 2010.

\bibitem{fichtner2009simulation}
{\sc A. Fichtner, H. Igel, H.-P. Bunge, and B. L. Kennett},
{\em Simulation and inversion of seismic wave propagation on continental scales based on a spectral-element method}, 
JNAIAM, 4 (2009), pp. 11–22.

\bibitem{geuzaine2009gmsh}
{\sc C. Geuzaine, J.-F. Remacle, and P. Dular},
{\em Gmsh: a three-dimensional finite element mesh generator},
Int. J. Numer. Methods Eng, 79 (2009), pp. 1309–1331.

\bibitem{glas2020reduced}
{\sc S. Glas, A. T. Patera, and K. Urban},
{\em A reduced basis method for the wave equation}, 
Int. J. Comut. Fluid Dyn., 34 (2020), pp. 139–146.

\bibitem{greif2019decay}
{\sc C. Greif and K. Urban},
{\em Decay of the Kolmogorov $N$-width for wave problems}, 
Appl. Math. Lett., 96 (2019), pp. 216–222.

\bibitem{grepl2005}
{\sc M. A. Grepl and A. T. Patera},
{\em A posteriori error bounds for reduced-basis approximations of parametrized parabolic partial differential equations}, 
ESAIM: Math. Model. Numer. Anal., 39 (2005), pp. 157–181.

\bibitem{guglielmi2022contour}
{\sc N. Guglielmi and M. Manucci},
{\em Model order reduction in contour integral methods for parametric PDEs}, 
SIAM J. Sci. Comput., 45 (2023), pp. A1711–A1740.

\bibitem{HaasdonkPODGreedyConvergence}
{\sc B. Haasdonk},
{\em Convergence rates of the POD-greedy method}, 
ESAIM Math. Model. Numer. Anal., 47 (2013), pp. 859–873.

\bibitem{haasdonk2008reduced}
{\sc B. Haasdonk and M. Ohlberger},
{\em Reduced basis method for finite volume approximations of parametrized linear evolution equations}, 
M2AN Math. Model. Numer. Anal., 42 (2008), pp. 277–302.

\bibitem{halko2011finding}
{\sc N. Halko, P. G. Martinsson, and J. A. Tropp},
{\em Finding structure with randomness: probabilistic algorithms for constructing approximate matrix decompositions}, 
SIAM Rev., 53 (2011), pp. 217–288.

\bibitem{hawkins2023model}
{\sc R. Hawkins, M. Khalid, K. Smetana, and J. Trampert},
{\em Model order reduction for seismic waveform modelling: inspiration from normal modes}, 
Geophys. J. Int., 234 (2023), pp. 2255–2283.

\bibitem{henning2022ultraweak}
{\sc J. Henning, D. Palitta, V. Simoncini, and K. Urban},
{\em An ultraweak space-time variational formulation for the wave equation: Analysis and efficient numerical solution}, 
ESAIM Math. Model. Numer. Anal., 56 (2022), pp. 1173–1198.

\bibitem{henriquez2024fast}
{\sc F. Henriquez and J. S. Hesthaven},
{\em Fast numerical approximation of linear, second-order hyperbolic problems using model order reduction and the Laplace transform}, 
arXiv preprint arXiv:2405.19896, (2024).

\bibitem{horgan1995korn}
{\sc C. O. Horgan},
{\em Korn’s inequalities and their applications in continuum mechanics}, 
SIAM Rev., 37 (1995), pp. 491–511

\bibitem{huynh2011laplace}
{\sc D. B. P. Huynh, D. J. Knezevic, and A. T. Patera},
{\em A Laplace transform certified reduced basis method; application to the heat equation and wave equation},
C. R. Math. Acad. Sci. Paris, 349 (2011), pp. 401–405.

\bibitem{MR2367928}
{\sc D. B. P. Huynh, G. Rozza, S. Sen, and A. T. Patera},
{\em A successive constraint linear optimization method for lower bounds of parametric coercivity and inf-sup stability constants},
C. R. Math. Acad. Sci. Paris, 345 (2007), pp. 473–478.

\bibitem{Brio2005ApplicationExponential}
{\sc P. O. Kano, M. Brio, and J. V. Moloney},
{\em Application of Weeks method for the numerical inversion of the Laplace transform to the matrix exponential},
Commun. Math. Sci., 3 (2005), pp. 335–372.

\bibitem{kaufl2016solving}
{\sc P. Kaufl, A. P. Valentine, R. W. de Wit, and J. Trampert},
{\em Solving probabilistic inverse problems rapidly with prior samples},
Geophys. J. Int., 205 (2016), pp. 1710–1728.

\bibitem{CODES}
{\sc M. H. Khalid},
{\em Source code to Model order reduction for seismic applications, 2024},
https://doi.org/10.5281/zenodo.11520031.

\bibitem{koko2016}
{\sc J. Koko},
{\em Fast MATLAB assembly of FEM matrices in 2D and 3D using cell-array approach},
Int. J. Model. Simul. Sci. Comput., 7 (2016), p. 1650010.

\bibitem{kolmogoroff1936uber}
{\sc A. Kolmogoroff},
{\em Uber die beste Ann¨aherung von Funktionen einer gegebenen Funktionen-klasse},
Ann. of Math. (2), 37 (1936), pp. 107–110.

\bibitem{komatitsch2002spectrala}
{\sc D. Komatitsch and J. Tromp},
{\em Spectral-element simulations of global seismic wave propagation—I. validation},
Geophys. J. Int., 149 (2002), pp. 390–412.

\bibitem{komatitsch2003perfectly}
{\sc D. Komatitsch and J. Tromp},
{\em A perfectly matched layer absorbing boundary condition for the second-order seismic wave equation},
Geophys. J. Int., 154 (2003), pp. 146–153.

\bibitem{Kruiver2017A}
{\sc P. Kruiver, E. van Dedem, R. Romijn, G. de Lange, M. Korff, J. Stafleu, J. Gunnink, A. Rodriguez-Marek, J. Bommer, and D. van Elk, J. Doornhof},
{\em An integrated shear wave velocity model for the Groningen gas field, the Netherlands},
Bull. Earthq. Eng., (2017).

\bibitem{MR1868765}
{\sc K. Kunisch and S. Volkwein},
{\em Galerkin proper orthogonal decomposition methods for parabolic problems},
Numer. Math., 90 (2001), pp. 117–148.

\bibitem{lee2020model}
{\sc K. Lee and K. T. Carlberg},
{\em Model reduction of dynamical systems on nonlinear manifolds using deep convolutional autoencoders},
J. Comput. Phys., 404 (2020), pp. 108973, 32.

\bibitem{mamonov2015nonlinear}
{\sc A. V. Mamonov, V. Druskin, and M. Zaslavsky},
{\em Nonlinear seismic imaging via reduced order model backprojection},
in SEG Technical Program Expanded Abstracts 2015, Society of Exploration Geophysicists, 2015, pp. 4375–4379.

\bibitem{newmark1959method}
{\sc N. M. Newmark},
{\em A method of computation for structural dynamics},
J. Eng. Mech., 85 (1959), pp. 67–94.

\bibitem{ohlberger2015reduced}
{\sc M. Ohlberger and S. Rave},
{\em Reduced basis methods: Success, limitations and future challenges},
Proceedings of the conference Algoritmy, (2016).

\bibitem{pagliantini2021dynamical}
{\sc C. Pagliantini},
{\em Dynamical reduced basis methods for Hamiltonian systems},
Numer. Math., 148 (2021), pp. 409–448.

\bibitem{panza2012seismic}
{\sc G. F. Panza, C. La Mura, A. Peresan, F. Romanelli, and F. Vaccari},
{\em Seismic hazard scenarios as preventive tools for a disaster resilient society},
in Adv. Geophys., vol. 53, Elsevier, 2012, pp. 93–165.

\bibitem{Peng2016SymplecticSystems}
{\sc L. Peng and K. Mohseni},
{\em Symplectic model reduction of Hamiltonian systems},
SIAM J. Sci. Comput., 38 (2016), pp. A1–A27.

\bibitem{Pratt2012SeismicModel}
{\sc R. G. Pratt},
{\em Seismic waveform inversion in the frequency domain, Part 1: Theory and verification in a physical scale model},
Geophysics, 64 (2012), pp. 888–901.

\bibitem{Rawlinson2014}
{\sc N. Rawlinson, A. Fichtner, M. Sambridge, and M. K. Young},
{\em Advances in Geophysics}, 
vol. 55, Academic Press, 2014, ch. Seismic tomography and the assessment of uncertainty.

\bibitem{reiss2018shifted}
{\sc J. Reiss, P. Schulze, J. Sesterhenn, and V. Mehrmann},
{\em The shifted proper orthogonal decomposition: a mode decomposition for multiple transport phenomena},
SIAM J. Sci. Comput., 40 (2018), pp. A1322–A1344.

\bibitem{Romijn2017A}
{\sc R. Romijn},
{\em Groningen velocity model 2017}, 
tech. report, Nederlandse Aardolie Maatschappij, Assen, The Netherlands, 2017.

\bibitem{MR4549101}
{\sc F. Romor, G. Stabile, and G. Rozza},
{\em Non-linear manifold reduced-order models with convolutional autoencoders and reduced over-collocation method},
J. Sci. Comput., 94 (2023), p. 74.

\bibitem{Rozza2008ReducedEquations}
{\sc G. Rozza, D. B. P. Huynh, and A. T. Patera},
{\em Reduced basis approximation and a posteriori error estimation for affinely parametrized elliptic coercive partial differential equations: application to transport and continuum mechanics},
Arch. Comput. Methods Eng., 15 (2008), pp. 229–275.

\bibitem{MR0910462}
{\sc L. Sirovich},
{\em Turbulence and the dynamics of coherent structures. I. Coherent structures},
Quart. Appl. Math., 45 (1987), pp. 561–571.

\bibitem{sriyanto2023performance}
{\sc S. P. D. Sriyanto, A. R. Puhi, and C. H. G. Sibuea},
{\em The performance of Butterworth and wiener filter for earthquake signal enhancement: a comparative study},
J Seismol, 27 (2023), pp. 219–232.

\bibitem{taddei2020registration}
{\sc T. Taddei},
{\em A registration method for model order reduction: data compression and geometry reduction},
SIAM J. Sci. Comput., 42 (2020), pp. A997–A1027.

\bibitem{talbot1979accurate}
{\sc A. Talbot},
{\em The accurate numerical inversion of Laplace transforms},
J. Inst. Math. Appl., 23 (1979), pp. 97–120.

\bibitem{tonn2011comparison}
{\sc T. Tonn, K. Urban, and S. Volkwein},
{\em Comparison of the reduced-basis and POD a posteriori error estimators for an elliptic linear-quadratic optimal control problem},
Math. Comput. Model. Dyn. Syst., 17 (2011), pp. 355–369.

\bibitem{treister2017full}
{\sc E. Treister and E. Haber},
{\em Full waveform inversion guided by travel time tomography},
SIAM J. Sci. Comput., 39 (2017), pp. S587–S609.

\bibitem{Tromp2019SeismicScales}
{\sc J. Tromp},
{\em Seismic wavefield imaging of Earth’s interior across scales},
Nature Reviews Earth \& Environment 2019 1:1, 1 (2019), pp. 40–53.

\bibitem{tromp2008spectral}
{\sc J. Tromp, D. Komatitsch, and Q. Liu},
{\em Spectral-element and adjoint methods in seismology},
Commun. Comput. Phys., 3 (2008), pp. 1–32.

\bibitem{Tromp2005SeismicKernels}
{\sc J. Tromp, C. Tape, and Q. Liu},
{\em Seismic tomography, adjoint methods, time reversal and banana-doughnut kernels},
Geophys. J. Int., 160 (2005), pp. 195–216.

\bibitem{van20143d}
{\sc T. van Leeuwen and F. J. Herrmann},
{\em 3D frequency-domain seismic inversion with controlled sloppiness},
SIAM J. Sci. Comput., 36 (2014), pp. S192–S217.

\bibitem{veroy2003posteriori}
{\sc K. Veroy, C. Prud’Homme, D. Rovas, and A. Patera},
{\em A posteriori error bounds for reduced-basis approximation of parametrized noncoercive and nonlinear elliptic partial differential equations},
in Proceedings of the 16th AIAA Computational Fluid Dynamics Conference, 2003, p. 3847.

\bibitem{volkwein2001optimal}
{\sc S. Volkwein},
{\em Optimal control of a phase-field model using proper orthogonal decomposition},
ZAMM Z. Angew. Math. Mech., 81 (2001), pp. 83–97.

\bibitem{wang2020regularized}
{\sc H. Wang, Q. Guo, T. Alkhalifah, and Z. Wu},
{\em Regularized elastic passive equivalent source inversion with full-waveform inversion: Application to a field monitoring microseismic data set},
Geophysics, 85 (2020), pp. KS207–KS219.

\bibitem{wang2015frequencies}
{\sc Y. Wang},
{\em Frequencies of the Ricker wavelet},
Geophysics, 80 (2015), pp. A31–A37.

\bibitem{Weeks1966NumericalFunctions}
{\sc W. T. Weeks},
{\em Numerical inversion of Laplace transforms using Laguerre functions},
J. Assoc. Comput. Mach., 13 (1966), pp. 419–429.

\bibitem{Weideman1999AlgorithmsTransform}
{\sc J. A. C. Weideman},
{\em Algorithms for parameter selection in the Weeks method for inverting the Laplace transform},
SIAM J. Sci. Comput., 21 (1999), pp. 111–128.

\bibitem{weideman2007parabolic}
{\sc J. A. C. Weideman and L. N. Trefethen},
{\em Parabolic and hyperbolic contours for computing the Bromwich integral},
Math. Comp., 76 (2007), pp. 1341–1356.

\bibitem{welper2017interpolation}
{\sc G. Welper},
{\em Interpolation of functions with parameter dependent jumps by transformed snapshots},
SIAM J. Sci. Comput., 39 (2017), pp. A1225–A1250.

\bibitem{woodward1992wave}
{\sc M. J. Woodward},
{\em Wave-equation tomography},
Geophysics, 57 (1992), pp. 15–26.

\bibitem{zitelli2011class}
{\sc J. Zitelli, I. Muga, L. Demkowicz, J. Gopalakrishnan, D. Pardo, and V. M. Calo},
{\em A class of discontinuous Petrov–Galerkin methods. part IV: The optimal test norm and time-harmonic wave propagation in 1D},
J. Comput. Phys., 230 (2011), pp. 2406–2432.

\end{thebibliography}

\end{document}

\maketitle

\section{Well-posedness}\label{supplementary:Wellposedness}
To show well-posedness of the weak formulation \cref{eq: Weak formulation Laplace domain} let us state a preliminary result. First, consider the following eigenvalue problem:
Find $ (v(\paramu),\tau(\paramu))\in \U\times \RR$ such that,
\begin{equation}
    b(v(\paramu),{\ldw};\paramu) = \tau(\paramu)\, \big(v(\paramu),\ldw\big)_{\U(\paramu)} \quad \forall \;\ldw\in \U. \label{eq:Stability bilinear proposition EVP coercive case cont.}
\end{equation}
We assume the normalization $\|v(\paramu)\|_{\U(m)}=1$.
The next result establishes that the eigenvalues $\tau(\paramu)$ are bounded uniformly in $\paramu$. For its proof, let us recall the coercivity of the bilinear form in the weighted $\|w\|_{H^1_{\rho}}^2:=\|\rho^{1/2} w\|_{L^2}^2+\|\nabla w\|_{L^2}^2$ norm, which follows from Korn's inequality \cite{horgan1995korn},
\begin{equation}\label{eq:Korn}
b(w,w;\paramu) \geq \CK\|w\|_{H^1_{\rho}}^2\quad\text{for all } w\in \U,
\end{equation}
for some $\CK>0$, which depends on $\Omega$, $\Gamma_D$ and the lower bound $\mu_0$ in \cref{eq:Lame_bounds}.

\begin{proposition}\label{prop:lower_lambda_0}
For any eigenvalue $\tau(\paramu)$ of \cref{eq:Stability bilinear proposition EVP coercive case cont.} it holds that
\begin{equation}
    \tau_0\leq \tau(\paramu) \leq  1\qquad \text{ for } \tau_0=\frac{\CK}{1+\CK}>0.
    \label{eq:Lambda0 constant}
\end{equation}
\end{proposition}
\begin{proof}
The upper bound $\tau(\paramu)\leq 1$ follows directly from the definition of $(\cdot,\cdot)_{\U(\paramu)}$ inner product %\eqref{eq:inner product def Cont} 
and \eqref{eq:Stability bilinear proposition EVP coercive case cont.} with $w=v(\paramu)$. To prove the other bound suppose $\tau(\paramu)<1$. Using 
the definition of $(\cdot,\cdot)_{\U(\paramu)}$ %\eqref{eq:inner product def Cont} 
and rearranging \eqref{eq:Stability bilinear proposition EVP coercive case cont.} we obtain that
\begin{equation}\label{eq:lambda_0_bound_1}
    b(v(\paramu),v(\paramu);\paramu) = \frac{\tau(\paramu)}{1 - \tau(\paramu)}\, (\rho v(\paramu),v\big(\paramu)).
\end{equation}
Using \eqref{eq:Korn} and $(\rho v(\paramu),v\big(\paramu))\leq \|v(\paramu)\|_{H_{\rho}^1}^2$, we obtain
$\CK \leq \frac{\tau(\paramu)}{1 - \tau(\paramu)}$, from which we conclude the proposition.
\end{proof}
\subsection{Proof of Proposition 3.1}
\begin{proof}
Using the Riesz representation theorem, we can define the supremizer ${T}(s;\paramu):\U \rightarrow \U $ by
\begin{equation}
    \big({T}(s;\paramu)\ldu,\ldw  \big)_{\U(\paramu)} = B(\ldu,\ldw;s;\paramu),\quad \forall \ldu,\ldw\in \U.\label{eq:def supremizer cont.}
\end{equation}
Denote $v(\paramu)$ an eigenfunction to the eigenvalue problem \cref{eq:Stability bilinear proposition EVP coercive case cont.} with eigenvalue $\tau(\paramu)$. Using the definitions of $T$ and $B$ we obtain that 
\begin{equation*}
    \big({T}(s;\paramu)v(\paramu),\ldw  \big)_{\U(\paramu)}=B(v(\paramu),\ldw;s;\paramu) 
    =\big(s^2(1-\tau(\paramu))+\tau(\paramu)\big) (v(\paramu),\ldw)_{\U(\paramu)}
\end{equation*}
for all $\ldw\in \U$. Therefore, $v(\paramu)$ is an eigenfunction of ${T}(s;\paramu)$ with eigenvalue 
$\chi(s;{\paramu}) := s^2(1-\tau(\paramu) ) + \tau(\paramu)$. Since the eigenfunctions to \cref{eq:Stability bilinear proposition EVP coercive case cont.} are complete in $\U$, it remains to bound $|\chi(s;\paramu)|$ from below.
Setting $c_R={\rm Re}(s^2)=\reals^2-\imags  ^2$ and $c_I={\rm Im}(s^2)=2\reals \imags  $, we have that
\begin{equation}
    |\chi(s;{\paramu})|^2 = \big((c_R-1)^2+c_I^2\big)\tau(\paramu)^2 +2\big(c_R(1-c_R)-c_I^2 \big)\tau(\paramu) + \big(c_R^2 +c_I^2\big).
\end{equation}
with $\tau_0\leq \tau(\paramu)\leq 1$ by Proposition~\ref{prop:lower_lambda_0}.
If $(c_R-1)^2+c_I^2=0$, then $|\chi(s;{\paramu})|^2=1$.
Otherwise, $(c_R-1)^2+c_I^2>0$, and $|\chi(s;{\paramu})|^2$ is a quadratic polynomial in $\tau$,
which attains its minimum in $\tau_*(s)$ defined in \cref{eq:betaLB cases continous} 
if $\tau_0\leq\tau_*(s)\leq 1$, or in $\tau=\tau_0$ if $\tau_*(s)<\tau_0$, or in $\tau=1$ if $\tau_*(s)>1$.
Hence, we obtain \cref{eq:transposed infsup cont.} 
with $\beta(s)$ defined in \cref{eq:betaLB cases continous}.

To investigate the value of $\beta(s)$ for $s=\reals +\imagI \imags  \in I_s$, we note that replacing $\tau_0$ by $0$ in \cref{eq:betaLB cases continous}
gives a lower bound.
For $\tau_*(s)\geq 1$ we have $\beta(s)=1$, and for $\tau_*(s)\leq0$, we have $\beta(s)\geq |s^2|=|s|^2\geq \reals ^2$.
Consider $\tau_*(s)\in (0,1)$. We investigate the behavior of $\beta(s)$ as a function of $\imags  $.
Calculus shows that a necessary condition for a local minimum is $\imags  =0$.
For $\imags  =0$ we obtain that $\tau_*(s)=-\reals ^2/(1-\reals ^2)$, which is positive only if $\reals >1$, but then $\tau_*(s)>1$.
Hence, $\beta(s)$ becomes minimal for $s_*=\reals +\imagI s_{max}$. For $s_{max}$ large, we have the asymptotic behavior $\beta_m(s_*)=|s_*^2(1-\tau_*(s_*))+\tau_*(s_*)|\approx 1/(1+(s_{max}/(2\reals ))^2)^{1/2}>0$.
Summarizing, we may choose $\beta_*=\min\{1, \reals ^2,\beta_m(s_*)/2\}$.

The other inf-sup condition in \cref{eq:transposed infsup cont.}
follows from the observation that the adjoint of the supremizer satisfies $T(s;\paramu)^*=T(\bar s;\paramu)$ and $|\chi(\bar s;\paramu)|=|\chi(s;\paramu)|$.
\end{proof}
\section{Analyticity of Ricker wavelet}\label{supplementary:Rickeranalytic}
 Laplace transforms are analytic in the region of absolute convergence, and we have for some fixed $C>0$ and $0<\xi\leq \alpha$
\begin{equation}
    |\mathcal{q}(t)|=|(1-\alpha^2(t-t_0)^2/2)|e^{-\alpha^2(t-t_0)^2/4}\leq Ce^{-\xi^2(t-t_0)^2/4}.
\end{equation}
Then, it is straightforward to check that 
\begin{equation}
    \int_0^{\infty} |\mathcal{q}(t)e^{-(\reals +\imagI \imags  )t}|dt \leq C \int_0^{\infty} e^{-(\xi^2(t-t_0)^2/4 + \reals  t)}dt <\infty,
\end{equation}
converges absolutely for all $s\in \CC$, and the Laplace transform of $\mathcal{q}(t)$ is analytic for all $s\in \CC$. 
\section{Supporting numerical result}
In the following section, we present numerical results obtained using the model problem \cref{sec:MODEL}. 
\subsection{Discrete inf-sup constant and its lower bound} \label{supplementary:infsup_plots}
In \cref{fig:discreteinfsup}, we observe that the derived lower-bound for the inf-sup constant in \cref{prop:infsup cont.} is sharp and very simple to evaluate. Here, we use the model problem from \cref{sec:Model problem numerical results} with $\paramu\in\Pspace$ fixed at values from \cref{fig:model}. The discrete counter part of the inf-sup constant in \cref{eq:transposed infsup cont.}, denoted as $\beta_h(s)$ is computed by the solution of an eigenvalue problem (cf., \cite[Sec. 2.4.6]{quarteroni2015reduced} for details). We compute $\beta(s)$ from \cref{eq:betaLB cases continous} using $\tau_0$ as in \cref{eq:Lambda0 constant}, where $\CK$ is obtained by solving the eigenvalue problem in \cref{eq:Stability bilinear proposition EVP coercive case cont.}.
\begin{figure}[h!]
\centering
\subfloat[$\beta_h(s)$]{\label{subfig:betah}\includegraphics[width=.33\linewidth,trim={00cm 0cm 0.0cm 0cm},clip]{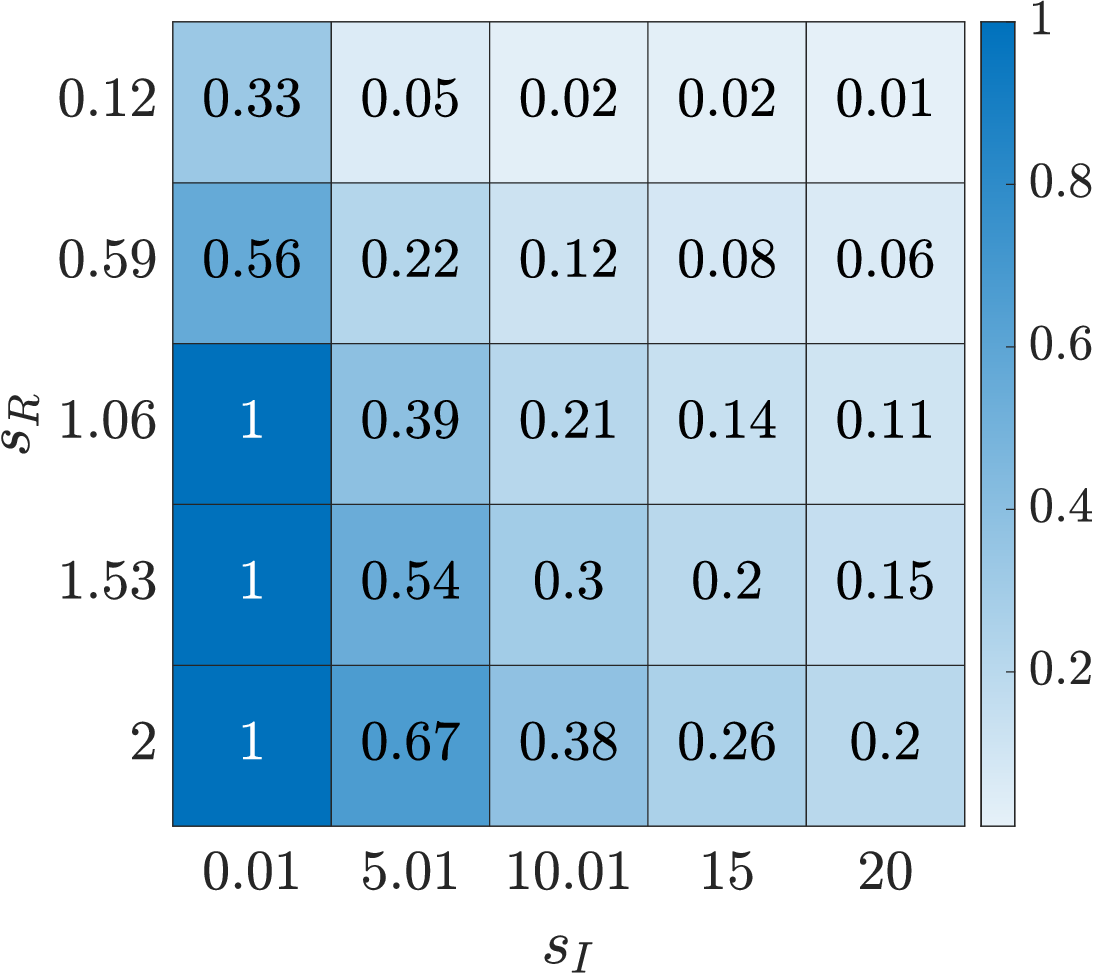}}
\subfloat[$\beta(s)$]{\label{subfig:betas}\includegraphics[width=.33\linewidth,trim={00cm 0cm 0.0cm 0cm},clip]{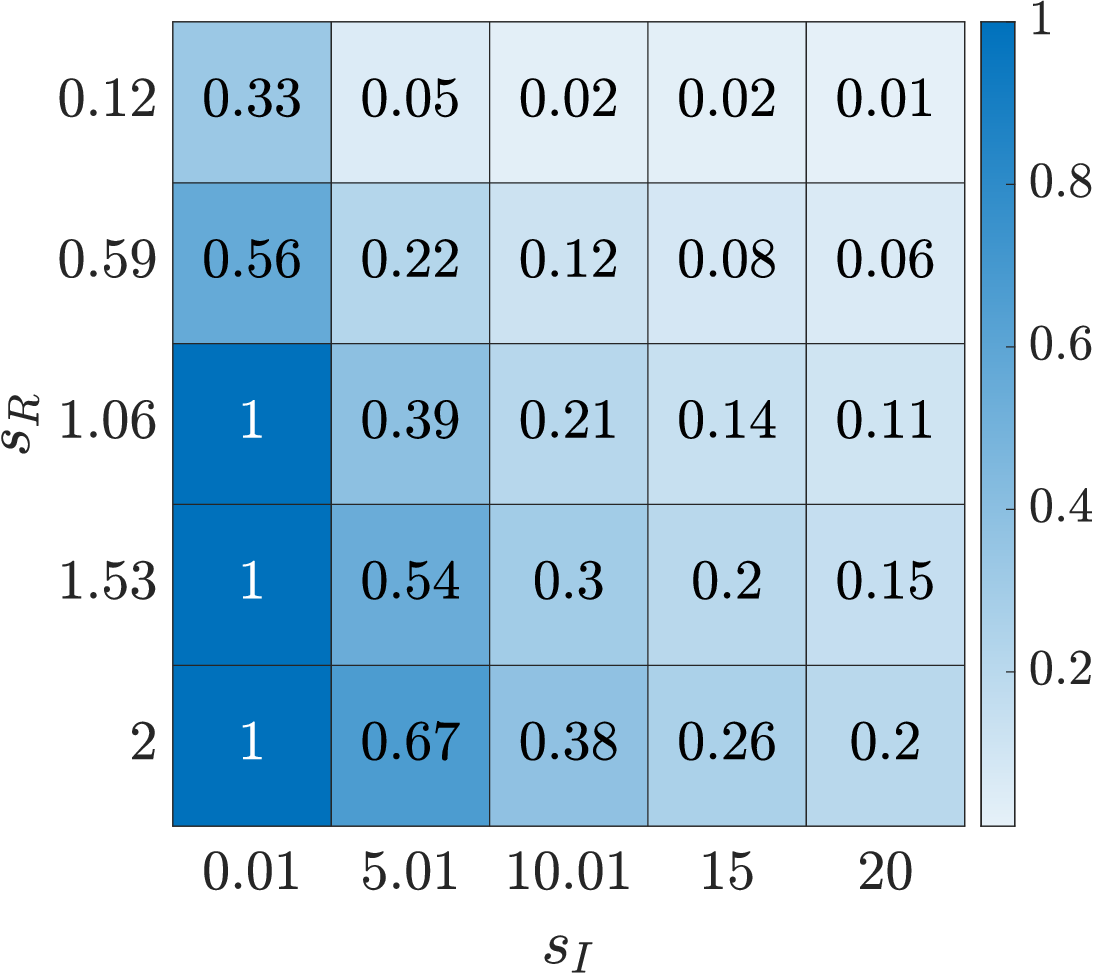}}
\subfloat[$\beta_*(s)$]{\label{subfig:betaunderline}\includegraphics[width=.33\linewidth,trim={00cm 0cm 0.0cm 0cm},clip]{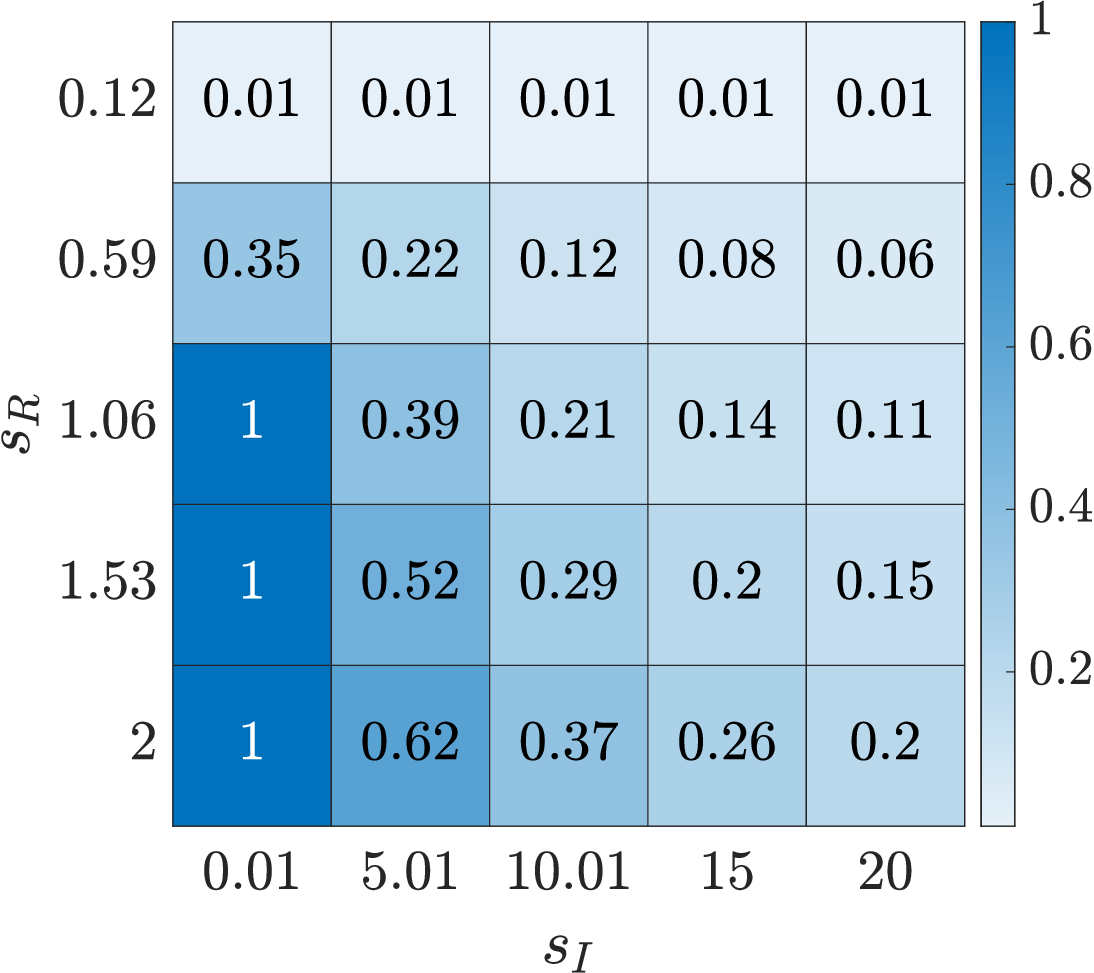}}
\caption{Illustration of discrete counter part of the inf-sup constant in \cref{eq:transposed infsup cont.} shown in (a), the lower-bound $\beta(s)$ from \cref{eq:betaLB cases continous} in (b) and $\beta_*(s)=\min\{1, \reals^2,2\reals/\sqrt{4\reals^2+\imags^2}\}$ shown in (c).}
\label{fig:discreteinfsup}
\end{figure}

\subsection{Supporting result: \cref{remark: POD vs SPOD}} \label{supplementary:remarkSPODPOD}
In \cref{fig:SPOD}, we show that a complex-valued structure of the POD basis functions is more effective to capture the interaction between the real and imaginary parts of the solutions than using real-valued POD basis functions. in contrast to having a real-valued structure of the POD basis functions as employed in \cite{Bigoni2020SimulationbasedMonitoring}. The parameters used are the same as in \cref{sec:reduction in the s parameter results}.

\begin{figure}
   \centering
   \subfloat[Offline phase]{\label{subfig:singular_valuedecay_SPOD}\includegraphics[width=.5\linewidth,trim={0.6cm 0cm 21cm 0cm},clip]{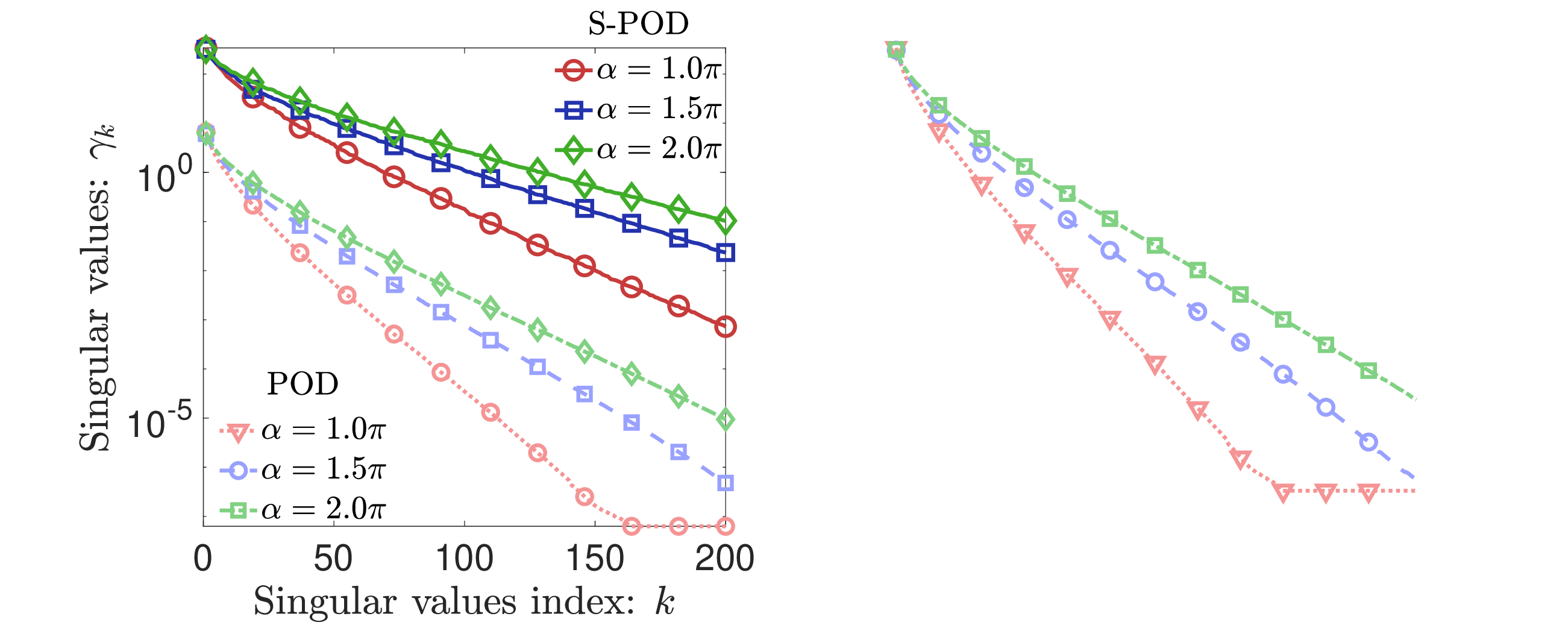}}
   \subfloat[Online phase]{\label{subfig:Mean_error_SPODTD}\includegraphics[width=.5\linewidth,trim={0.6cm 0cm 21cm 0cm},clip]{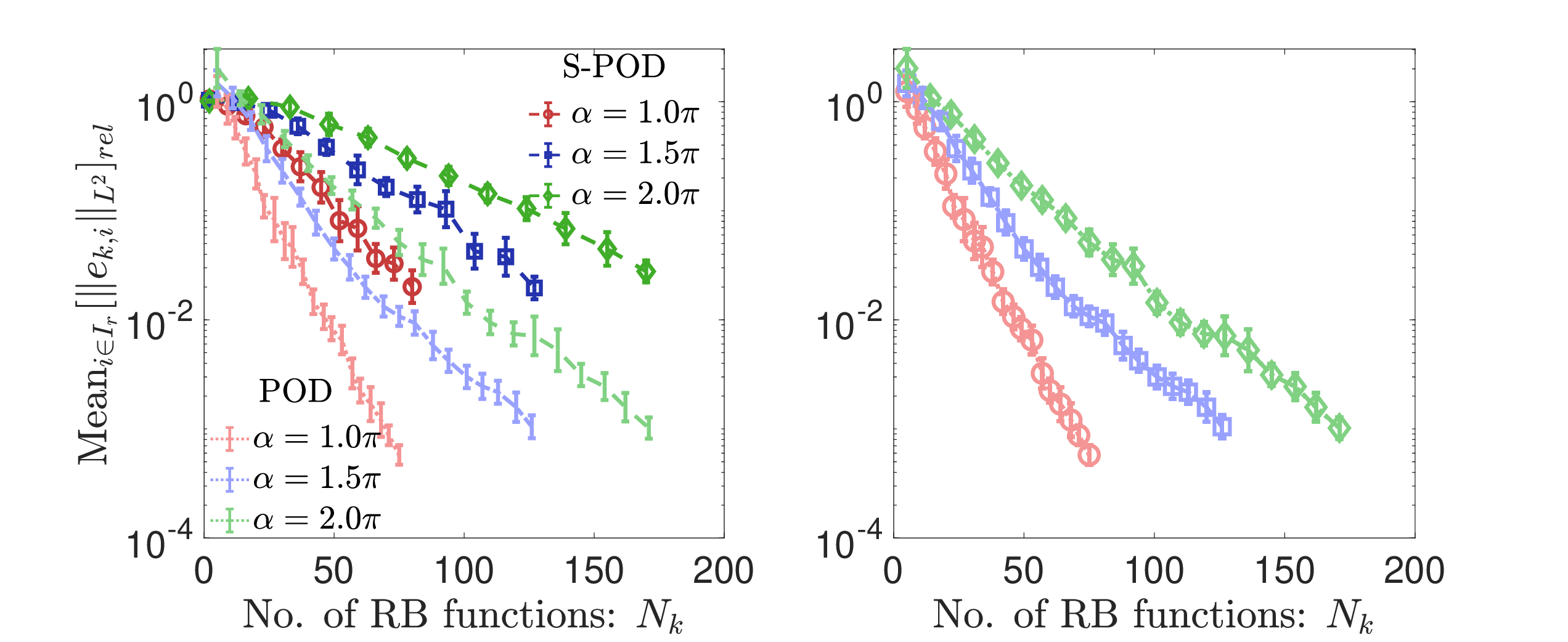}}
   \caption{Comparison of the complex-valued structure of reduced basis matrix with the real-valued structure of reduced basis from \cite{Bigoni2020SimulationbasedMonitoring}. In (a), we plot the decay in the singular values. In (b), we illustrate the mean error with the length indicating standard deviation, for the RB approximation of the time domain seismograms. %\cref{subfig:Meanerror_TD}. 
   }
   \label{fig:SPOD}
 \end{figure}

\bibliographystyle{siamplain}